\documentclass[11pt]{article}
\usepackage{epigamath}

\usepackage[notext]{kpfonts}
\usepackage{baskervald}

\setpapertype{A4}

\usepackage[english]{babel}

\usepackage[dvips]{graphicx}     

\removelength{2cm}

\title{Weighted projective lines and rational surface singularities}
\titlemark{Weighted projective lines and rational surface singularities}
\author{Osamu Iyama and Michael Wemyss}
\authoraddresses{
\authordata{Osamu Iyama}{\firstname{Osamu} \lastname{Iyama}\\
\institution{Graduate School of Mathematics, Nagoya University, Chikusa-ku, Nagoya, 464-8602, Japan}\\
\email{iyama@math.nagoya-u.ac.jp}} \\
\authordata{Michael Wemyss}{\firstname{Michael} \lastname{Wemyss}\\
\institution{School of Mathematics and Statistics, University of Glasgow, University Place, Glasgow, G12 8QQ, UK}\\
\email{michael.wemyss@glasgow.ac.uk}}
}
\authormark{O. Iyama and M. Wemyss}
\date{\vspace{-5ex}} 
\journal{\'Epijournal de G\'eom\'etrie Alg\'ebrique} 
\acceptation{Received by the Editors on August 18, 2018. 
Accepted on August 7, 2019.}

\newcommand{\articlereference}{Volume 3 (2019), Article Nr. 13}

\acknowledgement{Osamu Iyama was supported by JSPS Grant-in-Aid for Scientific Research (B)24340004, (B)16H03923, (C)23540045 and (S)15H05738, and Michael Wemyss was supported by EPSRC grant~EP/K021400/2. Both were partially supported by the American Institute of Mathematics.}

\usepackage[all]{xy}

\allowdisplaybreaks

\newdimen\origiwspc
\origiwspc=\fontdimen2\font

\usepackage{float}

\numberwithin{equation}{section}

\makeatletter

\renewcommand{\p@equation}{\arabic{section}.\Alph{equation}\expandafter\@gobble}
\makeatother

\setlist[enumerate,1]{label={\rm (\arabic*)}}

\makeatletter

\renewcommand{\p@enumi}{\arabic{enumi}\expandafter\@gobble}
\makeatother

\usepackage{array,multirow,booktabs}
\newcommand{\colonequals}{:=}
\newcommand{\qedhere}{}

\usepackage[colorlinks]{hyperref}
\usepackage{tikz,mathrsfs}
\usetikzlibrary{arrows,decorations.pathmorphing,decorations.pathreplacing,positioning,shapes.geometric,shapes.misc,decorations.markings,decorations.fractals,calc,patterns}

\tikzset{>=stealth',
       cvertex/.style={circle,draw=black,inner sep=1pt,outer sep=3pt},
       vertex/.style={circle,fill=black,inner sep=1pt,outer sep=3pt},
       star/.style={circle,fill=yellow,inner sep=0.75pt,outer sep=0.75pt},
       tvertex/.style={inner sep=1pt,font=\scriptsize},
       gap/.style={inner sep=0.5pt,fill=white}}

\tikzstyle{mybox} = [draw=black, fill=blue!10, very thick,
   rectangle, rounded corners, inner sep=10pt, inner ysep=20pt]
\tikzstyle{boxtitle} =[fill=blue!50, text=white,rectangle,rounded corners]

\usepackage{upgreek}

\newtheorem{thm}{Theorem}[section]
\newtheorem{prop}[thm]{Proposition}
\newtheorem{lemma}[thm]{Lemma}
\newtheorem{defin}[thm]{Definition}
\newtheorem{cor}[thm]{Corollary}

\newtheorem{example}[thm]{Example} 
\newtheorem{notation}[thm]{Notation}

\newtheorem{remark}[thm]{Remark}

\makeatletter
\renewcommand\thesubsection{\arabic{section}.\arabic{subsection}.}
\renewcommand{\p@subsection}{\arabic{section}.\arabic{subsection}\expandafter\@gobble}
\makeatother

\renewcommand{\t}[1]{\textnormal{#1}}

\newcommand{\looptop}[2]{\xy \SelectTips{cm}{10}
\POS(0,0) \endxy}

\newcommand{\ve}{\upvarepsilon}
\def\top{\mathop{\rm top}\nolimits}

\def\op{\mathop{\rm op}\nolimits}
\def\GL{\mathop{\rm GL}\nolimits}

\def\CM{\mathop{\sf CM}\nolimits}

\def\SCM{\mathop{\sf SCM}\nolimits}

\def\mod{\mathop{\sf mod}\nolimits}
\def\gr{\mathop{\sf gr}\nolimits}

\def\qgr{\mathop{\sf qgr}\nolimits}

\def\Qcoh{\mathop{\sf Qcoh}\nolimits}
\def\coh{\mathop{\sf coh}\nolimits}
\def\Mod{\mathop{\sf Mod}\nolimits}

\def\Hom{\mathop{\rm Hom}\nolimits}
\def\RHom{\mathop{\rm {\bf R}Hom}\nolimits}
\def\End{\mathop{\rm End}\nolimits}
\def\Ext{\mathop{\rm Ext}\nolimits}

\def\add{\mathop{\sf add}\nolimits}

\def\Ker{\mathop{\rm Ker}\nolimits}

\def\Im{\mathop{\rm Im}\nolimits}
\def\Sing{\mathop{\rm Sing}\nolimits}

\def\Spec{\mathop{\rm Spec}\nolimits}
\def\Max{\mathop{\rm Max}\nolimits}
\def\gl{\mathop{\rm gl.dim}\nolimits}

\def\D{\mathop{\rm{D}^{}}\nolimits}

\def\Db{\mathop{\rm{D}^b}\nolimits}

\def\vect{\mathop{{\sf vect}}\nolimits}
\def\inc{\mathop{\rm{inc}}\nolimits}

\newcommand{\p}{\mathfrak{p}}
\newcommand{\m}{\mathfrak{m}}

\def\bl{\mathop{\boldsymbol{\lambda}}\nolimits}
\def\bm{\mathop{\boldsymbol{\mu}}\nolimits}
\def\bp{\mathop{\mathbf{{}_{}p}}\nolimits}
\def\bq{\mathop{\mathbf{{}_{}q}}\nolimits}

\newcommand{\KK}{ \Bbbk}

\newcommand{\bL}{\mathbb{L}}
\newcommand{\bN}{\mathbb{N}}
\newcommand{\bP}{\mathbb{P}}
\newcommand{\bX}{\mathbb{X}}
\newcommand{\bZ}{\mathbb{Z}}
\newcommand{\bT}{\mathbb{T}}
\newcommand{\bTot}{\mathbb{T}\mathrm{ot}}

\newcommand{\oc}{{\vec{c}}}
\newcommand{\ox}{{\vec{x}}}
\newcommand{\os}{{\vec{s}}}

\newcommand{\ow}{{\vec{\omega}}}
\newcommand{\oy}{{\vec{y}}}
\newcommand{\oz}{{\vec{z}}}

\newcommand{\osc}{{\vec{\mathsf{c}}}}
\newcommand{\osx}{{\vec{\mathsf{x}}}}
\newcommand{\osym}{{\vec{\mathsf{y}}_m}}
\newcommand{\osyl}{{\vec{\mathsf{y}}_\ell}}

\renewcommand{\sp}{{\mathsf{p}}}
\newcommand{\st}{{\mathsf{t}}}
\newcommand{\sa}{{\mathsf{a}}}
\newcommand{\sx}{{\mathsf{x}}}

\newcommand{\cC}{\mathcal{C}}

\newcommand{\cE}{\mathcal{E}}
\newcommand{\cF}{\mathcal{F}}

\newcommand{\cK}{\mathcal{K}}
\newcommand{\cL}{\mathcal{L}}
\newcommand{\cM}{\mathcal{M}}

\newcommand{\cO}{\mathcal{O}}

\newcommand{\cS}{\mathcal{S}}

\newcommand{\cV}{\mathcal{V}}

\newcommand{\cZ}{\mathcal{Z}}

\def\bl{\mathop{\boldsymbol{\lambda}}\nolimits}
\def\bp{\mathop{\mathbf{{}_{}p}}\nolimits}

\def\Rq{\mathop{{\rm\bf R}q}\nolimits}

\def\RHom{\mathop{{\rm{\bf R}Hom}}\nolimits}

\def\sHom{\mathop{\mathcal{H}om}\nolimits}
\newcommand\RDerived[1]{\mathop{{\rm\bf R}{#1}}\nolimits}

\def\RHom{\mathop{{\rm{\bf R}Hom}}\nolimits}

\newcommand\Ltimes{\overset{\rm\bf L}{\otimes}}

\newcommand\Pos{\mathsf{GPos}}

\begin{document}


\maketitle



\begin{prelims}

\vspace{-0.55cm}

\def\abstractname{Abstract}
\abstract{In this paper we study rational surface singularities $R$ with star shaped dual graphs, and under very mild assumptions on the self-intersection numbers we give an explicit description of all their special Cohen--Macaulay modules.  We do this by realising $R$ as a certain $\mathbb{Z}$-graded Veronese subring $S^{\ox}$ of the homogeneous coordinate ring $S$ of the Geigle--Lenzing weighted projective line $\mathbb{X}$, and we realise the special CM modules as explicitly described summands of the canonical tilting bundle on $\mathbb{X}$.  We then give a second proof that these are special CM modules by comparing $\qgr S^{\ox}$ and $\coh\bX$, and we also give a necessary and sufficient combinatorial criterion for these to be equivalent categories.   In turn, we show that $\qgr S^{\ox}$ is equivalent to $\qgr\Upgamma$ where $\Upgamma$ is the corresponding reconstruction algebra, and that the degree zero piece of $\Upgamma$ coincides with Ringel's canonical algebra.   This implies that $\Upgamma$ contains the canonical algebra and furthermore $\qgr\Upgamma$ is derived equivalent to the canonical algebra, thus linking the reconstruction algebra of rational surface singularities to the canonical algebra of representation theory.}

\vspace{-0.1cm}


\vspace{-0.1cm}



\languagesection{Fran\c{c}ais}{%

\vspace{-0.2cm}
{\bf Titre. Droites projectives \`a poids et singularit\'es rationnelles de surfaces} \\[1mm]
{\bf R\'esum\'e.} Dans cet article, nous \'etudions les singularit\'es rationnelles de surfaces $R$ de graphes duaux \'etoil\'es. Sous de l\'eg\`eres hypoth\`eses sur les nombres d'auto-intersection, nous donnons une description explicite de tous leurs modules Cohen--Macaulay sp\'eciaux. Pour cela, nous r\'ealisons $R$ comme un certain sous-anneau de Veronese $\mathbb{Z}$-gradu\'e $S^{\ox}$ de l'anneau de coordonn\'ees homog\`enes $S$ de la droite projective pond\'er\'ee de Geigle--Lenzing $\mathbb{X}$ et nous r\'ealisons les modules CM sp\'eciaux comme des facteurs directs explicites du fibr\'e basculant canonique de $\mathbb{X}$. Nous donnons ensuite une seconde d\'emonstration du fait que ce sont des modules CM sp\'eciaux en comparant $\qgr S^{\ox}$ et $\coh\bX$, et nous \'enon\c{c}ons \'egalement un crit\`ere combinatoire n\'ecessaire et suffisant d'\'equivalence pour ces deux cat\'egories. En outre nous montrons que $\qgr S^{\ox}$ est \'equivalente \`a $\qgr\Upgamma$ o\`u $\Upgamma$ est l'alg\`ebre de reconstruction correspondante, et que la partie de degr\'e z\'ero de $\Upgamma$ co\"{\i}ncide avec l'alg\`ebre canonique de Ringel. Cela implique que $\Upgamma$ contient l'alg\`ebre canonique et de plus que $\qgr\Upgamma$ est \'equivalente au sens d\'eriv\'e \`a l'alg\`ebre canonique, reliant ainsi l'alg\`ebre de reconstruction des singularit\'es rationnelles de surfaces \`a la th\'eorie des repr\'esentations de l'alg\`ebre canonique.}

\end{prelims}


\newpage

\setcounter{tocdepth}{1} \tableofcontents

\section{Introduction}

\subsection{Motivation and Overview}\label{motivation section}
It is well known that any rational surface singularity has only finitely many indecomposable special CM modules, but it is in general a difficult task to classify and describe them explicitly.
In this paper we use the combinatorial structure encoded in the homogeneous coordinate ring $S$ of the Geigle--Lenzing weighted projective line $\bX$ to solve this problem for a large class of examples arising from star shaped dual graphs, extending our previous work \cite{IW} to cover a much larger class of varieties.  In the process, we link $S$, its Veronese subrings, the reconstruction algebra and the canonical algebra, through a range of categorical equivalences.

A hint of a connection between rational surface singularities and the canonical algebra can be found in the lecture notes \cite{Ringel}. In his study of the canonical algebra $\Lambda_{\bp,\bl}$, Ringel drew pictures \cite[p196]{Ringel} of canonical tilting bundles on $\bX$ for the cases $\bp=(2,3,3)$, $(2,3,4)$ and $(2,3,5)$, which correspond to Dynkin diagrams $E_6$, $E_7$ and $E_8$.  For example, in the $E_7$ case, Ringel's picture is the following.
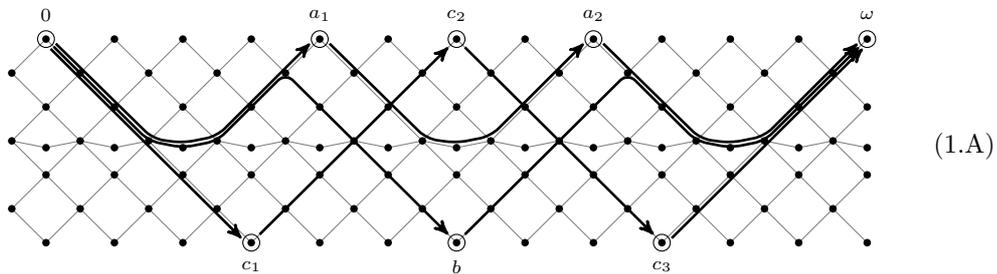
\begin{equation}
\begin{array}{c}
\begin{tikzpicture}[xscale=0.9,yscale=0.9]
\draw[gray] (0,-2) -- (0.5,-2.5);
\draw[gray] (0,-1) -- (1.5,-2.5);
\draw[gray] (0,0) -- (2.5,-2.5);
\draw[gray] (0.5,0.5) -- (3.5,-2.5);
\draw[gray] (1.5,0.5) -- (4.5,-2.5);
\draw[gray] (2.5,0.5) -- (5.5,-2.5);
\draw[gray] (3.5,0.5) -- (6.5,-2.5);
\draw[gray] (4.5,0.5) -- (7.5,-2.5);
\draw[gray] (5.5,0.5) -- (8.5,-2.5);
\draw[gray] (6.5,0.5) -- (9.5,-2.5);
\draw[gray] (7.5,0.5) -- (10.5,-2.5);
\draw[gray] (8.5,0.5) -- (11.5,-2.5);
\draw[gray] (9.5,0.5) -- (12.5,-2.5);
\draw[gray] (10.5,0.5) -- (12.5,-1.5);
\draw[gray] (11.5,0.5) -- (12.5,-0.5);
\draw[gray] (0,0) -- (0.5,0.5);
\draw[gray] (0,-1) -- (1.5,0.5);
\draw[gray] (0,-2) -- (2.5,0.5);
\draw[gray] (0.5,-2.5) -- (3.5,0.5);
\draw[gray] (1.5,-2.5) -- (4.5,0.5);
\draw[gray] (2.5,-2.5) -- (5.5,0.5);
\draw[gray] (3.5,-2.5) -- (6.5,0.5);
\draw[gray] (4.5,-2.5) -- (7.5,0.5);
\draw[gray] (5.5,-2.5) -- (8.5,0.5);
\draw[gray] (6.5,-2.5) -- (9.5,0.5);
\draw[gray] (7.5,-2.5) -- (10.5,0.5);
\draw[gray] (8.5,-2.5) -- (11.5,0.5);
\draw[gray] (9.5,-2.5) -- (12.5,0.5);
\draw[gray] (10.5,-2.5) -- (12.5,-0.5);
\draw[gray] (11.5,-2.5) -- (12.5,-1.5);
\draw[gray] (0,-1) -- (0.5,-1.1) -- (1,-1) -- (1.5,-1.1) -- (2,-1) -- (2.5,-1.1) -- (3,-1) -- (3.5,-1.1) -- (4,-1) -- (4.5,-1.1) -- (5,-1) -- (5.5,-1.1) -- (6,-1) --(6.5,-1.1) -- (7,-1) -- (7.5,-1.1) -- (8,-1) -- (8.5,-1.1) -- (9,-1) -- (9.5,-1.1) --(10,-1) -- (10.5,-1.1) -- (11,-1) -- (11.5,-1.1) -- (12,-1) --  (12.5,-1.1);
\node (R0) at (0.5,0.5) [vertex] {};
\node (R1) at (1.5,0.5) [vertex] {};
\node (R2) at (2.5,0.5) [vertex] {};
\node (R3) at (3.5,0.5) [vertex] {};
\node (R4) at (4.5,0.5) [vertex] {};
\node (R5) at (5.5,0.5) [vertex] {};
\node (R6) at (6.5,0.5) [vertex] {};
\node (R7) at (7.5,0.5) [vertex] {};
\node (R8) at (8.5,0.5) [vertex] {};
\node (R9) at (9.5,0.5) [vertex] {};
\node (R10) at (10.5,0.5) [vertex] {};
\node (R11) at (11.5,0.5) [vertex] {};
\node (R12) at (12.5,0.5) [vertex] {};
\node (A0) at (0,0) [vertex] {};
\node (A1) at (1,0) [vertex] {};
\node (A2) at (2,0) [vertex] {};
\node (A3) at (3,0) [vertex] {};
\node (A4) at (4,0) [vertex] {};
\node (A5) at (5,0) [vertex] {};
\node (A6) at (6,0) [vertex] {};
\node (A7) at (7,0) [vertex] {};
\node (A8) at (8,0) [vertex] {};
\node (A9) at (9,0) [vertex] {};
\node (A10) at (10,0) [vertex] {};
\node (A11) at (11,0) [vertex] {};
\node (A12) at (12,0) [vertex] {};
\node (B0) at (0.5,-0.5) [vertex] {};
\node (B1) at (1.5,-0.5) [vertex] {};
\node (B2) at (2.5,-0.5) [vertex] {};
\node (B3) at (3.5,-0.5) [vertex] {};
\node (B4) at (4.5,-0.5) [vertex] {};
\node (B5) at (5.5,-0.5) [vertex] {};
\node (B6) at (6.5,-0.5) [vertex] {};
\node (B7) at (7.5,-0.5) [vertex] {};
\node (B8) at (8.5,-0.5) [vertex] {};
\node (B9) at (9.5,-0.5) [vertex] {};
\node (B10) at (10.5,-0.5) [vertex] {};
\node (B11) at (11.5,-0.5) [vertex] {};
\node (B12) at (12.5,-0.5) [vertex] {};
\node (C0) at (0,-1)  [vertex] {};
\node (C1) at (0.5,-1.1)  [vertex] {};
\node (C2) at (1,-1)  [vertex] {};
\node (C3) at (1.5,-1.1)  [vertex] {};
\node (C4) at (2,-1)  [vertex] {};
\node (C5) at (2.5,-1.1)  [vertex] {};
\node (C6) at (3,-1)  [vertex] {};
\node (C7) at (3.5,-1.1)  [vertex] {};
\node (C8) at (4,-1)  [vertex] {};
\node (C9) at (4.5,-1.1)  [vertex] {};
\node (C10) at (5,-1)  [vertex] {};
\node (C11) at (5.5,-1.1)  [vertex] {};
\node (C12) at (6,-1)  [vertex] {};
\node (C13) at (6.5,-1.1)  [vertex] {};
\node (C14) at (7,-1)  [vertex] {};
\node (C15) at (7.5,-1.1)  [vertex] {};
\node (C16) at (8,-1)  [vertex] {};
\node (C17) at (8.5,-1.1)  [vertex] {};
\node (C18) at (9,-1)  [vertex] {};
\node (C19) at (9.5,-1.1)  [vertex] {};
\node (C20) at (10,-1)  [vertex] {};
\node (C21) at (10.5,-1.1)  [vertex] {};
\node (C22) at (11,-1)  [vertex] {};
\node (C23) at (11.5,-1.1)  [vertex] {};
\node (C24) at (12,-1)  [vertex] {};
\node (C25) at (12.5,-1.1)  [vertex] {};
\node (D0) at (0.5,-1.5)  [vertex] {};
\node (D1) at (1.5,-1.5)  [vertex] {};
\node (D2) at (2.5,-1.5)  [vertex] {};
\node (D3) at (3.5,-1.5)  [vertex] {};
\node (D4) at (4.5,-1.5)  [vertex] {};
\node (D5) at (5.5,-1.5)  [vertex] {};
\node (D6) at (6.5,-1.5)  [vertex] {};
\node (D7) at (7.5,-1.5)  [vertex] {};
\node (D8) at (8.5,-1.5)  [vertex] {};
\node (D9) at (9.5,-1.5)  [vertex] {};
\node (D10) at (10.5,-1.5)  [vertex] {};
\node (D11) at (11.5,-1.5)  [vertex] {};
\node (D12) at (12.5,-1.5)  [vertex] {};
\node (E0) at (0,-2)  [vertex] {};
\node (E1) at (1,-2)  [vertex] {};
\node (E2) at (2,-2)  [vertex] {};
\node (E3) at (3,-2)  [vertex] {};
\node (E4) at (4,-2)  [vertex] {};
\node (E5) at (5,-2) [vertex] {};
\node (E6) at (6,-2)  [vertex] {};
\node (E7) at (7,-2)  [vertex] {};
\node (E8) at (8,-2)  [vertex] {};
\node (E9) at (9,-2) [vertex] {};
\node (E10) at (10,-2)  [vertex] {};
\node (E11) at (11,-2)  [vertex] {};
\node (E12) at (12,-2)  [vertex] {};
\node (F0) at (0.5,-2.5)  [vertex] {};
\node (F1) at (1.5,-2.5)  [vertex] {};
\node (F2) at (2.5,-2.5)  [vertex] {};
\node (F3) at (3.5,-2.5)  [vertex] {};
\node (F4) at (4.5,-2.5)  [vertex] {};
\node (F5) at (5.5,-2.5)  [vertex] {};
\node (F6) at (6.5,-2.5)  [vertex] {};
\node (F7) at (7.5,-2.5)  [vertex] {};
\node (F8) at (8.5,-2.5)  [vertex] {};
\node (F9) at (9.5,-2.5)  [vertex] {};
\node (F10) at (10.5,-2.5)  [vertex] {};
\node (F11) at (11.5,-2.5)  [vertex] {};
\node (F12) at (12.5,-2.5)  [vertex] {};
\draw (0.5,0.5) circle (3.5pt);
\node at (0.5,0.85) {$\scriptstyle 0$};
\draw (4.5,0.5) circle (3.5pt);
\node at (4.5,0.85) {$\scriptstyle a_1$};
\draw (6.5,0.5) circle (3.5pt);
\node at (6.5,0.85) {$\scriptstyle c_2$};
\draw (8.5,0.5) circle (3.5pt);
\node at (8.5,0.85) {$\scriptstyle a_2$};
\draw (12.5,0.5) circle (3.5pt);
\node at (12.5,0.85) {$\scriptstyle \omega$};
\draw (3.5,-2.5) circle (3.5pt);
\node at (3.5,-2.85) {$\scriptstyle c_1$};
\draw (6.5,-2.5) circle (3.5pt);
\node at (6.5,-2.85) {$\scriptstyle b$};
\draw (9.5,-2.5) circle (3.5pt);
\node at (9.5,-2.85) {$\scriptstyle c_3$};
\draw[->,line width=1pt] ($(R0) + (-65:4.5pt)$) -- ($(F3) + (155:4.5pt)$);
\draw[->,line width=1pt] (F3) -- (R6);
\draw[->,line width=1pt] (R6) -- (F9);
\draw[->,line width=1pt] ($(F9) + (25:4.5pt)$) -- ($(R12) + (-115:4.5pt)$);
\draw[->,line width=1pt,rounded corners] ($(R0) + (-45:4.5pt)$) -- ($(C4) + (-135:0pt)$) -- ($(C5)$) -- ($(C6)$) -- ($(A4)$) -- (F6);
\draw[->,line width=1pt,rounded corners] ($(R0) + (-25:4.5pt)$) -- ($(C4) + (25:2pt)$) -- ($(C5)+ (90:2pt)$) -- ($(C6)+ (90:2pt)$) -- ($(R4)+ (-155:4.5pt)$);
\draw[->,line width=1pt,rounded corners] ($(R4) + (-25:4.5pt)$) -- ($(C12) + (25:2pt)$) -- ($(C13)+ (90:2pt)$) -- ($(C14)+ (90:2pt)$) -- ($(R8)+ (-155:4.5pt)$);
\draw[->,line width=1pt,rounded corners] ($(R8) + (-25:4.5pt)$) -- ($(C20) + (25:2pt)$) -- ($(C21)+ (90:2pt)$) -- ($(C22)+ (90:2pt)$) -- ($(R12)+ (-155:4.5pt)$);
\draw[->,line width=1pt,rounded corners] (F6) -- ($(A9)$) --  ($(C20) + (-135:0pt)$) -- ($(C21)$) -- ($(C22)$) -- (R12);
\end{tikzpicture}
\end{array}\label{RingelPicture}
\end{equation}
What is remarkable is that all of Ringel's pictures are identical to ones the authors drew in \cite[\S7--9]{IW} when classifying special CM modules for certain families of quotient singularities $\KK[[x,y]]^G$ with $G\leq \GL(2,\KK)$. For example, \cite[8.2]{IW} contains the following picture, indicating the positions of special CM modules in the AR quiver of $\KK[[x,y]]^{\mathbb{O}_{13}}$.
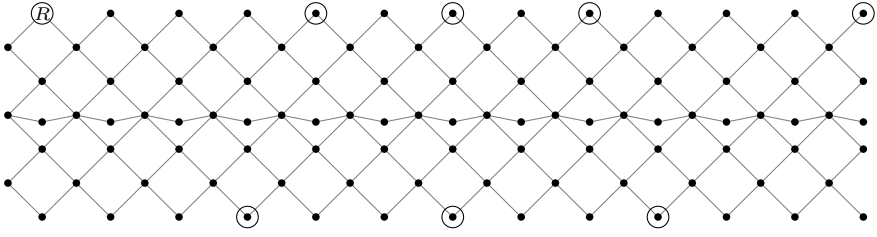
\[
\begin{array}{c}
\begin{tikzpicture}[xscale=0.9,yscale=0.9]
\draw[gray] (0,-2) -- (0.5,-2.5);
\draw[gray] (0,-1) -- (1.5,-2.5);
\draw[gray] (0,0) -- (2.5,-2.5);
\draw[gray] (0.5,0.5) -- (3.5,-2.5);
\draw[gray] (1.5,0.5) -- (4.5,-2.5);
\draw[gray] (2.5,0.5) -- (5.5,-2.5);
\draw[gray] (3.5,0.5) -- (6.5,-2.5);
\draw[gray] (4.5,0.5) -- (7.5,-2.5);
\draw[gray] (5.5,0.5) -- (8.5,-2.5);
\draw[gray] (6.5,0.5) -- (9.5,-2.5);
\draw[gray] (7.5,0.5) -- (10.5,-2.5);
\draw[gray] (8.5,0.5) -- (11.5,-2.5);
\draw[gray] (9.5,0.5) -- (12.5,-2.5);
\draw[gray] (10.5,0.5) -- (12.5,-1.5);
\draw[gray] (11.5,0.5) -- (12.5,-0.5);
\draw[gray] (0,0) -- (0.5,0.5);
\draw[gray] (0,-1) -- (1.5,0.5);
\draw[gray] (0,-2) -- (2.5,0.5);
\draw[gray] (0.5,-2.5) -- (3.5,0.5);
\draw[gray] (1.5,-2.5) -- (4.5,0.5);
\draw[gray] (2.5,-2.5) -- (5.5,0.5);
\draw[gray] (3.5,-2.5) -- (6.5,0.5);
\draw[gray] (4.5,-2.5) -- (7.5,0.5);
\draw[gray] (5.5,-2.5) -- (8.5,0.5);
\draw[gray] (6.5,-2.5) -- (9.5,0.5);
\draw[gray] (7.5,-2.5) -- (10.5,0.5);
\draw[gray] (8.5,-2.5) -- (11.5,0.5);
\draw[gray] (9.5,-2.5) -- (12.5,0.5);
\draw[gray] (10.5,-2.5) -- (12.5,-0.5);
\draw[gray] (11.5,-2.5) -- (12.5,-1.5);
\draw[gray] (0,-1) -- (0.5,-1.1) -- (1,-1) -- (1.5,-1.1) -- (2,-1) -- (2.5,-1.1) -- (3,-1) -- (3.5,-1.1) -- (4,-1) -- (4.5,-1.1) -- (5,-1) -- (5.5,-1.1) -- (6,-1) --(6.5,-1.1) -- (7,-1) -- (7.5,-1.1) -- (8,-1) -- (8.5,-1.1) -- (9,-1) -- (9.5,-1.1) --(10,-1) -- (10.5,-1.1) -- (11,-1) -- (11.5,-1.1) -- (12,-1) --  (12.5,-1.1);
\node (R0) at (0.5,0.5) [gap] {$\scriptstyle R$};
\node (R1) at (1.5,0.5) [vertex] {};
\node (R2) at (2.5,0.5) [vertex] {};
\node (R3) at (3.5,0.5) [vertex] {};
\node (R4) at (4.5,0.5) [vertex] {};
\node (R5) at (5.5,0.5) [vertex] {};
\node (R6) at (6.5,0.5) [vertex] {};
\node (R7) at (7.5,0.5) [vertex] {};
\node (R8) at (8.5,0.5) [vertex] {};
\node (R9) at (9.5,0.5) [vertex] {};
\node (R10) at (10.5,0.5) [vertex] {};
\node (R11) at (11.5,0.5) [vertex] {};
\node (R12) at (12.5,0.5) [vertex] {};
\node (A0) at (0,0) [vertex] {};
\node (A1) at (1,0) [vertex] {};
\node (A2) at (2,0) [vertex] {};
\node (A3) at (3,0) [vertex] {};
\node (A4) at (4,0) [vertex] {};
\node (A5) at (5,0) [vertex] {};
\node (A6) at (6,0) [vertex] {};
\node (A7) at (7,0) [vertex] {};
\node (A8) at (8,0) [vertex] {};
\node (A9) at (9,0) [vertex] {};
\node (A10) at (10,0) [vertex] {};
\node (A11) at (11,0) [vertex] {};
\node (A12) at (12,0) [vertex] {};
\node (B0) at (0.5,-0.5) [vertex] {};
\node (B1) at (1.5,-0.5) [vertex] {};
\node (B2) at (2.5,-0.5) [vertex] {};
\node (B3) at (3.5,-0.5) [vertex] {};
\node (B4) at (4.5,-0.5) [vertex] {};
\node (B5) at (5.5,-0.5) [vertex] {};
\node (B6) at (6.5,-0.5) [vertex] {};
\node (B7) at (7.5,-0.5) [vertex] {};
\node (B8) at (8.5,-0.5) [vertex] {};
\node (B9) at (9.5,-0.5) [vertex] {};
\node (B10) at (10.5,-0.5) [vertex] {};
\node (B11) at (11.5,-0.5) [vertex] {};
\node (B12) at (12.5,-0.5) [vertex] {};
\node (C0) at (0,-1)  [vertex] {};
\node (C1) at (0.5,-1.1)  [vertex] {};
\node (C2) at (1,-1)  [vertex] {};
\node (C3) at (1.5,-1.1)  [vertex] {};
\node (C4) at (2,-1)  [vertex] {};
\node (C5) at (2.5,-1.1)  [vertex] {};
\node (C6) at (3,-1)  [vertex] {};
\node (C7) at (3.5,-1.1)  [vertex] {};
\node (C8) at (4,-1)  [vertex] {};
\node (C9) at (4.5,-1.1)  [vertex] {};
\node (C10) at (5,-1)  [vertex] {};
\node (C11) at (5.5,-1.1)  [vertex] {};
\node (C12) at (6,-1)  [vertex] {};
\node (C13) at (6.5,-1.1)  [vertex] {};
\node (C14) at (7,-1)  [vertex] {};
\node (C15) at (7.5,-1.1)  [vertex] {};
\node (C16) at (8,-1)  [vertex] {};
\node (C17) at (8.5,-1.1)  [vertex] {};
\node (C18) at (9,-1)  [vertex] {};
\node (C19) at (9.5,-1.1)  [vertex] {};
\node (C20) at (10,-1)  [vertex] {};
\node (C21) at (10.5,-1.1)  [vertex] {};
\node (C22) at (11,-1)  [vertex] {};
\node (C23) at (11.5,-1.1)  [vertex] {};
\node (C24) at (12,-1)  [vertex] {};
\node (C25) at (12.5,-1.1)  [vertex] {};
\node (D0) at (0.5,-1.5)  [vertex] {};
\node (D1) at (1.5,-1.5)  [vertex] {};
\node (D2) at (2.5,-1.5)  [vertex] {};
\node (D3) at (3.5,-1.5)  [vertex] {};
\node (D4) at (4.5,-1.5)  [vertex] {};
\node (D5) at (5.5,-1.5)  [vertex] {};
\node (D6) at (6.5,-1.5)  [vertex] {};
\node (D7) at (7.5,-1.5)  [vertex] {};
\node (D8) at (8.5,-1.5)  [vertex] {};
\node (D9) at (9.5,-1.5)  [vertex] {};
\node (D10) at (10.5,-1.5)  [vertex] {};
\node (D11) at (11.5,-1.5)  [vertex] {};
\node (D12) at (12.5,-1.5)  [vertex] {};
\node (E0) at (0,-2)  [vertex] {};
\node (E1) at (1,-2)  [vertex] {};
\node (E2) at (2,-2)  [vertex] {};
\node (E3) at (3,-2)  [vertex] {};
\node (E4) at (4,-2)  [vertex] {};
\node (E5) at (5,-2)  [vertex] {};
\node (E6) at (6,-2)  [vertex] {};
\node (E7) at (7,-2)  [vertex] {};
\node (E8) at (8,-2)  [vertex] {};
\node (E9) at (9,-2) [vertex] {};
\node (E10) at (10,-2)  [vertex] {};
\node (E11) at (11,-2)  [vertex] {};
\node (E12) at (12,-2)  [vertex] {};
\node (F0) at (0.5,-2.5)  [vertex] {};
\node (F1) at (1.5,-2.5)  [vertex] {};
\node (F2) at (2.5,-2.5)  [vertex] {};
\node (F3) at (3.5,-2.5)  [vertex] {};
\node (F4) at (4.5,-2.5)  [vertex] {};
\node (F5) at (5.5,-2.5)  [vertex] {};
\node (F6) at (6.5,-2.5)  [vertex] {};
\node (F7) at (7.5,-2.5)  [vertex] {};
\node (F8) at (8.5,-2.5)  [vertex] {};
\node (F9) at (9.5,-2.5)  [vertex] {};
\node (F10) at (10.5,-2.5)  [vertex] {};
\node (F11) at (11.5,-2.5)  [vertex] {};
\node (F12) at (12.5,-2.5)  [vertex] {};
\draw (0.5,0.5) circle (4.5pt);
\draw (4.5,0.5) circle (4.5pt);
\draw (6.5,0.5) circle (4.5pt);
\draw (8.5,0.5) circle (4.5pt);
\draw (12.5,0.5) circle (4.5pt);
\draw (3.5,-2.5) circle (4.5pt);
\draw (6.5,-2.5) circle (4.5pt);
\draw (9.5,-2.5) circle (4.5pt);
\end{tikzpicture}
\end{array}
\]
Further, although they were not drawn in \cite{IW}, the arrows in \eqref{RingelPicture} are implicit in the calculation of the quiver of the corresponding reconstruction algebra \cite[\S4]{WemGL2}. 
This paper grew out of trying to give a conceptual explanation for this coincidence, since a connection between the mathematics underpinning the two pictures did not seem to be known. 

In fact, the connection turns out to be explained by a very general phenomenon.  Recall first that one of the basic properties of the canonical algebra $\Lambda_{\bp,\bl}$ is that there is always a derived equivalence \cite{GL1}
\[
\Db(\coh\bX_{\bp,\bl})\simeq \Db(\mod \Lambda_{\bp,\bl})
\]
where $\coh\bX_{\bp,\bl}$ is the weighted projective line of Geigle--Lenzing (for details, see \S\ref{geom intro}).  Thus, to explain the above coincidence, we are led to consider the possibility of linking the weighted projective line, viewed as a Deligne--Mumford stack, to the study of rational surface singularities.  However, the weighted projective line $\bX_{\bp,\bl}$ cannot itself be the stack that we are after, since it only has dimension one, and rational surface singularities have, by definition, dimension two.

We need to increase the dimension, and the most naive way to do this is to consider the total space of a line bundle over $\bX_{\bp,\bl}$.  We thus choose any member of the grading group $\ox\in\bL$ and consider the total space stack $\bTot(\cO_{\bX}(-\ox))$ (for definition, see \S\ref{geom intro}).  From tilting on this and its coarse moduli, under mild assumptions we prove that the Veronese subring $S^{\ox}\colonequals \bigoplus_{i\in\bZ}S_{i\ox}$ is a weighted homogeneous rational surface singularity, giving the first concrete connection between the above two settings.  Furthermore, from the stack $\bTot(\cO_{\bX}(-\ox))$ we then describe the special CM $S^{\ox}$-modules, and give  precise information regarding the minimal resolution of $\Spec S^{\ox}$ and its derived category.  

Our results recover known special cases, such as the domestic case (corresponding to Dynkin diagrams), where it is known that $S^{-\ow}$ is a simple singularity.   In that setting there is an equivalence $\CM^{\bL}\!S\simeq\CM^{\bZ}\!S^{-\ow}$, and this leads us to investigate more general categorical equivalences.  We do this for very general $S$ and $\ox$, and through a range of categorical equivalences we are then able to relate $\CM^{\bZ}\!S^{\ox}$ and $\vect\bX$, which  finally allows us in Section~\ref{domestic section} to explain categorically why the above two pictures must be the same.

We now describe our results in detail.

\subsection{Veronese Subrings and Special CM modules}\label{geom intro}
Throughout, let $\KK$ denote an algebraically closed field of characteristic zero.  For any $n\geq 0$, choose positive integers $p_1,\hdots,p_n$ with all $p_i\geq 2$ and set $\bp\colonequals (p_1,\hdots,p_n)$.
Furthermore, choose pairwise distinct points $\lambda_1,\hdots,\lambda_n$ in $\bP^1$, and denote $\bl\colonequals (\lambda_1,\hdots,\lambda_n)$. Let $\ell_i(t_0,t_1)\in\KK[t_0,t_1]$ be the linear form defining $\lambda_i$, and write
\[
S_{\bp,\bl}=S\colonequals \frac{\KK[t_0,t_1,x_1,\hdots,x_n]}{(x_i^{p_i}-\ell_i(t_0,t_1)\mid 1\leq i\leq n)}.
\]
Moreover, let $\bL=\bL(p_1,\hdots,p_n)$ be the abelian group generated by the elements $\ox_1,\hdots,\ox_n$ subject to the relations
$p_1\ox_1=p_2\ox_2=\cdots =p_n\ox_n=:\oc$.  With this input $S$ is an $\bL$-graded algebra with $\deg x_i\colonequals \ox_i$ and $\deg t_j\colonequals \oc$, and $\bL$ is a rank one abelian group, possibly containing torsion. 
Often we normalize $\bl$ so that $\lambda_1=0$, $\lambda_2=\infty$ and $\lambda_3=1$, however it is important for changing parameters later that we allow ourselves flexibility.

From this, consider the stack
\[
\bX_{\bp,\bl}=\bX\colonequals [(\Spec S_{\bp,\bl}\backslash 0)/\Spec \KK\bL],
\]
with coarse moduli space denoted $X_{\bp,\bl}=X$.  It is well known that $X\cong \bP^1$, regardless of $\bp$ and $\bl$ (see~\ref{notation throughout}\eqref{notation throughout 2b}).

To increase the dimension we choose an element $\ox\in\bL$ and consider both the Veronese subring given by $S^{\ox}\colonequals \bigoplus_{i\in\bZ}S_{i\ox}$ and the total space stack
\[
\bT^{\ox} =  \bTot(\cO_{\bX_{\bp,\bl}}(-\ox))\colonequals [(\Spec S_{\bp,\bl}\backslash 0\times\Spec\KK[t])/\Spec \KK\bL],
\]
where $\bL$ acts on $t$ with weight $-\ox$.  Writing $\ox=\sum_{i=1}^na_i\ox_i+a\oc$ in normal form (see \ref{notation throughout}\eqref{notation throughout 2}), we show in \ref{sings on T} that the coarse moduli space $T^{\ox}$ is a surface containing a $\bP^1$, and on that $\bP^1$ complete locally the singularities of $T^{\ox}$ are of the form 
\begin{equation}\label{T picture}
\begin{array}{c}
\begin{tikzpicture}[scale=1.5] 
\draw (-0.2,0) to [bend left=5] node[pos=0.75,above] {$\scriptstyle \bP^1$} (4.2,0);
\filldraw [black] (0.0,0.01) circle (1pt);
\filldraw [black] (1,0.08) circle (1pt);
\filldraw [black] (2,0.11) circle (1pt);
\filldraw [black] (4,0.01) circle (1pt);
\node at (0,-0.15) {$\scriptstyle \frac{1}{p_1}(1,-a_1)$};
\node at (0,0.2) {$\scriptstyle \lambda_1$};
\node at (1,-0.1) {$\scriptstyle \frac{1}{p_2}(1,-a_2)$};
\node at (1,0.275) {$\scriptstyle \lambda_2$};
\node at (2,-0.075) {$\scriptstyle \frac{1}{p_3}(1,-a_3)$};
\node at (2,0.3) {$\scriptstyle \lambda_3$};
\node at (3,-0.1) {$\scriptstyle \hdots$};
\node at (4,-0.15) {$\scriptstyle \frac{1}{p_n}(1,-a_n)$};
\node at (4,0.2) {$\scriptstyle \lambda_n$};
\end{tikzpicture} 
\end{array}
\end{equation}
where for notation see \ref{cyclic quot def}. 

As is standard, positivity conditions on $\ox$ are needed in order to contract the zero section of $T^{\ox}$.  It turns out that the correct notion is to assume that $\ox\in\bL$ is not torsion, and further that $\ow-i\ox\notin\bL_+$ for all $i\geq 0$ (for the definition of $\ow$ and $\bL_+$ see \ref{notation throughout}\eqref{notation throughout 2}). Under this mild positivity restriction, we show in  \ref{we have a map!} that there is a canonical morphism
\[
\upgamma\colon T^{\ox}\to\Spec S^{\ox}
\]
satisfying $\RDerived \upgamma_*\!\cO_{T^{\ox}}=\cO_{S^{\ox}}$, and further in \ref{is projective} that $\upgamma$ is projective birational.  Composing $\upgamma$ with the minimal resolution $\upvarphi\colon Y^{\ox}\to T^{\ox}$ of $T^{\ox}$, gives the following.
\begin{thm}[=\ref{Sox is rational}]\label{Sox is rational intro}
If $\ox\in\bL$ is not torsion, and $\ow-i\ox\notin\bL_+$ for all $i\geq 0$, then $S^{\ox}$ is a rational surface singularity.
\end{thm}

In the setting of the above theorem, all the datum can be summarized by the following commutative diagram

\vspace{-0.6cm}

\begin{equation}
\begin{array}{c}
\begin{tikzpicture}
\node (top 1) at (0,0) {$\bT^{\ox}$};
\node (top 2) at (2.5,0) {$\bX$};
\node (bottom 1) at (0,-1.5) {$T^{\ox}$};
\node (bottom 2) at (2.5,-1.5) {$X\cong\mathbb{P}^1$};
\node (min) at (-1.5,-1) {$Y^{\ox}$};
\node (V) at (0,-2.5) {$\Spec S^{\ox}$};
\draw[->] (top 1) -- node[left] {$\scriptstyle g$} (bottom 1);
\draw[->] (top 2) -- node[right] {$\scriptstyle f$} (bottom 2);
\draw[->] (top 1) -- node[above] {$\scriptstyle q$} (top 2);
\draw[->] (bottom 1) -- node[above] {$\scriptstyle p$} (bottom 2);
\draw[->] (min) -- node[above] {$\scriptstyle \upvarphi$} (bottom 1);
\draw[->] (bottom 1) -- node[right] {$\scriptstyle \upgamma$} (V);
\draw[->,densely dotted] (min) -- node[below left] {$\scriptstyle \uppi$} (V);
\end{tikzpicture}
\end{array}\label{stack diagram 1}
\end{equation}

\vspace{-0.1cm}

We remark that the coarse moduli space $T^{\ox}$ is a singular line bundle in the sense of Dolgachev \cite[\S4]{Dolgachev} and Pinkham \cite[\S3]{Pinkham}, which also appears in the work of Orlik--Wagreich \cite{OW} and many others.  However, the key difference in our approach is that the grading group giving the quotient is $\bL$ not $\bZ$, and indeed it is the extra combinatorial structure of $\bL$ that allows us to extract the geometry much more easily.

It is in fact easy to check (see \ref{last comb hope}\eqref{last comb hope 1}) that the positivity condition on $\ox$ in \ref{Sox is rational intro} is satisfied if $0\neq \ox\in\bL_+$.  This setting is particularly pleasant, since tilting behaves well.
\begin{thm}[=\ref{tilting on coarse T}]
If $0\neq \ox\in\bL_+$, then $p^*(\cO_{\bP^1}\oplus\cO_{\bP^1}(1))$ is a tilting bundle on $T^{\ox}$.  
\end{thm}

Writing $\cE\colonequals \bigoplus_{i\in [0,\oc\,]}\cO_{\bX}(i)$ for the Geigle--Lenzing tilting bundle on $\bX$ \cite{GL1}, our next main result is then the following.

\begin{thm}\label{tilting intro}
If $0\neq\ox\in\bL_+$, then with notation as in \eqref{stack diagram 1},
\begin{enumerate}
\item\label{tilting intro 2}\textnormal{(=\ref{tilting on stack T})} $q^*\cE$ is a tilting bundle on $\bT^{\ox}$ such that 
\[
\begin{tikzpicture}[xscale=1]
\node (a1) at (0,0) {$\D(\Qcoh\bT^{\ox})$};
\node (a2) at (5,0) {$\D(\Mod\End_{\bT^{\ox}}(q^*\cE))$};
\node (b1) at (0,-1.5) {$\D(\Qcoh\bX)$};
\node (b2) at (5,-1.5) {$\D(\Mod\Lambda_{\bp,\bl})$};
\draw[->] (a1) -- node[above] {$\scriptstyle \RHom_{\bT^{\ox}}(q^*\cE,-)$} node[below] {$\scriptstyle \sim$} (a2);
\draw[->] (b1) -- node[above] {$\scriptstyle \RHom_{\bX}(\cE,-)$} node[below] {$\scriptstyle \sim$} (b2);
\draw[->] (a1) -- node[left] {$\scriptstyle \Rq_*$} (b1);
\draw[->] (a2) -- node[right] {$\mbox{\scriptsize res}$}  (b2);
\end{tikzpicture}
\]
commutes, where $\Lambda_{\bp,\bl}$ is the canonical algebra of Ringel.
\item\label{tilting intro 3}\textnormal{(=\ref{stackminres})} 
There is a fully faithful embedding $ 
\Db(\coh Y^{\ox})\hookrightarrow \Db(\coh\bT^{\ox})$.
\end{enumerate}
\end{thm}

It is by analysing \ref{tilting intro}\eqref{tilting intro 3} that we are able to extract the special CM $S^{\ox}$-modules below.  We show in \S\ref{changing parameters} that for any non-torsion $\ox=\sum_{i=1}^na_i\ox_i+a\oc$ we can change parameters and replace  $(\bp,\bl,\ox)$ by $(\bp'\!\!\colonequals (\sp_i\mid i\in I),\bl',
\osx\colonequals \sum_{i\in I}\sa_i\osx_i+a\osc\in\bL^{\! \prime}_{\phantom\prime})$ such that $S^{\ox}_{\bp,\bl}=S^{\osx}_{\bp',\bl'}$ and $(\sp_i,\sa_i)=1$ holds.  See \ref{change parameters so coprime} for full details.  With this in mind, we are then able to precisely control when the embedding in \ref{tilting intro}\eqref{tilting intro 3} is an equivalence.

\begin{prop}[=\ref{stackminres 2A}]\label{stackminres 2A intro}
If $0\neq\ox=\sum_{i=1}^na_i\ox_i+a\oc\in\bL_+$ with  $(p_i,a_i)=1$ for all $1\leq i\leq n$, then the embedding in  \ref{tilting intro}\eqref{tilting intro 3}  is an equivalence if and only if every $a_i$ is $1$, that is $\ox=\sum_{i=1}^{n}\ox_i+a\oc$.
\end{prop}

Combining the above tilting result with \eqref{T picture} and a combinatorial argument,  we are in fact able to determine the precise dual graph (for definition see \ref{dual graph defin}) of the morphism $\uppi$ in \eqref{stack diagram 1}.  Recall that for each $\frac{1}{p_i}(1,-a_i)$ in \eqref{T picture} with $a_i\neq 0$, we can consider the Hirzebruch--Jung continued fraction expansion
\begin{equation}
\frac{p_i}{p_i-a_i}=
\upalpha_{i1}-\frac{1}{\upalpha_{i2} - \frac{1}{\upalpha_{i3} -
\frac{1}{(...)}}} \colonequals [\upalpha_{i1},\hdots,\upalpha_{im_i}],\label{HJ intro}
\end{equation}
with each $\upalpha_{ij}\geq 2$; see \S\ref{iseries section} for full details.

\begin{thm}[{=\ref{degen lemma}, \ref{middle SI number}}]\label{dual graph general intro}
Let $0\neq\ox\in\bL_+$ and as above write $\ox=\sum_{i=1}^na_i\ox_i+a\oc$ in normal form. Then the dual graph of the morphism $\uppi\colon Y^{\ox}\to\Spec S^{\ox}$ is
\begin{equation}\label{key dual graph}
\begin{array}{c}
\begin{tikzpicture}[xscale=1,yscale=0.8]
\node (0) at (0,0) [vertex] {};
\node (A1) at (-3,1) [vertex]{};
\node (A2) at (-3,2) [vertex] {};
\node (A3) at (-3,3) [vertex] {};
\node (A4) at (-3,4) [vertex] {};
\node (B1) at (-1.5,1) [vertex] {};
\node (B2) at (-1.5,2) [vertex] {};
\node (B3) at (-1.5,3) [vertex] {};
\node (B4) at (-1.5,4) [vertex] {};
\node (C1) at (0,1) [vertex] {};
\node (C2) at (0,2) [vertex] {};
\node (C3) at (0,3) [vertex] {};
\node (C4) at (0,4) [vertex] {};
\node (n1) at (2,1) [vertex] {};
\node (n2) at (2,2) [vertex] {};
\node (n3) at (2,3) [vertex] {};
\node (n4) at (2,4) [vertex] {};
\node at (-3,2.6) {$\vdots$};
\node at (-1.5,2.6) {$\vdots$};
\node at (0,2.6) {$\vdots$};
\node at (2,2.6) {$\vdots$};
\node at (1,3.5) {$\hdots$};
\node at (1,1.5) {$\hdots$};
\node (T) at (0,4.25) {};
\node at (0,-0.2) {$\scriptstyle -\upbeta$};
\node at (-2.6,1) {$\scriptstyle -\upalpha_{11}$};
\node at (-2.6,2) {$\scriptstyle -\upalpha_{12}$};
\node at (-2.35,3) {$\scriptstyle -\upalpha_{1m_1-1}$};
\node at (-2.45,4) {$\scriptstyle -\upalpha_{1m_1}$};
\node at (-1.1,1) {$\scriptstyle -\upalpha_{21}$};
\node at (-1.1,2) {$\scriptstyle -\upalpha_{22}$};
\node at (-0.85,3) {$\scriptstyle -\upalpha_{2m_2-1}$};
\node at (-0.95,4) {$\scriptstyle -\upalpha_{2m_2}$};
\node at (0.4,1) {$\scriptstyle -\upalpha_{31}$};
\node at (0.4,2) {$\scriptstyle -\upalpha_{32}$};
\node at (0.65,3) {$\scriptstyle -\upalpha_{3m_3-1}$};
\node at (0.55,4) {$\scriptstyle -\upalpha_{3m_3}$};
\node at (2.45,1) {$\scriptstyle -\upalpha_{n1}$};
\node at (2.45,2) {$\scriptstyle -\upalpha_{n2}$};
\node at (2.7,3) {$\scriptstyle -\upalpha_{nm_n-1}$};
\node at (2.6,4) {$\scriptstyle -\upalpha_{nm_n}$};
\draw (A1) -- (0);
\draw (B1) -- (0);
\draw (C1) -- (0);
\draw (n1) -- (0);
\draw (A2) -- (A1);
\draw (B2) -- (B1);
\draw (C2) -- (C1);
\draw (n2) -- (n1);
\draw (A4) -- (A3);
\draw (B4) -- (B3);
\draw (C4) -- (C3);
\draw (n4) -- (n3);
\end{tikzpicture}
\end{array}
\end{equation}
where the arm $[\upalpha_{i1},\hdots,\upalpha_{im_i}]$ corresponds to $i\in\{1,\ldots,n\}$ with $a_i\neq0$, and the $\upalpha_{ij}$ are given by the Hirzebruch--Jung continued fractions in \eqref{HJ intro}.
Furthermore, writing $v=\#\{i\mid a_i\neq 0\}$ for the number of arms, we have $\upbeta=a+v$.
\end{thm}

We first establish in \ref{when Y min res} that $\uppi$ is the minimal resolution if and only if $\ox\notin[0,\oc\,]$.  Theorem~\ref{dual graph general intro} is then proved by splitting into the two cases $\ox\notin[0,\oc\,]$ and $\ox\in[0,\oc\,]$, with the verification in both cases being rather different.  Note that the case $\ox\in [0,\oc\,]$ is degenerate as $[0,\oc\,]$ is a finite interval, containing only those $\ox$ of the form $a_i\ox_i$ for some $i$ and some $0\leq a_i\leq p_i$.  In this paper we are mostly interested in special CM modules and these are defined using the minimal resolution: this is why below the condition $\ox\notin[0,\oc\,]$ often appears.

We remark that for $0\neq \ox\in\bL_+$, $S^{\ox}$ is rarely a quotient singularity, and it is even more rare for it to be ADE.  Nevertheless, the dual graphs of all quotient singularities $\KK^{2}/G$ (where $G$ is a small subgroup of $\GL(2,\KK)$) are known \cite{Brieskorn}, and so whether $S^{\ox}$ is a quotient singularity can, if needed, be immediately determined by \ref{dual graph general intro}, after contracting ($-1$)-curves if necessary.

One key observation in this paper is that controlling the stack $\bT^{\ox}$ allows us not only to obtain a rational surface singularity $S^{\ox}$, with its dual graph, but furthermore it also allows us to determine the special CM $S^{\ox}$-modules.  Indeed, in effect we simply compare the two resolutions  
\[
\begin{tikzpicture}
\node (A) at (-1,1) {$Y^{\ox}$};
\node (B) at (1,1) {$\bT^{\ox}$};
\node (base) at (0,-0.25) {$\Spec S^{\ox}$};
\draw[->] (A) --node[left]{$\scriptstyle \uppi$}(base);
\draw[->] (B) --node[right]{$\scriptstyle \uppi^\prime\colonequals \upgamma\circ g$}(base);
\end{tikzpicture}
\]
constructed above.  It is known that $Y^{\ox}$ has a tilting bundle $\cM$  \cite{VdB1d, Wunram}, and by \ref{tilting intro}\eqref{tilting intro 2} that $\bT^{\ox}$ has tilting bundle $q^*\cE$, where $\cE$ is the Geigle--Lenzing tilting bundle on $\bX$.  Pushing these down to $\Spec S^{\ox}$ gives the following result. Throughout, we write $\SCM S^{\ox}$ for the category of special CM $S^{\ox}$-modules; for definitions see \S\ref{Prelim ReconAlg}.  For $\oy\in\bL$, write $S(\oy)^{\ox}\colonequals \bigoplus_{i\in\bZ}S_{\oy+i\ox}$.

\begin{thm}
If $0\neq \ox\in\bL_+$ with $\ox\notin [0,\oc\,]$, then the following hold.
\begin{enumerate}
\item{\rm \cite{Wunram}} $\SCM S^{\ox}= \add\uppi_*\cM$.
\item\textnormal{(=\ref{stackminres})} $\uppi_*\cM$ is a summand of $\uppi^\prime_*(q^*\cE)=\bigoplus_{\oy\in[0,\oc\,]}S(\oy)^{\ox}$.
\end{enumerate}
\end{thm}

Furthermore, we say precisely which summands of $\uppi^\prime_*(q^*\cE)$ give the special CM modules.  As notation, recall that the $i$-series associated to the Hirzebruch--Jung continued fraction expansion $\frac{r}{a}=[\upalpha_1,\hdots,\upalpha_m]$ is defined as $i_0=r$, $i_1=a$ and  $i_{t}=\upalpha_{t-1}i_{t-1}-i_{t-2}$ for all $t$ with $2\leq t\leq m+1$, and we write
\[
I(r,a)\colonequals \{i_0,i_1,\hdots,i_{m+1}\}.
\]  
As convention $I(r,r)=\emptyset$. 

\begin{thm}[{=\ref{specials determined thm}, \ref{specials lag approach thm}}]\label{specials general intro}
If $\ox\in\bL_+$ with $\ox\notin [0,\oc\,]$,  write $\ox$ in normal form $\ox=\sum_{i=1}^na_i\ox_i+a\oc$.
Then 
\[
\SCM S^{\ox}=\add\{S(u\ox_j)^{\ox}\mid j\in [1,n],  u\in I(p_j,p_j-a_j)   \}.
\]
\end{thm}
This allows us to construct both $R=S^{\ox}$, and its special CM modules, for (almost) every star shaped dual graph.  We remark that this is the first time that special CM modules have been classified in any example with infinite CM representation type, and indeed, due to the non--tautness of the dual graph, in an uncountable family of examples.  For simplicity in this paper, we restrict the explicitness to certain families of examples, and refer the reader to \S\ref{Sect 5.2} for more details. 

By construction, all the special CM $S^{\ox}$-modules have a natural $\bZ$-grading, and we let $N$ denote their sum.  By definition the \emph{reconstruction algebra} is defined to be $\Upgamma_{\ox}\colonequals \End_{S^{\ox}}(N)$, and in this setting it inherits a $\bN$-grading from the grading of the special CM modules in \ref{specials general intro}.  In general, it is not generated  in degree one over its degree zero piece, but nevertheless the degree zero piece is always some canonical algebra of Ringel.  We state the first half of the following result vaguely, giving a much more precise description of the parameters in \ref{deg 0 general prop}.
\begin{prop}\label{deg 0 intro}
Suppose that $x\in\bL_+$ with $\ox\notin [0,\oc\,]$.
\begin{enumerate}
\item\label{deg 0 intro 1} \textnormal{(=\ref{deg 0 general prop})} The degree zero part of $\Upgamma_{\ox}$ is isomorphic to the canonical algebra  $\Lambda_{\bq,\bm}$, for some suitable parameters $(\bq,\bm)$.
\item\label{deg 0 intro 2} \textnormal{(=\ref{grading inherited})} For $\os\colonequals \sum_{i=1}^n\ox_i$, then $\Upgamma_{\os}$ is generated  in degree one over its degree zero piece. Moreover the degree zero piece is the canonical algebra $\Lambda_{\bp,\bl}$.
\end{enumerate}
\end{prop}

\subsection{Geigle--Lenzing Weighted Projective Lines via Rational Surface Singularities}\label{Section 1.3}

For an abelian group $G$ and a $G$-graded noetherian $\KK$-algebra, we write $\mod^G\! A$ for the category of finitely generated $G$-graded $A$-modules, $\mod^G_0\! A$ for the subcategory of finite dimensional modules, and $\qgr^G\! A\colonequals \mod^G\! A\,/\!\mod^G_0\! A$ for the Serre quotient.  Motivated by the above, and also the fact that when studying curves it should not matter how we embed them into surfaces (and thus be independent of any self-intersection numbers that appear), we then investigate when $\qgr^\bZ\! S^{\ox}\simeq \coh\bX$.

In very special cases, $\coh\mathbb{X}_{\bp,\bl}$ is already known to be equivalent to $\qgr^\bZ\! R$ for some connected graded commutative ring $R$ \cite[8.4]{GL91}.  The nicest situation is when the star-shaped dual graph is of Dynkin type, and further $R$ is the ADE quotient singularity associated to the Dynkin diagram via the McKay correspondence (with a slightly non-standard grading).  However, all the previous attempts to link the weighted projective line to rational singularities have taken all self-intersection numbers to be $-2$, which is well-known to restrict the possible configurations to ADE Dynkin type.

One of our main results is the following, which does not even require that $\ox\in\bL_+$.
 
\begin{thm}[{=\ref{WPL as qgrZ}}]\label{WPL as qgrZ intro}
Suppose that $\ox=\sum_{i=1}^na_i\ox_i+a\oc$ is not torsion, and write $R\colonequals S^{\ox}$.  Then the following conditions are equivalent.
\begin{enumerate}
\item The natural functor $(-)^{\ox}\colon\CM^\bL\!S\to\CM^\bZ\!R$ is an equivalence.
\item The natural functor $(-)^{\ox}\colon\qgr^\bL\!S\to\qgr^\bZ\!R$ is an equivalence.
\item For any $\oz\in\bL$, the ideal $I^{\oz}\colonequals S(\oz)^{\ox}\cdot S(-\oz)^{\ox}$ of $R$ satisfies $\dim_\KK(R/I^{\oz})<\infty$.
\item $(p_i,a_i)=1$ for all $1\le i\le n$.
\end{enumerate}
\end{thm}

The above theorem implies that for a non-torsion element $\ox=\sum_{i=1}^na_i\ox_i+a\oc$ of $\bL$, there is an equivalence  $\qgr^\bZ\!S_{\bp,\bl}^{\ox}\simeq \coh\bX_{\bp,\bl}$ if and only if $(p_i,a_i)=1$ for all $i$ with $1\leq i\leq n$. Thus, by choosing a suitable $\ox$, the weighted projective line can be defined using only connected $\bN$-graded rational surface singularities.  Also, we remark that in the case $(p_i,a_i)\neq 1$ we still have that $\qgr S^{\ox}$ is equivalent to some weighted projective line, but the parameters are no longer $(\bp,\bl)$.  We leave the details to \S\ref{qgr Veronese section}. 

Combining the above gives our next main result.

\begin{cor}[{=\ref{WPL as qgrZ}, \ref{qgrR via qgrLambda}}]\label{S=R=Gamma intro}
Let $\ox\in\bL_+$ with $\ox\notin[0,\oc\,]$, and write $\ox=\sum_{i=1}^na_i\ox_i+a\oc$ in normal form.  If $(p_i,a_i)=1$ for all $1\leq i\leq n$, then
\[
\coh\bX_{\bp,\bl}\simeq \qgr^\bZ\! S^{\ox}\simeq\qgr^\bZ\!\Upgamma_{\ox},
\]
and further $\Upgamma_{\ox}$ is an $\bN$-graded ring, with zeroth piece a canonical algebra.
\end{cor}
In the case when $(p_i,a_i)\neq 1$ we have a similar result but again there is a change of parameters, so we refer the reader to \ref{qgrR via qgrLambda} for details.  Combining \ref{S=R=Gamma intro} with \ref{deg 0 intro}\eqref{deg 0 intro 2}, we can view the weighted projective line $\bX_{\bp,\bl}$ as an Artin--Zhang noncommutative projective scheme over the canonical algebra $\Lambda_{\bp,\bl}$ \cite{Minamoto}. 

Note that \ref{WPL as qgrZ intro}(2)$\Leftrightarrow$(4) was shown independently in \cite[6.6]{CCZ}.  

\subsection{Some Particular Veronese Subrings}
We then investigate the particular Veronese subrings $S^{\os_a}$ for $\os_a\colonequals \os+a\oc$ for some $a\geq 0$, where $\os\colonequals \sum_{i=1}^n\ox_i$.  We call $S^{\os_a}$ the \emph{$a$-Wahl Veronese subring}, and in this case, the singularities in \eqref{T picture} are all of the form $\frac{1}{p_i}(1,-1)$, which are cyclic Gorenstein and so have a resolution consisting of only ($-2$)-curves.  Thus resolving the singularities in \eqref{T picture}, by \ref{dual graph general intro} we see that the dual graph of the minimal resolution of $\Spec S^{\os_a}$ is 
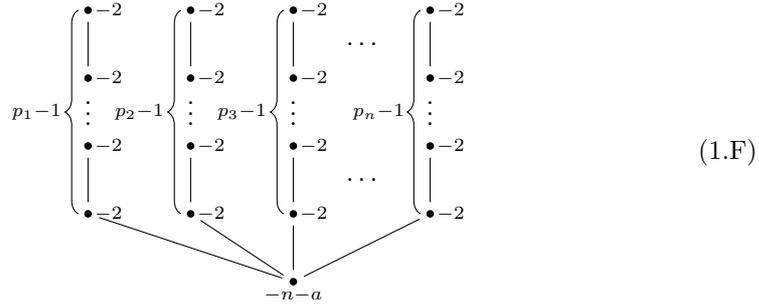
\begin{equation}\label{s Veron dual graph}
\begin{array}{c}
\begin{tikzpicture}[xscale=0.9,yscale=0.9]
\node (0) at (0,0) [vertex] {};
\node (A1) at (-3,1) [vertex] {};
\node (A2) at (-3,2) [vertex] {};
\node (A3) at (-3,3) [vertex] {};
\node (A4) at (-3,4) [vertex] {};
\node (B1) at (-1.5,1) [vertex] {};
\node (B2) at (-1.5,2) [vertex] {};
\node (B3) at (-1.5,3) [vertex] {};
\node (B4) at (-1.5,4) [vertex] {};
\node (C1) at (0,1) [vertex] {};
\node (C2) at (0,2) [vertex] {};
\node (C3) at (0,3) [vertex] {};
\node (C4) at (0,4) [vertex] {};
\node (n1) at (2,1) [vertex] {};
\node (n2) at (2,2) [vertex] {};
\node (n3) at (2,3) [vertex] {};
\node (n4) at (2,4) [vertex] {};
\node at (-3,2.6) {$\vdots$};
\node at (-1.5,2.6) {$\vdots$};
\node at (0,2.6) {$\vdots$};
\node at (2,2.6) {$\vdots$};
\node at (1,3.5) {$\hdots$};
\node at (1,1.5) {$\hdots$};
\node (T) at (0,4.25) {};
\node at (0,-0.2) {$\scriptstyle -n-a$};
\node at (-2.7,1) {$\scriptstyle -2$};
\node at (-2.7,2) {$\scriptstyle -2$};
\node at (-2.7,3) {$\scriptstyle -2$};
\node at (-2.7,4) {$\scriptstyle -2$};
\node at (-1.2,1) {$\scriptstyle -2$};
\node at (-1.2,2) {$\scriptstyle -2$};
\node at (-1.2,3) {$\scriptstyle -2$};
\node at (-1.2,4) {$\scriptstyle -2$};
\node at (0.3,1) {$\scriptstyle -2$};
\node at (0.3,2) {$\scriptstyle -2$};
\node at (0.3,3) {$\scriptstyle -2$};
\node at (0.3,4) {$\scriptstyle -2$};
\node at (2.3,1) {$\scriptstyle -2$};
\node at (2.3,2) {$\scriptstyle -2$};
\node at (2.3,3) {$\scriptstyle -2$};
\node at (2.3,4) {$\scriptstyle -2$};
\draw (A1) -- (0);
\draw (B1) -- (0);
\draw (C1) -- (0);
\draw (n1) -- (0);
\draw (A2) -- (A1);
\draw (B2) -- (B1);
\draw (C2) -- (C1);
\draw (n2) -- (n1);
\draw (A4) -- (A3);
\draw (B4) -- (B3);
\draw (C4) -- (C3);
\draw (n4) -- (n3);
\draw [decorate,decoration={brace,amplitude=5pt},xshift=-4pt,yshift=0pt]
(2,1) -- (2,4) node [black,midway,xshift=-0.55cm] 
{$\scriptstyle p_n-1$};
\draw [decorate,decoration={brace,amplitude=5pt},xshift=-4pt,yshift=0pt]
(0,1) -- (0,4) node [black,midway,xshift=-0.55cm] 
{$\scriptstyle p_3-1$};
\draw [decorate,decoration={brace,amplitude=5pt},xshift=-4pt,yshift=0pt]
(-1.5,1) -- (-1.5,4) node [black,midway,xshift=-0.55cm] 
{$\scriptstyle p_2-1$};
\draw [decorate,decoration={brace,amplitude=5pt},xshift=-4pt,yshift=0pt]
(-3,1) -- (-3,4) node [black,midway,xshift=-0.55cm] 
{$\scriptstyle p_1-1$};
\end{tikzpicture}
\end{array}
\end{equation}
where there are $n$ arms, and the number of vertices on arm $i$ is $p_i-1$.  It turns out that these particular Veronese subrings have many nice properties; not least by \ref{stackminres 2A intro} they are precisely the Veronese subrings for which 
\[
\Db(\coh Y^{\ox})\hookrightarrow \Db(\coh\bT^{\ox})
\]
is an equivalence.  In \S\ref{domestic section} we investigate $S^{\os_a}$ in the case when $(p_1,p_2,p_3)$ forms a Dynkin triple, in which case $S^{\os_a}$ is isomorphic to a quotient singularity by some finite subgroup of $\GL(2,\KK)$ of type $\mathbb{T}$, $\mathbb{O}$ or $\mathbb{I}$ (see \ref{Veronese=quotient} for details).  In this situation $S^{\os_a}$ and its reconstruction algebra have a very nice relationship to the preprojective algebra of the canonical algebra, and this is what turns out to explain the motivating coincidence from \S\ref{motivation section} in \ref{equiv last intro} below.

For arbitrary parameters $(\bp,\bl)$, the Veronese subring $S^{\os}$ has a particularly nice form.

\begin{thm}[{=\ref{S is generated}}]
For any $\mathbb{X}_{\bp,\bl}$, $S^{\os}$ is generated by the homogeneous elements 
\begin{eqnarray*}
\mathsf{u}_i&\colonequals &\left\{\begin{array}{ll}x_1^{p_1+p_2}x_3^{p_2}\hdots x_n^{p_2}&i=1,\\
x_2^{p_1+p_2}x_3^{p_1}\hdots x_n^{p_1}&i=2,\\ 
-x_1^{p_i}x_2^{p_2+p_i}x_3^{p_i}\hdots \widehat{x_i}\hdots x_n^{p_i}&3\leq i\leq n,\end{array}\right.\\
\mathsf{v}&\colonequals &x_1x_2\hdots x_n.
\end{eqnarray*}
\end{thm}

\begin{prop}[{=\ref{dual graph assignment}}]
With notation as above, the modules $S(u\ox_j)^{\os}$ appearing in \ref{specials general intro} are precisely the following ideals of $S^{\os}$, and furthermore they correspond to the dual graph of the minimal resolution of $\Spec S^{\os}$ \eqref{s Veron dual graph} in the following way:
\[
\begin{array}{c}
\begin{tikzpicture}[xscale=1,yscale=1,bend angle=30, looseness=1]
\node (0) at (0,0) {$\scriptstyle (\mathsf{v}^{p_2},\mathsf{u}_1)$};
\node (A1) at (-3.25,1) {$\scriptstyle (\mathsf{v}^{p_2+1},\mathsf{u}_1)$};
\node (A2) at (-3.25,2) {$\scriptstyle (\mathsf{v}^{p_2+2},\mathsf{u}_1)$};
\node (A3) at (-3.25,3) {$\scriptstyle (\mathsf{v}^{p_2+p_1-2},\mathsf{u}_1)$};
\node (A4) at (-3.25,4) {$\scriptstyle (\mathsf{v}^{p_2+p_1-1},\mathsf{u}_1)$};
\node (B1) at (-1.5,1) {$\scriptstyle (\mathsf{u}_1,\mathsf{v}^{p_2-1})$};
\node (B2) at (-1.5,2) {$\scriptstyle (\mathsf{u}_1,\mathsf{v}^{p_2-2})$};
\node (B3) at (-1.5,3) {$\scriptstyle (\mathsf{u}_1,\mathsf{v}^2)$};
\node (B4) at (-1.5,4) {$\scriptstyle (\mathsf{u}_1,\mathsf{v})$};
\node (C1) at (0,1) {$\scriptstyle (\mathsf{u}_3,\mathsf{v}^{p_3-1})$};
\node (C2) at (0,2) {$\scriptstyle (\mathsf{u}_3,\mathsf{v}^{p_3-2})$};
\node (C3) at (0,3) {$\scriptstyle (\mathsf{u}_3,\mathsf{v}^{2})$};
\node (C4) at (0,4) {$\scriptstyle (\mathsf{u}_3,\mathsf{v})$};
\node (n1) at (2,1) {$\scriptstyle (\mathsf{u}_n,\mathsf{v}^{p_n-1})$};
\node (n2) at (2,2) {$\scriptstyle (\mathsf{u}_n,\mathsf{v}^{p_n-2})$};
\node (n3) at (2,3) {$\scriptstyle (\mathsf{u}_n,\mathsf{v}^{2})$};
\node (n4) at (2,4) {$\scriptstyle (\mathsf{u}_n,\mathsf{v})$};
\node at (-3.25,2.6) {$\vdots$};
\node at (-1.5,2.6) {$\vdots$};
\node at (0,2.6) {$\vdots$};
\node at (2,2.6) {$\vdots$};
\node at (1,2.5) {$\hdots$};
\draw [-] (A1)+(-30:8.5pt) -- ($(0) + (160:11.5pt)$);
\draw [-] (B1) --(0);
\draw [-] (C1) --(0);
\draw [-] (n1) --(0);
\draw [-] (A2) -- (A1);
\draw [-] (B2) --(B1);
\draw [-] (C2) --(C1);
\draw [-] (n2) --(n1);
\draw [-] (A4) -- (A3);
\draw [-] (B4) --(B3);
\draw [-] (C4) --(C3);
\draw [-] (n4) --  (n3);
\end{tikzpicture}
\end{array}
\]
\end{prop}

The relations between $\mathsf{u}_1,\hdots,\mathsf{u}_n,\mathsf{v}$ turn out to be easy to describe, and remarkably have already appeared in the literature. It is  well-known \cite[3.6]{Wahl} that there is a family of rational surface singularities $R_{\bp,\bl}$ where the dual graph of the minimal resolution of $\Spec R_{\bp,\bl}$ is precisely \eqref{dual graph} with $a=0$.  Indeed, in \cite{Wahl} $R_{\bp,\bl}$ is defined as follows: given the same data $(\bp,\bl)$ as above normalised so that $\lambda_1=(1:0)$, $\lambda_2=(0:1)$ and $\lambda_3,\hdots,\lambda_n\in\KK^*$ are pairwise distinct, we can consider the commutative $\KK$-algebra $R_{\bp,\bl}$, generated by $u_1,\hdots,u_n,v$ subject to the relations given by the $2\times 2$ minors of the matrix
\[
\left(
\begin{array}{ccccc}
u_2&u_3&\hdots&u_{n}&v^{p_2}\\
v^{p_1}&\lambda_3u_3+v^{p_3}&\hdots&\lambda_nu_n+v^{p_n}&u_1
\end{array}
\right)
\]
This is a connected $\mathbb{N}$-graded ring graded by $\deg v\colonequals 1$, $\deg u_1\colonequals p_2$, $\deg u_2\colonequals p_1$ and $\deg u_i\colonequals p_i$ for all $3\leq i\leq n$.  

We show that $S^{\os}$ recovers precisely the above $R_{\bp,\bl}$.
\begin{thm}[{=\ref{Veronese result}}]\label{R and S intro}
There is an isomorphism $R_{\bp,\bl}\cong S^{\os}$ of $\bZ$-graded algebras given by $u_i\mapsto \mathsf{u}_i$ for $1\leq i\leq n$ and $v\mapsto\mathsf{v}$.
\end{thm}

Thus the Veronese method we develop in this paper for constructing rational surface singularities recovers as a special case the example of \cite{Wahl}, but in a way suitable for arbitrary labelled star-shaped graphs, and also in a way suitable for obtaining the special CM modules.  

We then present the reconstruction algebra of $R_{\bp,\bl}\cong S^{\os}$, since again in this situation it has a particularly nice form. In principle, using \ref{specials general intro}, we can do this for any Veronese $S^{\ox}$ with $0\neq\ox\in\bL_+$, but for notational ease we restrict ourselves to the case $\ox=\os$.  

\begin{thm}[=\ref{recon relations}]
The reconstruction algebra $\Upgamma_{\bp,\bl}$ of $R_{\bp,\bl}$ can be written explicitly as a quiver with relations.   It is the path algebra of the double of the quiver $Q_{\bp}$ of the canonical algebra, subject to the relations induced by the canonical relations, and furthermore at every vertex, all 2-cycles that exist at that vertex are equal.
\end{thm}

We refer the reader to \ref{recon relations} for more details, but remark that the reconstruction algebra was originally invented in order to extend the notion of a preprojective algebra to a more general geometric setting.  In our situation here, the reconstruction algebra is not quite the preprojective algebra of the canonical algebra $\Lambda_{\bp,\bl}$, but the relations in \ref{recon relations} are mainly of the same form as the preprojective relations; the reconstruction algebra should perhaps be thought of as a better substitute.

In the last section of the paper, finally we can explain the coincidence of the two motivating pictures, as a consequence of the following result. 

\begin{thm}[=\ref{equiv last}]\label{equiv last intro}
Let $R$ be the $(m-3)$-Wahl Veronese subring associated with $(p_1,p_2,p_3)=(2,3,3)$, $(2,3,4)$ or $(2,3,5)$ and $m\ge3$, and $\mathfrak{R}$ its completion. Let $G\leq \bL$ be the cyclic group generated by $(h(m-2)+1)\ow$, where $h=6$, $12$ or $30$ respectively.  Then
\begin{enumerate}
\item There are equivalences $\vect\bX\simeq\CM^\bZ\!R$ and 
\[
F\colon (\vect\bX)/G\xrightarrow{\sim}\CM\mathfrak{R},
\]
where $(\vect\bX)/G$ is the complete orbit category (for the definition, see \S\ref{domestic section}).
\item For the canonical tilting bundle $\mathcal{E}$ on $\bX$, we have $\SCM\mathfrak{R}=\add F\mathcal{E}$.
\end{enumerate}
\end{thm}

\medskip
\noindent
\textbf{Acknowledgements.}  The authors thank Kazushi Ueda and Atsushi Takahashi for many helpful comments and remarks, and  Mitsuyasu Hashimoto, Ryo Takahashi and Yuji Yoshino for valuable discussions on reflexive modules. Part of this work was completed when O.I.\ visited Edinburgh in March and September 2012, and he thanks people in Edinburgh for hospitality during his visit.

\medskip
\noindent
\textbf{Conventions.}  Throughout, $\KK$ denotes an algebraically closed field of characteristic zero. All modules will be right modules, and for a ring $A$ write $\mod A$ for the category of finitely generated right $A$-modules.  If $G$ is an abelian group and $A$ is a noetherian $G$-graded ring, $\gr^G\! A$ will denote the category of finitely generated $G$-graded right $A$-modules.  Throughout when composing maps $fg$ will mean $f$ then $g$, similarly for arrows $ab$ will mean $a$ then $b$.  Note that with this convention $\Hom_R(M,N)$ is an $\End_R(M)^{\op}$-module and an $\End_R(N)$-module.  For $M\in\mod A$ we write $\add M$ for the full subcategory consisting of summands of finite direct sums of copies of $M$.

\section{Preliminaries}\label{prelim}

\subsection{Notation} We first fix notation.  For $n\geq 0$, choose positive integers $p_1,\hdots,p_n$ with all $p_i\geq 2$, and set $\bp\colonequals (p_1,\hdots,p_n)$.  Likewise, for pairwise distinct points $\lambda_1,\hdots,\lambda_n\in\mathbb{P}^1$,  set $\bl\colonequals (\lambda_1,\hdots,\lambda_n)$.   Let $\ell_i(t_0,t_1)\in\KK[t_0,t_1]$ be the linear form defining $\lambda_i$. 

\begin{notation}\label{notation throughout}{\rm
To this data we associate the following.
\begin{enumerate}
\item\label{notation throughout 2} The abelian group $\bL=\bL(p_1,\hdots,p_n)$ generated by the elements $\ox_1,\hdots,\ox_n$ subject to the relations
$p_1\ox_1=p_2\ox_2=\cdots =p_n\ox_n=:\oc$.  Note that $\bL$ is an ordered group with $\bL_+=\sum_{i=1}^n\mathbb{Z}_{\geq 0}\ox_i$ as positive elements .  Since $\bL/\bZ\oc\cong\prod_{i=1}^n\bZ/p_i\bZ$ canonically, each $\ox\in\bL$ can be written uniquely in \emph{normal form} as $\ox=\sum_{i=1}^na_i\ox_i+a\oc$ with $0\leq a_i<p_i$ and $a\in\mathbb{Z}$.  Then $\ox$ belongs to $\bL_+$ if and only if $a\geq0$.  The \emph{dualizing element} $\ow\in\bL$ is defined to be
\[
\ow\colonequals (n-2)\oc\,-\sum_{i=1}^n\ox_i.
\]
\item\label{notation throughout 2b} The commutative $\KK$-algebra $S_{\bp,\bl}$ defined as
\[
S_{\bp,\bl}=S\colonequals \frac{\KK[t_0,t_1,x_1,\hdots,x_n]}{(x_i^{p_i}-\ell_i(t_0,t_1)\mid 1\leq i\leq n)}.
\]
As in the introduction, this is $\bL$-graded by $\deg x_i\colonequals \ox_i$, and defines the weighted projective line $\mathbb{X}_{\bp,\bl}\colonequals [(\Spec S\backslash 0)/\Spec \KK\bL]$. Then its coarse moduli space $X_{\bp,\bl}\colonequals (\Spec S\backslash 0)/\Spec \KK\bL$ is $\bP^1$.
In fact, the open cover $\Spec S\backslash0=U_0\cup U_1$ with $U_i\colonequals \Spec S_{t_i}$ induces an open cover $X_{\bp,\bl}=X_0\cup X_1$ with $X_i\colonequals \Spec (S_{t_i})_0$, where $(S_{t_i})_0$ is the degree zero part of $S_{t_i}$.
By inspection $(S_{t_i})_0=\KK[t_{1-i}/t_i]$, and it follows that $X_{\bp,\bl}\cong\bP^1$.
\end{enumerate}
When $n\geq2$, often we choose $\lambda_1=(1:0)$ and $\lambda_2=(0:1)$, in which case $\ell_1=t_1$, $\ell_2=t_0$ and $\ell_i=\lambda_it_0-t_1$ for $3\le i\le n$, and there is a presentation
\[
S_{\bp,\bl}=\frac{\KK[x_1,\hdots,x_n]}{(x_i^{p_i}+x_1^{p_1}-\lambda_ix_2^{p_2}\mid 3\leq i\leq n)}.
\]
Moreover, when $n\geq 2$, we can further associate  
\begin{enumerate}[resume]
\item The quiver
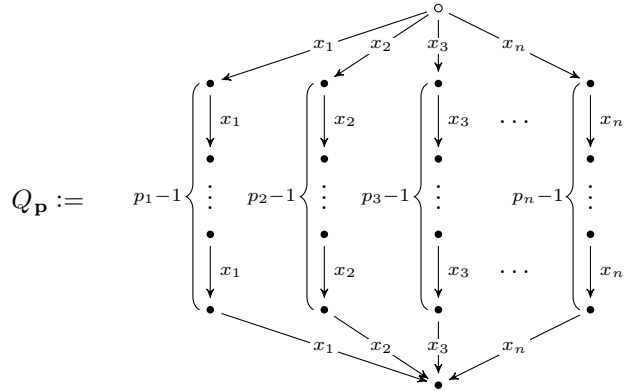
\[
\begin{array}{cc}
Q_{\bp}\colonequals  &
\begin{array}{c}
\begin{tikzpicture}[xscale=1,yscale=1]
\node (0) at (0,0) [vertex] {};
\node (A1) at (-3,1) [vertex] {};
\node (A2) at (-3,2) [vertex] {};
\node (A3) at (-3,3) [vertex] {};
\node (A4) at (-3,4) [vertex] {};
\node (B1) at (-1.5,1) [vertex] {};
\node (B2) at (-1.5,2) [vertex] {};
\node (B3) at (-1.5,3) [vertex] {};
\node (B4) at (-1.5,4) [vertex] {};
\node (C1) at (0,1) [vertex] {};
\node (C2) at (0,2) [vertex] {};
\node (C3) at (0,3) [vertex] {};
\node (C4) at (0,4) [vertex] {};
\node (n1) at (2,1) [vertex] {};
\node (n2) at (2,2) [vertex] {};
\node (n3) at (2,3) [vertex] {};
\node (n4) at (2,4) [vertex] {};
\node at (-3,2.6) {$\vdots$};
\node at (-1.5,2.6) {$\vdots$};
\node at (0,2.6) {$\vdots$};
\node at (2,2.6) {$\vdots$};
\node at (1,3.5) {$\hdots$};
\node at (1,1.5) {$\hdots$};
\node (T) at (0,5) [cvertex] {};
\draw [->] (A1) -- node[fill=white,inner sep=1pt]{$\scriptstyle x_1$}(0);
\draw [->] (B1) -- node[fill=white,inner sep=1pt]{$\scriptstyle x_2$}(0);
\draw [->] (C1) -- node[fill=white,inner sep=1pt]{$\scriptstyle x_3$}(0);
\draw [->] (n1) -- node[fill=white,inner sep=1pt]{$\scriptstyle x_n$}(0);
\draw [->] (A2) -- node[right] {$\scriptstyle x_1$} (A1);
\draw [->] (B2) -- node[right]  {$\scriptstyle x_2$}(B1);
\draw [->] (C2) -- node[right]  {$\scriptstyle x_3$}(C1);
\draw [->] (n2) -- node[right]  {$\scriptstyle x_n$}(n1);
\draw [->] (A4) -- node[right] {$\scriptstyle x_1$} (A3);
\draw [->] (B4) -- node[right]  {$\scriptstyle x_2$}(B3);
\draw [->] (C4) -- node[right]  {$\scriptstyle x_3$}(C3);
\draw [->] (n4) -- node[right]  {$\scriptstyle x_n$}(n3);
\draw [->] (T) -- node[fill=white,inner sep=1pt]{$\scriptstyle x_1$}(A4);
\draw [->] (T) -- node[fill=white,inner sep=1pt]{$\scriptstyle x_2$}(B4);
\draw [->] (T) -- node[fill=white,inner sep=1pt]{$\scriptstyle x_3$}(C4);
\draw [->] (T) -- node[fill=white,inner sep=1pt]{$\scriptstyle x_n$}(n4);
\draw [decorate,decoration={brace,amplitude=5pt},xshift=-4pt,yshift=0pt]
(2,1) -- (2,4) node [black,midway,xshift=-0.55cm] 
{$\scriptstyle p_n-1$};
\draw [decorate,decoration={brace,amplitude=5pt},xshift=-4pt,yshift=0pt]
(0,1) -- (0,4) node [black,midway,xshift=-0.55cm] 
{$\scriptstyle p_3-1$};
\draw [decorate,decoration={brace,amplitude=5pt},xshift=-4pt,yshift=0pt]
(-1.5,1) -- (-1.5,4) node [black,midway,xshift=-0.55cm] 
{$\scriptstyle p_2-1$};
\draw [decorate,decoration={brace,amplitude=5pt},xshift=-4pt,yshift=0pt]
(-3,1) -- (-3,4) node [black,midway,xshift=-0.55cm] 
{$\scriptstyle p_1-1$};
\end{tikzpicture}
\end{array}
\end{array}
\]
(where there are $n$ arms, and the number of vertices on arm $i$ is $p_i-1$).
\item The \emph{canonical algebra} $\Lambda_{\bp,\bl}$, namely the path algebra of the quiver $Q_{\bp}$ subject to the relations 
\[
I\colonequals \langle x_1^{p_1}-\lambda_ix_2^{p_2}+x_i^{p_i}\mid 3\leq i\leq n\rangle.
\]
There is a degenerate definition of the canonical algebra if $0\leq n\leq 1$; see \cite{GL1}.
\item\label{notation throughout 5} The commutative $\KK$-algebra $R_{\bp,\bl}$, generated by $u_1,\hdots,u_n,v$ subject to the relations given by the $2\times 2$ minors of the matrix
\[
\left(
\begin{array}{ccccc}
u_2&u_3&\hdots&u_{n}&v^{p_2}\\
v^{p_1}&\lambda_3u_3+v^{p_3}&\hdots&\lambda_nu_n+v^{p_n}&u_1
\end{array}
\right)
\]
This is a connected $\mathbb{N}$-graded ring graded by $\deg u_1\colonequals p_2$, $\deg u_2\colonequals p_1$, $\deg v\colonequals 1$, and $\deg u_i\colonequals p_i$ for all $3\leq i\leq n$.  
\end{enumerate}
We will also consider
\begin{enumerate}[resume]
\item Star-shaped graphs of the form
\begin{equation}\label{dual graph}
\begin{array}{c}
\begin{tikzpicture}[xscale=1,yscale=1]
\node (0) at (0,0) [vertex] {};
\node (A1) at (-3,1) [vertex]{};
\node (A2) at (-3,2) [vertex] {};
\node (A3) at (-3,3) [vertex] {};
\node (A4) at (-3,4) [vertex] {};
\node (B1) at (-1.5,1) [vertex] {};
\node (B2) at (-1.5,2) [vertex] {};
\node (B3) at (-1.5,3) [vertex] {};
\node (B4) at (-1.5,4) [vertex] {};
\node (C1) at (0,1) [vertex] {};
\node (C2) at (0,2) [vertex] {};
\node (C3) at (0,3) [vertex] {};
\node (C4) at (0,4) [vertex] {};
\node (n1) at (2,1) [vertex] {};
\node (n2) at (2,2) [vertex] {};
\node (n3) at (2,3) [vertex] {};
\node (n4) at (2,4) [vertex] {};
\node at (-3,2.6) {$\vdots$};
\node at (-1.5,2.6) {$\vdots$};
\node at (0,2.6) {$\vdots$};
\node at (2,2.6) {$\vdots$};
\node at (1,3.5) {$\hdots$};
\node at (1,1.5) {$\hdots$};
\node (T) at (0,4.25) {};
\node at (0,-0.2) {$\scriptstyle -\upbeta$};
\node at (-2.6,1) {$\scriptstyle -\upalpha_{11}$};
\node at (-2.6,2) {$\scriptstyle -\upalpha_{12}$};
\node at (-2.35,3) {$\scriptstyle -\upalpha_{1n_1-1}$};
\node at (-2.45,4) {$\scriptstyle -\upalpha_{1n_1}$};
\node at (-1.1,1) {$\scriptstyle -\upalpha_{21}$};
\node at (-1.1,2) {$\scriptstyle -\upalpha_{22}$};
\node at (-0.85,3) {$\scriptstyle -\upalpha_{2n_2-1}$};
\node at (-0.95,4) {$\scriptstyle -\upalpha_{2n_2}$};
\node at (0.4,1) {$\scriptstyle -\upalpha_{31}$};
\node at (0.4,2) {$\scriptstyle -\upalpha_{32}$};
\node at (0.65,3) {$\scriptstyle -\upalpha_{3n_3-1}$};
\node at (0.55,4) {$\scriptstyle -\upalpha_{3n_3}$};
\node at (2.45,1) {$\scriptstyle -\upalpha_{v1}$};
\node at (2.45,2) {$\scriptstyle -\upalpha_{v2}$};
\node at (2.7,3) {$\scriptstyle -\upalpha_{vn_v-1}$};
\node at (2.6,4) {$\scriptstyle -\upalpha_{vn_v}$};
\draw (A1) -- (0);
\draw (B1) -- (0);
\draw (C1) -- (0);
\draw (n1) -- (0);
\draw (A2) -- (A1);
\draw (B2) -- (B1);
\draw (C2) -- (C1);
\draw (n2) -- (n1);
\draw (A4) -- (A3);
\draw (B4) -- (B3);
\draw (C4) -- (C3);
\draw (n4) -- (n3);
\end{tikzpicture}
\end{array}
\end{equation}
where there are $v\geq 2$ arms, each $n_i\geq 1$, each $\upalpha_{ij}\geq 2$ and $\upbeta\geq 1$. Later, we will assume $\upbeta\geq v$.
\end{enumerate}}
\end{notation}

We next give some properties of $S_{\bp,\bl}$ and related rings that will be required later, all being elementary in nature.
We start with a general result. Let $G$ be an abelian group and $A$ a $G$-graded ring.
Then $A$ is a \emph{$G$-domain} whenever a product of non-zero homogeneous elements is again non-zero, and a \emph{$G$-field} if any non-zero homogeneous element is invertible. A homogeneous element $x\in A$ is \emph{$G$-prime} if $A/(x)$ is an $G$-domain.
If $A$ is a $G$-domain, then the \emph{quotient $G$-field} of $A$ is the localization of $A$ with respect to the set of all non-zero homogeneous elements.
A $G$-domain is \emph{$G$-factorial} if every non-zero homogeneous element in $A$ is a product of $G$-prime elements in $A$. 

\begin{prop}\label{L-factorial}
With notation as above,
\begin{enumerate}
\item\label{L-factorial 1} $S_{\bp,\bl}$ is an $\bL$-factorial $\bL$-domain.
\item\label{L-factorial 3} Any $G$-factorial $G$-domain $A$ is \emph{$G$-normal}, i.e.\ if a homogeneous element $x$ in the quotient $G$-field of $A$ satisfies an equality $x^m+a_1x^{m-1}+\cdots+a_m=0$ for some $a_i\in A$, then $x\in A$. In particular, $A_0$ is a normal domain (in the usual sense).
\item\label{L-factorial 4} Let $A$ be a $G$-field and $A[y]$ the $G$-graded polynomial ring with a homogeneous indeterminate $y$. Then any homogeneous ideal of $A[y]$ is principal.
\item\label{L-factorial 5} If $A$ is a $G$-factorial $G$-domain, the localization of $A$ by a set of homogeneous elements is $G$-factorial.  The $G$-graded polynomial ring $A[y_1,\ldots,y_m]$ with homogeneous indeterminates $y_1,\ldots,y_m$ is also $G$-factorial.
\end{enumerate}
\end{prop}
\begin{proof}
Part \eqref{L-factorial 1} is well known, see e.g. \cite[1.3]{GL1} or \cite[Section 2.2]{HIMO}. Parts \eqref{L-factorial 3}, \eqref{L-factorial 4} and the first assertion of \eqref{L-factorial 5} are easy. The second assertion of \eqref{L-factorial 5} follows from a parallel argument to the classical case \cite[Section VII.3.5]{B4} using \eqref{L-factorial 4} and the $G$-version of Gauss's Lemma.
\end{proof}

Let $S=S_{\bp,\bl}$.
Recall from the introduction that for $\ox\in\bL$, $S^{\ox}\colonequals \bigoplus_{i\in\bZ}S_{i\ox}$. We will also be interested in the $\bN$-graded version, so define $S^{\bN\ox}\colonequals \bigoplus_{i\geq 0}S_{i\ox}$.

\begin{cor}\label{Two dim and normal prelim}
With notation as above, let $\ox\in\bL$.
\begin{enumerate}
\item\label{L-factorial 2} If $\ox\in\bL$ is not torsion, then $S^{\ox}$ is a noetherian $\KK$-algebra with $\dim S^{\ox}=2$, and $S$ is a finitely generated $S^{\ox}$-module.
\item\label{L-factorial 6} Let $S[t]$ be the $\bL$-graded polynomial ring with $\deg t=-\ox$. Then $(S[t])_0\cong S^{\bN\ox}$ holds, and this is a normal domain.
\item\label{L-factorial 7} Suppose  $-i\ox\notin \bL_+$ for all $i> 0$.  Then $S^{\ox}$ is a noetherian normal domain with $\dim S^{\ox}=2$, and has at worst a unique singular point corresponding to $\bigoplus_{i>0}S_{i\ox}$.
\end{enumerate}
\end{cor}
\begin{proof}
\eqref{L-factorial 2} Since $\ox$ is not torsion, $\bZ\ox\subseteq\bL$ has finite index, and so the first two assertions of \eqref{L-factorial 2} are easy; see e.g.\ \ref{Veron is fg2}.  In particular, necessarily $\dim S^{\ox}=\dim S=2$. \\
\eqref{L-factorial 6} The equality $(S[t])_0=\bigoplus_{i\geq 0}S_{i\ox}t^i\cong S^{\bN\ox}$ is clear. Furthermore, $S[t]$ is an $\bL$-factorial $\bL$-domain according to \ref{L-factorial}\eqref{L-factorial 1}\eqref{L-factorial 5}. Thus $(S[t])_0$ is a normal domain by \ref{L-factorial}\eqref{L-factorial 3}.\\
\eqref{L-factorial 7} The assumption $-i\ox\notin \bL_+$ for all $i> 0$  implies that $\ox$ is not torsion, since otherwise $-Nx=0\in\bL_+$ for some $N>0$. It also forces $S^{\ox}=S^{\bN\ox}$, so the first half of the result follows by combining parts \eqref{L-factorial 2}  and \eqref{L-factorial 6}. 

The second half is a general property of a positively graded two-dimensional normal domain (e.g.\ \cite[p1]{Pinkham}). In fact, since $S^{\ox}$ is a $\bZ$-graded finitely generated $\KK$-algebra, by the Jacobian criterion, there is a $\bZ$-graded ideal $I$ of $S^{\ox}$ such that $\Sing S^{\ox}=\Spec(S^{\ox}/I)$. 
Since $S^{\ox}$ is normal, all singularities of $S^{\ox}$ are isolated and $\dim_\KK(S^{\ox}/I)<\infty$ holds. Since $I$ is $\bZ$-graded, it contains $\bigoplus_{i>\ell}S_{i\ox}$ for $\ell\gg0$ and hence $\sqrt{I}$ contains $\bigoplus_{i>0}S_{i\ox}$. Thus $\Sing S^{\ox}\subseteq\{\bigoplus_{i>0}S_{i\ox}\}$.
\end{proof}

The following will be required later, and all are well known  (see \cite{GL1}). 

\begin{lemma}\label{basic observation}
If $\ox=\sum_{i=1}^na_i\ox+a\oc\in\bL_+$ is in normal form, then the following hold.
\begin{enumerate}
\item\label{basic observation 1} $S_{\ox}=(\prod_{i=1}^nx_i^{a_i})S_{a\oc}$.
\item\label{basic observation 2} $S_{a\oc}$ is an $(a+1)$-dimensional vector space,
and a basis of $S_{a\oc}$ is given by
$t_0^\ell t_1^{a-\ell}$ with $0\le\ell\le a$.
\item\label{basic observation 3} $S_{\ox+m\oc}=S_{\ox}\cdot S_{m\oc}$ for all $m\geq 0$.
\end{enumerate}
\end{lemma}

\subsection{Preliminaries on Rational Surface Singularities}\label{prelim RSS}
We briefly review some combinatorics associated to rational surface singularities.  Let $R$ be a finitely generated noetherian $\KK$-algebra, or alternatively the completion of such an algebra at a maximal ideal.
Recall that $R$ is said to be a \emph{rational surface singularity} if $\dim R=2$ and there exists $f\colon X\to\Spec R$ a resolution such that $\mathbf{R}f_*\mathcal{O}_X=\mathcal{O}_R$. If this property holds for one resolution, it holds for all resolutions \cite[5.10]{KM}, and automatically $R$ must be normal \cite[5.8]{KM}.

In our setting later $R$ will be a rational surface singularity with a unique singular point, at the origin.  Completing at this maximal ideal to give $\mathfrak{R}$,  in the minimal resolution $Y\to\Spec \mathfrak{R}$ the fibre above the origin is well-known to be a tree (i.e.\ a finite connected graph with no cycles) of $\mathbb{P}^1$s denoted $\{ E_i\}_{i\in I}$.  Their self-intersection numbers satisfy $E_i\cdot E_i\leq -2$, and moreover the intersection matrix $(E_i\cdot E_j)_{i,j\in I}$ is negative definite. We encode the intersection matrix in the form of the labelled dual graph:

\begin{defin}\label{dual graph defin}
Suppose that $\{ E_i \}_{i\in I}$  is a collection of $\mathbb{P}^1$s forming the exceptional locus in a resolution of some rational surface singularity.  The dual graph is defined as follows: for each curve $E_i$ there is a vertex, with $E_i\cdot E_j$ edges connecting  the vertices corresponding to $E_i$ and $E_j$.  Furthermore, every vertex is labelled with the self-intersection number of the corresponding curve.
\end{defin}

The dual graph of a complete local rational surface singularity is well known to be a labelled tree (see e.g.\ \cite[1.3]{Brieskorn}).   Conversely, suppose that $T$ is a tree, with vertices denoted $E_1,\hdots,E_n$, labelled with integers $w_1,\hdots,w_n$.  To this data we associate the symmetric matrix $M_T=(b_{ij})_{1\leq i,j\leq n}$ with $b_{ii}$ defined by $b_{ii}\colonequals w_i$, and $b_{ij}$ (with $i\neq j$) defined to be the number of edges linking the vertices $E_i$ and $E_j$.  We write $\cZ$ for the free abelian group generated by the vertices $E_i$, and call its elements \emph{cycles}.  The matrix $M_T$ defines a symmetric bilinear form $(-,-)$ on $\cZ$ and in analogy with the geometry, we will often write $Y\cdot Z$ instead of $(Y,Z)$, and consider
\[
\cZ_{\top}\colonequals \{ Z=\sum_{i=1}^na_iE_i\in\cZ\mid Z\neq 0, \mbox{ all }a_i\geq 0, \mbox{ and }Z\cdot E_i\leq 0 \mbox{ for all } i \}.
\]
If there exists $Z\in\cZ_{\top}$ such that $Z\cdot Z<0$, then automatically $M_T$ is negative definite \cite[Prop 2(ii)]{Artin}.   In this case, $\cZ_{\top}$ admits a unique smallest element $Z_f$, called the \emph{fundamental cycle}. Whenever all the coefficients in $Z_f$ are one, the fundamental cycle is said to be \emph{reduced}.

We now consider the case of the labelled graph \eqref{dual graph} and calculate some combinatorics that will be needed later.  Denoting the set of vertices of \eqref{dual graph} by $I$, considering $Z\colonequals \sum_{i\in I}E_i$ it is easy to see that 
\begin{equation}\label{Zf dot Ei}
\begin{array}{cc}
(-Z\cdot E_i)_{i\in I}=&
\begin{array}{c}
\begin{tikzpicture}[xscale=1,yscale=1]
\node (0) at (0,0) {$\scriptstyle \upbeta-v$};
\node (A1) at (-3,1) {$\scriptstyle \upalpha_{11}-2$};
\node (A2) at (-3,2) {$\scriptstyle \upalpha_{12}-2$};
\node (A3) at (-3,3) {$\scriptstyle \upalpha_{1n_1-1}-2$};
\node (A4) at (-3,4) {$\scriptstyle \upalpha_{1n_1}-1$};
\node (B1) at (-1.5,1) {$\scriptstyle \upalpha_{21}-2$};
\node (B2) at (-1.5,2) {$\scriptstyle \upalpha_{22}-2$};
\node (B3) at (-1.5,3) {$\scriptstyle \upalpha_{2n_2-1}-2$};
\node (B4) at (-1.5,4) {$\scriptstyle \upalpha_{2n_2}-1$};
\node (C1) at (0,1) {$\scriptstyle \upalpha_{31}-2$};
\node (C2) at (0,2) {$\scriptstyle \upalpha_{32}-2$};
\node (C3) at (0,3) {$\scriptstyle \upalpha_{3n_3-1}-2$};
\node (C4) at (0,4) {$\scriptstyle \upalpha_{3n_3}-1$};
\node (n1) at (2,1) {$\scriptstyle \upalpha_{v1}-2$};
\node (n2) at (2,2) {$\scriptstyle \upalpha_{v2}-2$};
\node (n3) at (2,3) {$\scriptstyle \upalpha_{vn_v-1}-2$};
\node (n4) at (2,4) {$\scriptstyle \upalpha_{vn_v}-1$};
\node at (-3,2.6) {$\vdots$};
\node at (-1.5,2.6) {$\vdots$};
\node at (0,2.6) {$\vdots$};
\node at (2,2.6) {$\vdots$};
\node at (1,3.5) {$\hdots$};
\node at (1,1.5) {$\hdots$};
\node (T) at (0,4.25) {};
\draw (A1) -- (0);
\draw (B1) -- (0);
\draw (C1) -- (0);
\draw (n1) -- (0);
\draw (A2) -- (A1);
\draw (B2) -- (B1);
\draw (C2) -- (C1);
\draw (n2) -- (n1);
\draw (A4) -- (A3);
\draw (B4) -- (B3);
\draw (C4) -- (C3);
\draw (n4) -- (n3);
\end{tikzpicture}
\end{array}
\end{array}
\end{equation}
and so $Z$ satisfies $Z\cdot E_i\leq 0$ for all $i\in I$ if and only if $\upbeta\geq v$. Since $\cZ_{\top}$ does not contain elements smaller than $Z$, the fundamental cycle $Z_f$ is given by $Z=\sum_{i\in I}E_i$ if and only if $\upbeta\geq v$. In this case $Z_f$ is reduced.

We remark that in general there will be many singularities with dual graph \eqref{dual graph}, and indeed a labelled graph $T$ is called \emph{taut} if there exists a unique  rational surface singularity (up to isomorphism in the category of complete local $\KK$-algebras) which has  $T$ for its dual graph of its minimal resolution.  It is well known that the labelled graph \eqref{dual graph} is taut if and only if $v=3$ \cite{Laufer}.   

\subsection{Preliminaries on Reconstruction Algebras}\label{Prelim ReconAlg}
Let $R$ be a rational surface singularity. A CM $R$-module $M$ is called \emph{special} if $\Ext^1_R(M,R)=0$ \cite{IW}, and we write $\SCM R$ for the category of special CM $R$-modules. 

The following local-to-global lemma is useful. In particular, if $R$ has a unique singular point $\m$, to conclude that $\add M=\SCM R$ it suffices to check this complete locally at $\m$.  

\begin{lemma}\label{Check locally gen}
Let $R$ be a rational surface singularity, and $M\in \CM R$.  If $\add \widehat{M}_\m=\SCM \widehat{R}_\m$ for all $\m\in\Max R$, then  $\add M=\SCM R$. 
\end{lemma}
\begin{proof}
Since Ext groups localise and complete, certainly $M\in\SCM R$ and thus $\add M\subseteq \SCM R$. Next, let $X\in\SCM R$.  Then $\add\widehat{X}_\m\subseteq\SCM \widehat{R}_{\m}$ for all $\m\in\Max R$, so by assumption $\add\widehat{X}_\m\subseteq\add \widehat{M}_\m$ for all $\m\in\Max R$.  By \cite[2.26]{IW4} we conclude that $\add X\subseteq\add M$, so $X\in\add M$ and thus $\add M\supseteq \SCM R$.
\end{proof}

The following asserts that a global additive generator of $\SCM R$ exists, regardless of the number of points in the singular locus.

\begin{thm}[{\cite{VdB1d}}]\label{recon Db min res}
Let $R$ be a rational surface singularity, and $\uppi\colon X\to \Spec R$ the minimal resolution.  Then the following statements hold.
\begin{enumerate}
\item There exists $M\in\SCM R$ such that $\SCM R=\add M$. 
\item There is a triangle equivalence $\Db(\mod\End_R(M))\cong\Db(\coh X)$.
\end{enumerate}
\end{thm}
\begin{proof}
This is known but usually only stated when $R$ is complete, so for the convenience of the reader we provide a proof. By \cite[3.2.5]{VdB1d} there is a progenerator $\cO_X\oplus\cM$ for the category of perverse sheaves (with perversity $-1$), which induces an equivalence
\[
\Db(\mod\End_X(\cO_X\oplus\cM))\cong\Db(\coh X).
\]  
There is an isomorphism $\End_X(\cO_X\oplus\cM)\cong\End_R(R\oplus \uppi_*\cM)$ by \cite[4.1]{DW2}. Furthermore, $\cO_X\oplus\cM$ remains a progenerator under flat base change \cite[3.1.6]{VdB1d}, so $\add \widehat{M}_\m=\SCM \widehat{R}_\m$ by \cite{Wunram,IW}.  The result then follows using \ref{Check locally gen}.
\end{proof}

\begin{defin}
For any $M\in \SCM R$ such that $\SCM R=\add M$, we call $\End_R(M)$ the reconstruction algebra.
\end{defin}
In this global setting, the reconstruction algebra is only defined up to Morita equivalence. Only after completing $R$, or in certain other settings (see \ref{Gamma convention remark}) will there be a canonical choice.

When $\mathfrak{R}$ is a complete local rational surface singularity with minimal resolution $X\to\Spec\mathfrak{R}$, there is a much more explicit description of the additive generator of $\SCM \mathfrak{R}$.  Let $\{E_i\mid i\in I\}$ denote the irreducible exceptional curves, then for each $i\in I$, by \cite{Wunram} there exists a CM $\mathfrak{R}$-module $M_i$ such that $H^1(\mathcal{M}_i^\vee)=0$ and $c_1(\mathcal{M}_i)\cdot E_j=\updelta_{ij}$ hold, where $\mathcal{M}_i\colonequals \uppi^*(M_i)^{\vee\vee}$ for $(-)^\vee=\sHom_X(-,\cO_X)$. 

\begin{thm}[{\cite[1.2]{Wunram}}]\label{Wunram main} There is a bijection
\[
\begin{array}{c}
\begin{tikzpicture}
\node (A) at (-2.5,0) {$\{ \mbox{irreducible exceptional curves in min.\ resolution}\}$};
\node (B) at (6,0) {$
\{\mbox{non-free, indecomposable special CM $\mathfrak{R}$-modules} \}
$};
\draw[<->] (2.1,0) -- node [above] {$\scriptstyle $} (1.45,0);
\draw[|->] (1.45,-0.6) -- node [above] {$\scriptstyle $} (2.1,-0.6);
\node at (1,-0.6) {$E_i$};
\node at (2.6,-0.6) {$M_i.$} ;
\end{tikzpicture}
\end{array}
\]
Furthermore, the rank of $M_i$, as an $\mathfrak{R}$-module, coincides with the co-efficient of $E_i$ in $Z_f$.
\end{thm} 

It follows that $\mathfrak{R}\oplus\bigoplus_{i\in I}M_i$ is the natural additive generator for $\SCM\mathfrak{R}$.
\begin{defin}
Let $\mathfrak{R}$ be a complete local rational surface singularity. We call $\Upgamma\colonequals \End_{\mathfrak{R}}(\mathfrak{R}\oplus(\bigoplus_{i\in I}M_i))$
the \emph{reconstruction algebra} of $\mathfrak{R}$.
\end{defin}

\begin{remark}\label{Gamma convention remark}{\rm
If $R$ is a rational surface singularity with a unique singular point, and if there exist $L_i\in \CM R$ such that $\widehat{L}_i\cong M_i$ for all $i$, then we also use the letter $\Upgamma$ to denote the particular reconstruction algebra 
\[
\Upgamma\colonequals \End_R(R\oplus\bigoplus_{i\in I}L_i)
\]
of $R$.  Such $L_i$ are not guaranteed to exist, in general.}
\end{remark}

In the complete local setting, the quiver of $\Upgamma$, and the number of its relations, is completely determined by the intersection theory.  

\begin{thm}[{\cite[3.3]{WemGL2}}]\label{GL2 for R}
Let $\mathfrak{R}$ be a complete local rational surface singularity. The quiver and the numbers of relations of $\Upgamma$ is given as follows: for every $i\in I$ associate a vertex labelled $i$ corresponding to $M_i$, and also associate a vertex labelled $\begin{tikzpicture} \node at (0,0) [cvertex] {};\end{tikzpicture}$ corresponding to $\mathfrak{R}$.  Then the number of arrows and relations between the vertices is
\[
\begin{tabular}{*4c}
\toprule
&&Number of arrows&Number of relations\\
\midrule
$i\rightarrow j$&&$(E_i\cdot E_j)_+$  & $(-1-E_i\cdot E_j)_+$\\
$\begin{tikzpicture} \node at (0,0) [cvertex] {};\end{tikzpicture}\rightarrow\begin{tikzpicture} \node at (0,0) [cvertex] {};\end{tikzpicture}$&&$0$&$-Z_K\cdot Z_f+1=-1-Z_f\cdot Z_f$\\
$i\rightarrow \begin{tikzpicture} \node at (0,0) [cvertex] {};\end{tikzpicture}$&&$-E_i\cdot Z_f$&$0$\\
$\begin{tikzpicture} \node at (0,0) [cvertex] {};\end{tikzpicture}\rightarrow i$&& $((Z_K-Z_f)\cdot E_i)_+$& $((Z_K-Z_f)\cdot E_i)_-$\\
\bottomrule\\
\end{tabular}
\]
where for $a\in\bZ$ 
\[
a_+\colonequals \left\{\begin{array}{ccc}a& \t{if}& a\geq 0\\ 0 & \t{if}&a<0\end{array}\right. \quad \t{and}\quad a_-=\left\{\begin{array}{rcc}0& \t{if} &a\geq 0\\ -a & \t{if}&a<0\end{array}\right. ,
\]  
and the canonical cycle $Z_K$ is by definition the rational cycle defined by the condition $Z_K\cdot E_i=E_i^2+2$ for all $i\in I$.
\end{thm}

\subsection{Hirzebruch--Jung Continued Fraction Combinatorics}\label{iseries section} We review briefly the notation and combinatorics surrounding dimension two cyclic quotient singularities.

\begin{defin}\label{cyclic quot def}
For $r,a\in\bN$ with $r>a$ the group
$G=\frac{1}{r}(1,a)$ is defined by
\[
G= \left\langle \zeta\colonequals \left(\begin{array}{cc} \ve & 0\\ 0& \ve^a
\end{array}\right)\right\rangle \leq \GL(2,\KK),
\]
where $\ve$ is a primitive $r^{\rm th}$ root of unity.
By abuse of notation, we also write $\frac{1}{r}(1,a)$ for the corresponding quotient singularity $\KK[x,y]^G$.
\end{defin}

\begin{remark}{\rm
In the literature it is often assumed that the greatest common divisor $(r,a)$ is $1$, which is equivalent to the group having no pseudoreflections.  However we do not make this assumption, since in our construction later groups with pseudoreflections naturally appear.}
\end{remark}

Provided that $a\neq 0$, we consider the Hirzebruch--Jung continued fraction expansion of $\frac{r}{a}$, namely
\[
\frac{r}{a}=\upalpha_1-\frac{1}{\upalpha_2 - \frac{1}{\upalpha_3 -
\frac{1}{(...)}}} \colonequals [\upalpha_1,\hdots,\upalpha_n]
\]
with each $\upalpha_i\geq 2$.  The labelled Dynkin diagram
\[
\begin{tikzpicture}[xscale=1.2]
  \node (1) at (1,0) [vertex] {};
  \node (2) at (2,0) [vertex] {};
 \node (3) at (4,0) [vertex] {};
 \node (4) at (5,0)[vertex] {};
 \node at (3,0) {$\cdots$};
  \node (1a) at (0.9,-0.25) {$\scriptstyle -\upalpha_{1}$};
  \node (2a) at (1.9,-0.25) {$\scriptstyle - \upalpha_{2}$};
 \node (3a) at (3.9,-0.25) {$\scriptstyle -\upalpha_{n-1}$};
 \node (4a) at (4.9,-0.25) {$\scriptstyle - \upalpha_{n}$};
\draw [-] (1) -- (2);
\draw [-] (2) -- (2.6,0);
\draw [-] (3.4,0) -- (3);
\draw [-] (3) -- (4);
\end{tikzpicture}
\]
is precisely the dual graph of the
minimal resolution of $\KK^{2}/\frac{1}{r}(1,a)$ \cite[Satz8]{Riemen}.  Note that \cite{Riemen} assumed the condition $(r,a)=1$, but the result holds generally: if we write $h\colonequals (r,a)$, then the quotient singularities $\frac{1}{r}(1,a)$ and $\frac{1}{r/h}(1,a/h)$ are isomorphic, and furthermore both have the same Hirzebruch--Jung continued fraction expansion.

\begin{defin}\label{combdata}
For integers $1\leq a<r$ as above, consider the continued fraction expansion $\frac{r}{a}=[\upalpha_1,\hdots,\upalpha_n]$.  Then the $i$-series is defined as $i_0=r$, $i_1=a$ and 
\[
i_{t}=\upalpha_{t-1}i_{t-1}-i_{t-2}
\]
for all $t$ with $2\leq t\leq n+1$. Thus $i_{n+1}=0$ holds. Let $I(r,a)\colonequals \{i_0,i_1,\hdots,i_{n+1}\}$, and by convention $I(r,r)\colonequals\emptyset$.
\end{defin}

The following lemma is elementary, and will be needed later.

\begin{lemma}\label{i series all}
For integers $1\leq a<r$, $I(r,a)=[0,r]$ if and only if $a=r-1$.
\end{lemma}

For a cyclic quotient singularity $G=\frac{1}{r}(1,a)$,  consider  
\[
S_{t}\colonequals \{ f\in\KK[x,y] \mid \sigma\cdot f=\upvarepsilon^t f \},
\]
for $t\in [0,r]$, and note that $S_0\cong S_r$.  Further, for $k$ with $0\leq k\leq r-1$,  we say that a monomial $x^my^n$ has \emph{weight $k$} if $m+an=k$ mod $r$, that is $x^my^n\in S_k$.  It is the $i$-series that determines which CM $S^G$-modules are special.

\begin{thm}\label{Wunram specials thm}
For $G=\frac{1}{r}(1,a)$,   
\begin{enumerate}
\item \textnormal{\cite{Herzog}} $\CM S^G=\add \{ S_t\mid t\in [0,r] \}$.
\item \textnormal{\cite{WunramCyclic}} $\SCM S^G=\add \{ S_{t}\mid t\in I(r,a) \}$.
\end{enumerate}
\end{thm}
\begin{proof}
Both results are usually stated in the complete case, with no pseudoreflections, so since we are working more generally, we give the proof. Since $S^G$ has a unique singular point, by \ref{Check locally gen} (and its counterpart in the $\CM S^G$ case) it suffices to prove both results in the complete local setting.  In this case, when $(r,a)=1$, part (1) is \cite{Herzog} and part (2) is \cite{WunramCyclic}. When $(r,a)\neq1$, the result is still true since $\frac{1}{r}(1,a)=\frac{1}{r/h}(1,a/h)$ for $h\colonequals (r,a)$.
\end{proof}

In what follows, we will require a different characterization of members of the $i$-series, by reinterpreting a result of Ito \cite[3.7]{Ito}.  As notation, if $(r,a)=1$ then the \emph{$G$-basis} $B(G)$ is defined to be the set of monomials which are not divisible by any $G$-invariant monomial.  We usually draw $B(G)$ in a $2\times 2$ grid.

\begin{example}{\rm 
Consider $G=\frac{1}{17}(1,10)$.  Then $B(G)$ is 
\[
\begin{tikzpicture}[xscale=0.8,yscale=0.4]
\node at (0,0) {$\scriptstyle 1$};
\node at (0,-1) {$\scriptstyle x$};
\node at (0,-2) {$\scriptstyle x^2$};
\node at (0,-3) {$\scriptstyle x^3$};
\node at (0,-4) {$\scriptstyle x^4$};
\node at (0,-5) {$\scriptstyle x^5$};
\node at (0,-6) {$\scriptstyle x^6$};
\node at (0,-7) {$\scriptstyle x^7$};
\node at (0,-8) {$\scriptstyle \vdots$};
\node at (0,-9) {$\scriptstyle x^{16}$};
\node at (1,0) {$\scriptstyle y^{\phantom 2}$};
\node at (1,-1) {$\scriptstyle xy^{\phantom 2}$};
\node at (1,-2) {$\scriptstyle x^2y$};
\node at (1,-3) {$\scriptstyle x^3y$};
\node at (1,-4) {$\scriptstyle x^4y$};
\node at (1,-5) {$\scriptstyle x^5y$};
\node at (1,-6) {$\scriptstyle x^6y$};
\node at (2,0) {$\scriptstyle y^2$};
\node at (2,-1) {$\scriptstyle xy^2$};
\node at (2,-2) {$\scriptstyle x^2y^2$};
\node at (2,-3) {$\scriptstyle x^3y^2$};
\node at (2,-4) {$\scriptstyle x^4y^2$};
\node at (2,-5) {$\scriptstyle x^5y^2$};
\node at (2,-6) {$\scriptstyle x^6y^2$};
\node at (3,0) {$\scriptstyle y^3$};
\node at (3,-1) {$\scriptstyle xy^3$};
\node at (3,-2) {$\scriptstyle x^2y^3$};
\node at (3,-3) {$\scriptstyle x^3y^3$};
\node at (4,0) {$\scriptstyle y^4$};
\node at (4,-1) {$\scriptstyle xy^4$};
\node at (4,-2) {$\scriptstyle x^2y^4$};
\node at (4,-3) {$\scriptstyle x^3y^4$};
\node at (5,0) {$\scriptstyle y^5$};
\node at (6,0) {$\scriptstyle \hdots$};
\node at (7,0) {$\scriptstyle y^{16}$};
\draw (-0.5,0.5) -- (7.5,0.5) -- (7.5,-0.5) -- (4.5,-0.5) -- (4.5,-3.5) -- (2.5,-3.5)--(2.5,-6.5)--(0.5,-6.5)--(0.5,-9.5)--(-0.5,-9.5)--cycle;
\end{tikzpicture}
\]}
\end{example}

For $G=\frac{1}{r}(1,a)$ with $(r,a)=1$, recall that the \emph{$L$-space} $L(G)$ is defined to be 
\[
L(G)\colonequals \{1,x,\hdots,x^{r-1},y,\hdots,y^{r-1}\},
\]
so called since in the $2\times 2$ grid the shape of $L(G)$ looks like the letter L.

\begin{thm}[{=\cite[3.7]{Ito}}]\label{Ito thm}
When $(r,a)=1$, the elements of $I(r,a)$ are precisely those numbers in $[0,r]$ that do not appear as weights of monomials in the region $B(G)\backslash L(G)$.
\end{thm}

\begin{example}{\rm
Consider $G=\frac{1}{17}(1,10)$.  Then $B(G)\backslash L(G)$ is the region 
\[
\begin{tikzpicture}[xscale=0.8,yscale=0.4]
\node at (0,0) {$\scriptstyle 1$};
\node at (0,-1) {$\scriptstyle x$};
\node at (0,-2) {$\scriptstyle x^2$};
\node at (0,-3) {$\scriptstyle x^3$};
\node at (0,-4) {$\scriptstyle x^4$};
\node at (0,-5) {$\scriptstyle x^5$};
\node at (0,-6) {$\scriptstyle x^6$};
\node at (0,-7) {$\scriptstyle x^7$};
\node at (0,-8) {$\scriptstyle \vdots$};
\node at (0,-9) {$\scriptstyle x^{16}$};
\node at (1,0) {$\scriptstyle y^{\phantom 2}$};
\node at (1,-1) {$\scriptstyle xy^{\phantom 2}$};
\node at (1,-2) {$\scriptstyle x^2y$};
\node at (1,-3) {$\scriptstyle x^3y$};
\node at (1,-4) {$\scriptstyle x^4y$};
\node at (1,-5) {$\scriptstyle x^5y$};
\node at (1,-6) {$\scriptstyle x^6y$};
\node at (2,0) {$\scriptstyle y^2$};
\node at (2,-1) {$\scriptstyle xy^2$};
\node at (2,-2) {$\scriptstyle x^2y^2$};
\node at (2,-3) {$\scriptstyle x^3y^2$};
\node at (2,-4) {$\scriptstyle x^4y^2$};
\node at (2,-5) {$\scriptstyle x^5y^2$};
\node at (2,-6) {$\scriptstyle x^6y^2$};
\node at (3,0) {$\scriptstyle y^3$};
\node at (3,-1) {$\scriptstyle xy^3$};
\node at (3,-2) {$\scriptstyle x^2y^3$};
\node at (3,-3) {$\scriptstyle x^3y^3$};
\node at (4,0) {$\scriptstyle y^4$};
\node at (4,-1) {$\scriptstyle xy^4$};
\node at (4,-2) {$\scriptstyle x^2y^4$};
\node at (4,-3) {$\scriptstyle x^3y^4$};
\node at (5,0) {$\scriptstyle y^5$};
\node at (6,0) {$\scriptstyle \hdots$};
\node at (7,0) {$\scriptstyle y^{16}$};
\draw (0.5,-0.5) -- (4.5,-0.5)  -- (4.5,-3.5) -- (2.5,-3.5)--(2.5,-6.5)--(0.5,-6.5)--cycle;
\end{tikzpicture}
\]
Replacing the monomials in the above region by their corresponding weights gives
\[
\begin{tikzpicture}[xscale=0.8,yscale=0.4]
\node at (1,-1) {$\scriptstyle 11$};
\node at (1,-2) {$\scriptstyle 12$};
\node at (1,-3) {$\scriptstyle 13$};
\node at (1,-4) {$\scriptstyle 14$};
\node at (1,-5) {$\scriptstyle 15$};
\node at (1,-6) {$\scriptstyle 16$};
\node at (2,-1) {$\scriptstyle 4$};
\node at (2,-2) {$\scriptstyle 5$};
\node at (2,-3) {$\scriptstyle 6$};
\node at (2,-4) {$\scriptstyle 7$};
\node at (2,-5) {$\scriptstyle 8$};
\node at (2,-6) {$\scriptstyle 9$};
\node at (3,-1) {$\scriptstyle 14$};
\node at (3,-2) {$\scriptstyle 15$};
\node at (3,-3) {$\scriptstyle 16$};
\node at (4,-1) {$\scriptstyle 7$};
\node at (4,-2) {$\scriptstyle 8$};
\node at (4,-3) {$\scriptstyle 9$};
\draw (0.5,-0.5) -- (4.5,-0.5)  -- (4.5,-3.5) -- (2.5,-3.5)--(2.5,-6.5)--(0.5,-6.5)--cycle;
\end{tikzpicture}
\]
and so by \ref{Ito thm}, the $i$-series consists of those numbers that do not appear in the above region, which are precisely the numbers $0$, $1$, $2$, $3$, ${10}$ and $17$.   Indeed, in this example $\frac{17}{10}=[2,4,2,2]$ and the $i$-series is
\[
i_0=17>i_1=10>i_2=3>i_3=2>i_4=1>i_5=0.
\]}
\end{example}

The following lemma, which we use later, is an extension of \ref{Ito thm}. For integers $r>0$ and $k$, write $[k]_r$ for the unique integer $k'$ satisfying $0\le k'\le r-1$ and $k-k'\in r\bZ$.

\begin{lemma}\label{i-series region}
Assume $(r,a)=1$. For $0\le u\le r-1$, the following are equivalent.
\begin{enumerate}
\item $u\in I(r,r-a)$.
\item $u$ does not appear in $B(G)\backslash L(G)$ for $G\colonequals \frac{1}{r}(1,-a)$.
\item For every integer $\ell\ge1$, there exists an integer $m\in [1,\ell]$ such that $[u+\ell a-1]_r\ge[ma-1]_r$.
\end{enumerate}
\end{lemma}

\begin{proof}
(1)$\Leftrightarrow$(2) This is \ref{Ito thm}.\\
(2)$\Leftrightarrow$(3) We will establish the following claim: $u$ does not appear in column $\ell$ of $B(G)\backslash L(G)$ if and only if there exists an integer $m$ satisfying $1\le m\le\ell$ and $[u+\ell a-1]_r\ge[ma-1]_r$.

The first row of $B(G)$ is $0,-a,-2a,-3a,\ldots$. Now for each $m$ with $1\le m\le \ell$, we find the first occurrence of weight $0$ in column $m$, and use this to draw the following diagram:
\[
\begin{tikzpicture}[xscale=0.8,yscale=0.4]
\node at (1,-1) {$\scriptstyle -ma$};
\node at (1,-2) {$\scriptstyle 1-ma$};
\node at (1,-3) {$\scriptstyle 2-ma$};
\node at (1,-4) {$\scriptstyle \vdots$};
\node at (1,-5) {$\scriptstyle -2$};
\node at (1,-6) {$\scriptstyle -1$};
\node at (1,-7) {$\scriptstyle 0$};
\node at (2,-1) {$\scriptstyle\cdots$};
\node at (2,-2) {$\scriptstyle\cdots$};
\node at (2,-3) {$\scriptstyle\cdots$};
\node at (2,-4) {$\scriptstyle\ddots$};
\node at (2,-5) {$\scriptstyle\cdots$};
\node at (2,-6) {$\scriptstyle\cdots$};
\node at (2,-7) {$\scriptstyle\cdots$};
\node at (3.5,-1) {$\scriptstyle -\ell a$};
\node at (3.5,-2) {$\scriptstyle 1-\ell a$};
\node at (3.5,-3) {$\scriptstyle 2-\ell a$};
\node at (3.5,-4) {$\scriptstyle \vdots$};
\node at (3.5,-5) {$\scriptstyle -2+(m-\ell)a$};
\node at (3.5,-6) {$\scriptstyle -1+(m-\ell)a$};
\node at (3.5,-7) {$\scriptstyle (m-\ell)a$};
\draw (0.25,-1.5) -- (4.75,-1.5)  -- (4.75,-6.5) --(0.25,-6.5)--cycle;
\draw[densely dotted] (2.5,-1.6) -- (4.65,-1.6)  -- (4.65,-6.4) --(2.5,-6.4)--cycle;
\end{tikzpicture}
\]
The column $\ell$ of $B(G)\backslash L(G)$ is the intersection, over all $m$ with $1\le m\le \ell$,  of the above dotted regions.
It is clear that $u$ does not appear in the dotted region in the above diagram if and only if $[u+\ell a-1]_r\geq[ma-1]_r$.
The claim follows.
\end{proof}

\begin{notation}\label{notationremark}{\rm
Throughout the remainder of the paper, to aid readability we will use the following simplified notation.
\begin{equation*}
\begin{tabular}{*3c}
\toprule
{\bf Notation}&{\bf Meaning}&{\bf Simplified Notation}\\
\midrule
$\mathbb{X}_{\bp,\bl}$&Weighted projective line & $\mathbb{X}$\\
$S_{\bp,\bl}$ & Defining ring of $\mathbb{X}_{\bp,\bl}$& $S$\\
$\Lambda_{\bp,\bl}$&Canonical algebra&$\Lambda$\\
$S_{\bp,\bl}^{\ox}$ & Veronese of $S_{\bp,\bl}$ with respect to $\ox\in\bL$& $S^{\ox}$\\
${Y}^{\ox}_{\bp,\bl}$&Resolution of $\Spec S^{\ox}_{\bp,\bl}$ in \eqref{stack diagram 1} & $Y^{\ox}$\\
\bottomrule\\
\end{tabular}
\end{equation*}
Throughout it will be implicit that we are working generally, with parameters $(\bp,\bl)$.}
\end{notation}

\section{The Total Space $\bT$}\label{stack}

\subsection{Definition and First Properties}
With notation as in \ref{notationremark}, let $\ox\in\bL$ and consider the Veronese subring $S^{\ox}$, and the total space stack defined by
\[
\bT^{\ox}\colonequals \bTot(\cO_{\bX}(-\ox))\colonequals [(\Spec S\backslash 0\times\Spec\KK[t])/\Spec \KK\bL],
\]
where $\bL$ acts on $t$ with weight $-\ox$.  There is a natural projection $q\colon\bT^{\ox}\to\bX$, and a natural map $g\colon\bT^{\ox}\to T^{\ox}$ to its coarse moduli space.

We remark that $T^{\ox}$ has a natural open cover.  Indeed, the open covering $\Spec S\backslash0=U_0\cup U_1$ with $U_i\colonequals \Spec S_{t_i}$ induces an open cover 
\[
(\Spec S\backslash0)\times\Spec \KK[t]=U'_0\cup U'_1\quad \mbox{with}\quad U'_i\colonequals U_i\times\Spec \KK[t]=\Spec S_{t_i}[t],
\]
which in turn implies that  $T^\ox$ has an open cover
\begin{equation}\label{T=V0 cup V1}
T^{\ox}=V_0\cup V_1\ \mbox{ with }\ V_i\colonequals \Spec (S_{t_i}[t])_0,
\end{equation}
where $(S_{t_i}[t])_0$ is the degree zero part of the $\bL$-graded ring $S_{t_i}[t]$ with $\deg t=-\ox$.
As in \ref{notation throughout}\eqref{notation throughout 2b}, the curve $X_i\colonequals \Spec (S_{t_i})_0$ in $V_i$ gives the coarse moduli $X=X_0\cup X_1\cong\mathbb{P}^1$ of $\bX$.

We first investigate the singularities of $T^{\ox}$. 

\begin{prop}\label{sings on T2}
If $\ox\in\bL$, then $T^{\ox}$ is a surface containing the coarse moduli $X\cong \bP^1$ of $\bX$.
Moreover $T^{\ox}$ is normal, and all its singularities are isolated and lie on $X$.
\end{prop}

\begin{proof}
Fix $i=0,1$ and let $A\colonequals S_{t_i}[t]$ and $B\colonequals (S_{t_i}[t])_0$ so that $V_i=\Spec B$. 

Since $A$ is an $\bL$-factorial $\bL$-domain by \ref{L-factorial}\eqref{L-factorial 1}\eqref{L-factorial 5}, its degree zero part $B$ is a normal domain by \ref{L-factorial}\eqref{L-factorial 3}.
Now we claim $\dim B=2$. Note that $S$ is a finitely generated $\bL$-graded module over the $\bZ\oc$-graded subring $C\colonequals \KK[t_0,t_1]$.
Thus $A$ is a finitely generated $\bL$-graded $C_{t_i}[t]$-module, and similarly $B$ is a finitely generated $(C_{t_i}[t])_0$-module.
Let $p$ be the smallest positive integer satisfying $p\ox\in\bZ\oc$, and $p\ox=q\oc$ for $q\in\bZ$. Then $(C_{t_i}[t])_0=(C_{t_i}[t^p])_0$ is the polynomial ring with two variables $t_{1-i}/t_i$ and $t_i^qt^p$. Thus $\dim B=\dim (C_{t_i}[t])_0=2$.

Consider next the $\bZ$-grading on $A=S_{t_i}[t]$ defined by $\deg t=1$ and $\deg x=0$ for any $x\in S_{t_i}$. This gives a $\bZ$-grading on $B$ such that $B=\bigoplus_{j\ge0}B_j$ and $B_0=(S_{t_i})_0$.
Since $B$ is a $\bZ$-graded finitely generated $\KK$-algebra, by the Jacobian criterion, there is a $\bZ$-graded ideal $I$ of $B$ such that $\Sing B=\Spec(B/I)$. 
Since $B$ is normal, all the singularities of $B$ are isolated and $\dim_\KK(B/I)<\infty$ holds. Since $I$ is $\bZ$-graded, it contains $\bigoplus_{j>\ell}B_j$ for $\ell\gg0$ and hence $\sqrt{I}$ contains $\bigoplus_{j>0}B_j$. Consequently, $\Sing B$ is contained in $\Spec B_0=X_i\subset X$.
\end{proof}

\begin{prop}\label{sings on T}
Suppose that $\ox\in\bL$ and write $\ox$ in normal form as $\ox=\sum_{i=1}^na_i\ox_i+a\oc$ for some $0\leq a_i<p_i$ and $a\in\mathbb{Z}$.  Then on $X\cong\bP^1$, complete locally the singularities of $T^{\ox}$ are of the form 
\begin{equation}\label{T picture 2}
\begin{array}{c}
\begin{tikzpicture}[scale=1.5] 
\draw (-0.2,0) to [bend left=5] node[pos=0.75,above] {$\scriptstyle \bP^1$} (4.2,0);
\filldraw [black] (0.0,0.01) circle (1pt);
\filldraw [black] (1,0.08) circle (1pt);
\filldraw [black] (2,0.11) circle (1pt);
\filldraw [black] (4,0.01) circle (1pt);
\node at (0,-0.15) {$\scriptstyle \frac{1}{p_1}(1,-a_1)$};
\node at (0,0.2) {$\scriptstyle \lambda_1$};
\node at (1,-0.1) {$\scriptstyle \frac{1}{p_2}(1,-a_2)$};
\node at (1,0.275) {$\scriptstyle \lambda_2$};
\node at (2,-0.075) {$\scriptstyle \frac{1}{p_3}(1,-a_3)$};
\node at (2,0.3) {$\scriptstyle \lambda_3$};
\node at (3,-0.1) {$\scriptstyle \hdots$};
\node at (4,-0.15) {$\scriptstyle \frac{1}{p_n}(1,-a_n)$};
\node at (4,0.2) {$\scriptstyle \lambda_n$};
\end{tikzpicture} 
\end{array}
\end{equation}
\end{prop}
\begin{proof}
We will show that $\widehat{\cO}_{T^{\ox},\lambda_i}$ is the completion of $\frac{1}{p_i}(1,-a_i)$. By symmetry, we only have to consider the case $i=1$.
We use the presentation of $S$ given in \ref{notation throughout}(2)
\[
S=\frac{\KK[x_1,\hdots,x_n]}{(x_i^{p_i}+x_1^{p_1}-\lambda_ix_2^{p_2}\mid 3\leq i\leq n)}.
\]
and the open cover $T^{\ox}=V_0\cup V_1$ given in \eqref{T=V0 cup V1}, where $t_0=x_2^{p_2}$ and $t_1=x_1^{p_1}$. Thus $\lambda_1=(1:0)$ belongs to $V_0=\Spec B$ for $B\colonequals(S_{t_0}[t])_0=(S_{x_2}[t])_0$.

Let $\mathfrak{m}$ be the maximal ideal of $B$ corresponding to $\lambda_1$. We shall show that $\widehat{\cO}_{T^{\ox},\lambda_1}=\widehat{B}_{\mathfrak{m}}$ is the completion of $\frac{1}{p_1}(1,-a_1)$.
For the polynomial ring $\KK[x_1,\ldots,x_n,t]$ and the formal power series ring $\KK[[\sx_1,\st]]$, consider the morphism
\[
f\colon\KK[x_1,\ldots,x_n,t]\to P\colonequals \KK[[\sx_1,\st]]
\]
of $\KK$-algebras defined by $f(t)=\st$, $f(x_1)=\sx_1$, $f(x_2)=1$ and $f(x_i)=(\lambda_i-\sx_1^{p_1})^{1/p_i}$ for $3\le i\le n$, where a $p_i$-th root of $\lambda_i-\sx_1^{p_1}$ exists since $\KK$ is an algebraically closed field of characteristic zero.
Since $f$ sends $x_i^{p_i}+x_1^{p_1}-\lambda_i x_2^{p_2}$ to zero for all $3\le i\le n$, it induces a morphism of $\KK$-algebras
 $f\colon S[t]\to P$, and further since $f(t_0)=f(x_2^{p_2})=1$ this induces a morphism of $\KK$-algebras
\begin{equation}\label{invert x_2}
f\colon S_{t_0}[t]\to P.
\end{equation}
Let $C\colonequals \frac{1}{p_1}(1,-a_1)=\langle\zeta\rangle$ be the cyclic group acting on $P$ by $\zeta\sx_1=\ve\sx_1$ and $\zeta\st=\ve^{-a_1}\st$ for a primitive $p_1$-th root $\ve$ of unity. Certainly $f(x_i)$ with $2\le i\le n$ belongs to $\KK[[\sx_1^{p_1}]]\subset P^C$. Now we claim that \eqref{invert x_2} induces a morphism of $\KK$-algebras
\begin{equation}\label{then degree zero}
f\colon B=(S_{t_0}[t])_0\to P^C.
\end{equation}
If a monomial $X=x_1^{\ell_1}\ldots x_n^{\ell_n}t^\ell\in S_{x_2}[t]=S_{t_0}[t]$ has degree zero, then $\ell_1\ox_1+\ldots+\ell_n\ox_n+\ell\ox=0$ holds. Looking at the coefficient of $\ox_1$, we have $\ell_1+\ell a_1\in p_1\bZ$. Thus $f(X)=\sx_1^{\ell_1}\st^\ell\prod_{i=2}^nf(x_i)^{\ell_i}$ belongs to $P^C$, and the assertion follows.

The closed subscheme $X_0=\Spec(S_{t_0})_0$ of $V_0=\Spec B$ is defined by the ideal $(tS_{t_0}[t])_0$ of $B$, and the closed point $\lambda_1$ of $X_0$ is defined by the ideal $(t_1/t_0)$ of $(S_{t_0})_0=\KK[t_1/t_0]$.
Therefore the maximal ideal $\mathfrak{m}$ of $B$ is generated by monomials $x_1^{\ell_1}\cdots x_n^{\ell_n}t^\ell$ with $\ell\ge1$ and $t_1/t_0=x_1^{p_1}/x_2^{p_2}$.
In particular, $f(\mathfrak{m})$ is contained in the maximal ideal $\mathfrak{n}$ of $P^C$, and hence \eqref{then degree zero} induces a morphism
\begin{equation}\label{then complete}
f\colon \widehat{B}_{\mathfrak{m}}\to P^C.
\end{equation}
We show that this is an isomorphism. Since $B$ is a normal domain which is finitely generated over a field, $\widehat{B}_{\mathfrak{m}}$ is also a normal domain by Zariski's Main Theorem \cite[VIII.13 Theorem 32]{ZS}.
Since $\dim\widehat{B}_{\mathfrak{m}}=2=\dim P^C$, it suffices to show that \eqref{then complete} is surjective, or equivalently, \eqref{then degree zero} gives a surjective map $\mathfrak{m}\to\mathfrak{n}/\mathfrak{n}^2$.
Since the $\KK$-vector space $\mathfrak{n}/\mathfrak{n}^2$ is spanned by monomials in $P^C$, it suffices to show that any monomial $\sx_1^{\ell_1}\st^\ell$ in $P^C$ belongs to $\Im f+\mathfrak{n}^2$.
Since $\sx_1^{\ell_1}\st^\ell$ is invariant under the action of $C$, the coefficient of $\ox_1$ in the normal form of $\ell_1\ox_1+\ell\ox$ is zero.
Thus there exist $\ell_2\in\bZ$ and $\ell_3,\ldots,\ell_n\in\bZ_{\ge0}$ such that $\ell_1\ox_1+\ldots+\ell_n\ox_n+\ell\ox=0$.
Now $X\colonequals x_1^{\ell_1}\cdots x_n^{\ell_n}t^\ell\in B$ satisfies
\[
f(\upalpha X)\equiv\sx_1^{\ell_1}\st^\ell\ \mod\mathfrak{n}^2\ \mbox{ for }\ \upalpha\colonequals \prod_{i=3}^n\lambda_i^{-\ell_i/p_i}.
\]
Hence \eqref{then complete} is an isomorphism.
\end{proof}

The following calculation will be one of our main technical tools.

\begin{prop}\label{the Cech argument}
For any $\ox\in\bL$,
\[
H^i(\cO_{T^{\ox}})=\left\{\begin{array}{ll}
\bigoplus_{j\ge0}S_{j\ox}&i=0,\\
\bigoplus_{j\ge0}(S_{\ow-j\ox})^*&i=1,\\
0&i\ge2.\end{array}\right.
\]
Therefore there is a canonical morphism $\upgamma\colon T^{\ox}\to\Spec S^{\bN\ox}$.
\end{prop}
\begin{proof}
We calculate $H^i(\cO_{T^{\ox}})$ as the \v{C}ech cohomology with respect to the open affine cover $T^{\ox}=V_0\cup V_1$ in \eqref{T=V0 cup V1}. Since $H^0(V_i,\cO_{T^{\ox}})=(S_{t_i}[t])_0$ for $i=0,1$ and $H^0(V_0\cap V_1,\cO_{T^{\ox}})=(S_{t_0t_1}[t])_0$, the complex
\begin{equation}\label{cech1}
0\to(S_{t_0}[t])_0\oplus(S_{t_1}[t])_0\to (S_{t_0t_1}[t])_0\to0
\end{equation}
has cohomology $H^i(\cO_{T^{\ox}})$ with $i\ge0$. Thus $H^i(\cO_{T^{\ox}})=0$ for any $i\ge2$.

Now let ${\mathfrak a}=St_0+St_1$. Then the local cohomologies $H^i_{\mathfrak a}(S)$ of $S$ are the cohomologies of the extended \v{C}ech complex
\begin{equation}\label{cech2}
0\to S\to S_{t_0}\oplus S_{t_1}\to S_{t_0t_1}\to0
\end{equation}
by \cite[Theorem 5.1.20]{BS}.
Since $t_0,t_1$ is an $S$-sequence, we have $H^0_{\mathfrak a}(S)=H^1_{\mathfrak a}(S)=0$ by \cite[Theorem 6.2.7]{BS}. Thus \eqref{cech2} gives an exact sequence
\[
0\to (S[t])_0\to (S_{t_0}[t])_0\oplus(S_{t_1}[t])_0\to(S_{t_0t_1}[t])_0\to(H^2_{\mathfrak a}(S)\otimes_\KK\KK[t])_0\to0.
\]
Comparing with \eqref{cech1} gives isomorphisms
\[
H^0(\cO_{T^{\ox}})\simeq(S[t])_0=\bigoplus_{j\ge0}S_{j\ox}\ \mbox{ and }\ H^1(\cO_{T^{\ox}})\simeq(H^2_{\mathfrak a}(S)\otimes_\KK\KK[t])_0=\bigoplus_{j\ge0}H^2_{\mathfrak a}(S)_{j\ox}.
\]
Since $\sqrt{\mathfrak a}$ is the $\bL$-graded maximal ideal $\m$ of $S$, we have $H^i_{\mathfrak a}(S)=H^i_{\m}(S)$.
Furthermore, $S(\ow)$ being an $\bL$-graded canonical module of $S$, the $\bL$-graded local duality theorem \cite[Theorem 14.4.1]{BS} gives the required isomorphism
\[
H^2_{\mathfrak a}(S)_{j\ox}=H^2_{\m}(S)_{j\ox}\simeq(\Hom_S(S,S(\ow))_{-j\ox})^*=(S_{\ow-j\ox})^*.
\]
The last statement follows from \cite[Exercise II.2.4]{Hartshorne}, since $S^{\bN\ox}\colonequals \bigoplus_{j\ge0}S_{j\ox}$.
\end{proof}

There is also a map $f\colon\bX\to X$ from $\bX$ to its coarse moduli space, and $p: T^{\ox}\to X$ an obvious morphism, which together with the above form a commutative diagram
\begin{equation}
\begin{array}{c}
\begin{tikzpicture}
\node (top 1) at (0,0) {$\bT^{\ox}$};
\node (top 2) at (2.5,0) {$\bX$};
\node (bottom 1) at (0,-1.5) {$T^{\ox}$};
\node (bottom 2) at (2.5,-1.5) {$X\cong\mathbb{P}^1$};
\draw[->] (top 1) -- node[left] {$\scriptstyle g$} (bottom 1);
\draw[->] (top 2) -- node[right] {$\scriptstyle f$} (bottom 2);
\draw[->] (top 1) -- node[above] {$\scriptstyle q$} (top 2);
\draw[->] (bottom 1) -- node[above] {$\scriptstyle p$} (bottom 2);
\end{tikzpicture}
\end{array}\label{need comm for proj}
\end{equation}

We now try to contract the $\mathbb{P}^1$ in \eqref{T picture 2} by taking global sections.  As is usual, to do this requires some form of negativity for $\bT^{\ox}=\bTot(\cO_{\bX}(-\ox))$;  in the language here, this translates into some form of positivity for $\ox$.  This is slightly technical to state, and we will require the following lemma.  Recall that there is a group homomorphism
\[
\updelta\colon\bL\to\mathbb{Q}
\]
sending $\oc\mapsto 1$ and $\ox_i\mapsto \frac{1}{p_i}$.  It is elementary that $\updelta(\bL_+)\subset{\mathbb Q}_{\ge0}$, and $\ox$ is torsion if and only if $\updelta(\ox)=0$. 
Also, using normal form, it is clear that $\bL\setminus\bL_+$ has the maximum element $\sum_{i=1}^n(p_i-1)\ox_i-\oc=\ow+\oc$. In particular, $\ow\notin\bL_+$. 

\begin{lemma}\label{positivity combinatorial lemma} 
If  $\ox\in\bL$, then the following hold.
\begin{enumerate}
\item\label{positivity combinatorial lemma 1}  $-i\ox\notin\bL_+$ for all $i>0$ $\iff$ $\updelta(\ox)> 0$.
\item\label{positivity combinatorial lemma 2}  $\ow-i\ox\notin\bL_+$ for all $i\geq 0$ $\Rightarrow$ $\updelta(\ox)\geq 0$.
\end{enumerate}
\end{lemma}
\begin{proof}
(1)($\Leftarrow$) If $i>0$, then $\updelta(-i\ox)=-i\updelta(\ox)<0$. The result follows since $\updelta(\bL_+)\subset{\mathbb Q}_{\ge0}$.\\
(1)(2)($\Rightarrow$) Since $\ow+\oc$ is the maximum element of $\bL\setminus\bL_+$, any $\oz\in\bL$ satisfying 
\begin{equation}
\updelta(\oz)>\updelta(\ow+\oc) \label{ineq to plus}
\end{equation} 
belongs to $\bL_+$. If $\updelta(\ox)<0$, then
\[\updelta(-j\ox)=-j\updelta(\ox)\to+\infty\ \mbox{ and }\ \updelta(\ow-j\ox)=\updelta(\ow)-j\updelta(\ox)\to+\infty\]
as $j\to\infty$. Hence for sufficiently large $j$, both $\updelta(-j\ox)$ and $\updelta(\ow-j\ox)$ are larger than $\updelta(\ow+\oc)$. Thus, by \eqref{ineq to plus}, both $-j\ox$ and $\ow-j\ox$ belong to $\bL_+$, a contradiction.  Hence $\updelta(\ox)\ge0$. Thus (2) follows. To complete the proof of (1)($\Rightarrow$), notice that the assumption implies that $\ox$ is not torsion, as in the proof of \ref{Two dim and normal prelim}\eqref{L-factorial 7}, so $\updelta(\ox)\neq0$.~
\end{proof}

This leads to our key new definition.
\begin{defin}
We define the \emph{geometrically positive} elements of $\bL$ to be
\[
\Pos(\bL)\colonequals \{ \ox\in\bL\mid \ox\mbox{ is not torsion, and } \ow-j\ox\notin \bL_+ \mbox{ for all }j\geq 0\}.
\]
\end{defin}

Given any $\ox\in\bL$, recall from \ref{the Cech argument} that $H^0(\cO_{T^{\ox}})=S^{\bN\ox}$ holds, giving rise to a canonical morphism $\upgamma\colon T^{\ox}\to\Spec S^{\bN\ox}$.

\begin{prop}\label{last comb hope}
Suppose that $\ox\in\bL$.
\begin{enumerate}
\item\label{last comb hope 1} If $0\neq\ox\in\bL_+$, then $\ox\in\Pos(\bL)$.
\item\label{last comb hope 2} The following conditions are equivalent.
\begin{enumerate}
\item $\ox\in\Pos(\bL)$.
\item $-i\ox\notin\bL_+$ for all $i>0$, and $\ow-j\ox\notin \bL_+$ for all $j\geq 0$.
\item $S^{\bN\ox}=S^{\ox}$ and $\mathbf{R}^{\kern -1pt t} \upgamma_*\cO_{T^{\ox}}=0$ for all $t>0$.
\end{enumerate}
\end{enumerate}
\end{prop}
\begin{proof}
(1) Clearly $\ox$ is not torsion. Since $\ow\notin\bL_+$ and $\ow-j\ox\le\ow$ for $j\geq0$, we have $\ow-j\ox\notin\bL_+$. Thus $\ox\in\Pos(\bL)$.\\
(2)(a)$\Leftrightarrow$(b). The condition $\ow-j\ox\notin \bL_+$ for all $j\geq 0$ is common to both.  Thus we just need to prove, assuming this condition, that $\ox$ is not torsion (equivalently, $\updelta(\ox)\neq 0$) if and only if $-i\ox\notin\bL_+$ for all $i>0$.  But this follows from \ref{positivity combinatorial lemma}.\\
(b)$\Leftrightarrow$(c) Follows from \ref{the Cech argument}.
\end{proof}

\begin{cor}\label{we have a map!}
If $\ox\in\Pos(\bL)$, then there is a canonical morphism
\[
\upgamma\colon T^{\ox}\to\Spec S^{\ox}
\]
such that $\RDerived \upgamma_*\!\cO_{T^{\ox}}=\cO_{S^{\ox}}$. 
\end{cor}

\subsection{The morphism $\upgamma$} In this subsection we show, under the assumption in \ref{we have a map!}, that  $\upgamma$ is a projective birational morphism.  This then implies that $T^{\ox}$ is a partial resolution of singularities of $\Spec S^{\ox}$, which indeed is our motivation for studying the stack $\bT^{\ox}$ and its coarse moduli $T^{\ox}$.

\begin{lemma}\label{rel ample}
Suppose that $\ox\in\Pos(\bL)$. Then
\begin{enumerate}
\item\label{rel ample 1} $\upgamma$ is a finite type morphism between noetherian schemes.
\item\label{rel ample 2} $\cL\colonequals p^*\cO(1)$ is an ample bundle on $T^{\ox}$.
\item\label{rel ample 3} $\cL$ is $\upgamma$-relatively ample.
\end{enumerate}
\end{lemma}
\begin{proof}
(1) $T^{\ox}$ is noetherian since it is covered by a finite number of affine charts (namely two) in \eqref{T=V0 cup V1}, each given by a noetherian ring.  Further $\Spec S^{\ox}$ is noetherian since $S^{\ox}$ is by \ref{Two dim and normal prelim}.  Now the morphism $\upgamma$ is quasi-compact since $T^{\ox}$ is noetherian and thus quasi-compact, and $\Spec S^{\ox}$ is affine.    Further, composing $\upgamma$ with the structure morphisms $s\colon \Spec S^{\ox}\to \Spec \KK$ gives a morphism $s\circ \upgamma$ of finite type, since  $T^{\ox}$ is covered by finitely generated $\KK$-algebras.  By the left cancelation property \cite[II.Ex.3.13(f)]{Hartshorne}, $\upgamma$ also has finite type.\\
(2) It is well-known that $\cO(1)$ is ample on $\bP^1$, or equivalently, relatively ample with respect to the structure morphism $\bP^1\to\Spec \KK$.    Since $p$ is affine, pulling back yields a bundle $p^*\cO(1)$ which is relatively ample with respect to the composition $T^{\ox}\to\bP^1\to\Spec \KK$ \cite[II.5.1.12]{EGA}.  But this is just the structure morphism for $T^{\ox}$, hence it follows that $p^*\cO(1)$ is ample on $T^{\ox}$. \\
(3) This follows immediately from (2), given $\Spec S^{\ox}$ is affine.
\end{proof}

As notation, for $\oy\in\bL$ write $S(\oy)^{\ox}\colonequals \bigoplus_{i\in\bZ}S_{\oy+i\ox}\supset S(\oy)^{\bN\ox}\colonequals \bigoplus_{i\geq 0}S_{\oy+i\ox}$.
\begin{lemma}\label{pushforward L}
Suppose that $\ox\in\Pos(\bL)$. Then
\begin{enumerate}
\item\label{pushforward L 0} For all $\oy\in\bL$, $\upgamma_*g_*q^*\cO_{\bX}(\oy)=S(\oy)^{\bN\ox}$.
\item\label{pushforward L 1} $\upgamma_*\cL=S(\oc)^{\bN\ox}$  holds, and this is a finitely generated $S^{\ox}$-module.
\item\label{pushforward L 2} $\upgamma_*\cL^{-n}$ and $\mathbf{R}^1\upgamma_*\cL^{-n}$ are finitely generated $S^{\ox}$-modules for all $n\geq 0$. 
\end{enumerate}
\end{lemma}
\begin{proof}
(1) Note first that
\[
\upgamma_*g_*q^*\cO_{\bX}(\oy)
=
\mathrm{H}^0(\bT,q^*\cO_{\bX}(\oy))
=
\mathrm{H}^0(\bX,q_*q^*\cO_{\bX}(\oy)).
\]
By the projection formula $q_*q^*(\cO_{\bX}(\oy))=\bigoplus_{i\geq 0}\cO_{\bX}(i\ox+\oy)$, and so the above equals  
\[
\bigoplus_{i\geq 0}\mathrm{H}^0(\bX,\cO_{\bX}(i\ox+\oy))
=
\bigoplus_{i\geq 0} S_{i\ox+\oy}
=
S(\oy)^{\bN\ox}.
\]
(2) Note that $g_*g^*\cL=\cL$ by the projection formula, and so
\[
\upgamma_*\cL=\upgamma_*g_*g^*\cL
=\upgamma_*g_*g^*p^*\cO_{\bP^1}(1)
=\upgamma_*g_*q^*f^*\cO_{\bP^1}(1)
=\upgamma_*g_*q^*\cO_{\bX}(\oc).
\]
Hence $\upgamma_*\cL=S(\oc)^{\bN\ox}$ by \eqref{pushforward L 0}. Now by \ref{Two dim and normal prelim}\eqref{L-factorial 2}, $S(\oc)^{\ox}$ is a finitely generated $S^{\ox}$-module, hence its submodule $\upgamma_*\cL$ is also a finitely generated $S^{\ox}$-module, since $S^{\ox}$ is noetherian.\\
(3) We know by \ref{we have a map!} that the result is true for $n=0$ since $\RDerived \upgamma_*\!\cO_{T^{\ox}}=\cO_{S^{\ox}}$.   Part \eqref{pushforward L 1} shows that $f_*\cL$ is finitely generated.  Pulling up the Euler exact sequence from $\bP^1$ gives an exact sequence
\begin{equation}
0\to\cL^{-1}\to\cO^{\oplus 2}\to \cL\to 0\label{Euler upstairs}
\end{equation}
on $T^{\ox}$, and pushing down gives an exact sequence
\[
0\to \upgamma_*\cL^{-1}\to (S^{\ox})^{\oplus 2}\to \upgamma_*\cL\to 
\mathbf{R}^1\upgamma_*\cL^{-1}\to 0.
\]
Since $S^{\ox}$ is noetherian, and the middle two objects are finitely generated, necessarily the outer objects are also finitely generated. Hence the result is true for $n=1$.

By induction, we can thus assume that the result is true for $n-1$ and $n-2$.  Twisting the sequence \eqref{Euler upstairs} appropriately, then pushing down, gives an exact sequence
\[
0\to \upgamma_*\cL^{-n}\to (\upgamma_*\cL^{-n+1})^{\oplus 2}\to \upgamma_*\cL^{-n+2}\to 
\mathbf{R}^1\upgamma_*\cL^{-n}\to (\mathbf{R}^1\upgamma_*\cL^{-n+1})^{\oplus 2} \to \mathbf{R}^1\upgamma_*\cL^{-n+2}\to 0.
\]
By induction the second, third, fifth and sixth objects are finitely generated.  Hence so too are the first and fourth.  By induction the result follows.
\end{proof}

\begin{thm}\label{is projective}
Suppose that $\ox\in\Pos(\bL)$. Then $\upgamma\colon T^{\ox}\to\Spec S^{\ox}$ is a projective birational morphism, satisfying $\RDerived \upgamma_*\!\cO_{T^{\ox}}=\cO_{S^{\ox}}$.
\end{thm}
\begin{proof}
We first claim that $\upgamma$ is proper.  Since by \ref{rel ample}\eqref{rel ample 1} $\upgamma$ is a finite type morphism between noetherian schemes, by \cite{Rydh} it suffices to show that both $\upgamma_*$ and $\mathbf{R}^1\upgamma_*$ preserve coherent sheaves. Pick $\cF\in\coh T^{\ox}$.

Since by \ref{rel ample}\eqref{rel ample 2} $\cL$ is ample, there exists some $n\geq 0$ such that $\cF\otimes \cL^n$ is generated by its global sections.  Hence for some $N> 0$ there exists a surjection $\cO^{\oplus N}\twoheadrightarrow \cF\otimes \cL^n$ and thus a surjection $(\cL^{-n})^{\oplus N}\twoheadrightarrow\cF$. Write $\cK$ for the kernel, then pushing down yields an exact sequence
\[
0\to\upgamma_*\cK
\to\upgamma_*(\cL^{-n})^{\oplus N}
\to\upgamma_*\cF
\to\mathbf{R}^1\upgamma_*\cK
\to\mathbf{R}^1\upgamma_*(\cL^{-n})^{\oplus N}
\to\mathbf{R}^1\upgamma_*\cF
\to 0,
\]
since $\mathbf{R}^2\upgamma_*=0$ by \v{C}ech cohomology.  But $\mathbf{R}^1\upgamma_*(\cL^{-n})^{\oplus N}$ is coherent by \ref{pushforward L}\eqref{pushforward L 2}, so it follows from the above exact sequence that $\mathbf{R}^1\upgamma_*\cF$ is also coherent.  Since $\cF$ was an arbitrary coherent sheaf, we also deduce that $\mathbf{R}^1\upgamma_*\cK$ is coherent.  Thus in the above exact sequence, combining with \ref{pushforward L}\eqref{pushforward L 2} we see that the second, fourth, fifth and sixth objects are coherent.  It follows that the third one is too, namely $\upgamma_*\cF$.

Hence $\upgamma$ is proper.  Further $\cL$ is $\upgamma$-relatively ample by \ref{rel ample}\eqref{rel ample 3}, and $\Spec S^{\ox}$ is separated since it is affine, and it is clearly quasi-compact.  It is well known that these conditions imply that $\upgamma$ is projective \cite[II.5.5.3]{EGA}.  Lastly, $\upgamma$ is birational by inspection, and the statement $\RDerived \upgamma_*\!\cO_{T^{\ox}}=\cO_{S^{\ox}}$ is just \ref{we have a map!}.
\end{proof}

\begin{cor}\label{Sox is rational}
Suppose that $\ox\in\Pos(\bL)$. Then $S^{\ox}$ is a rational surface singularity. 
\end{cor}
\begin{proof}
By \ref{Two dim and normal prelim}\eqref{L-factorial 2}, $S^{\ox}$ is two-dimensional and noetherian.  Further, by \ref{is projective}, $\upgamma\colon T^{\ox}\to\Spec S^{\ox}$ is a projective birational morphism such that $\RDerived \upgamma_*\!\cO_{T^{\ox}}=\cO_{S^{\ox}}$. Now by \ref{sings on T}, all the singularities on $T^{\ox}$ are rational, hence there exists a resolution $\upvarphi\colon Y\to T^{\ox}$ such that $\RDerived \upvarphi_*\!\cO_{Y}=\cO_{T^{\ox}}$.  Composing gives a projective birational morphism 
\[
\upgamma\circ \upvarphi\colon Y\to\Spec S^{\ox}
\] 
such that $\RDerived( \upgamma\circ \upvarphi)_*\cO_{Y}=\cO_{S^{\ox}}$.
\end{proof}

In the sequel write $\upvarphi\colon Y^{\ox}\to T^{\ox}$ for the minimal resolution of $T^{\ox}$, and consider the composition $\uppi\colon Y^{\ox}\to T^{\ox}\to \Spec S^{\ox}$.  We remark that this composition need not be the minimal resolution of $\Spec S^{\ox}$, and indeed later in \ref{when Y min res} we give a precise criterion for when it is.  Nevertheless, as in the introduction, we summarize the above information in the following commutative diagram
\begin{equation}
\begin{array}{c}
\begin{tikzpicture}
\node (top 1) at (0,0) {$\bT^{\ox}$};
\node (top 2) at (2.5,0) {$\bX$};
\node (bottom 1) at (0,-1.5) {$T^{\ox}$};
\node (bottom 2) at (2.5,-1.5) {$X\cong\mathbb{P}^1$};
\node (min) at (-1.5,-1) {$Y^{\ox}$};
\node (V) at (0,-2.5) {$\Spec S^{\ox}$};
\draw[->] (top 1) -- node[left] {$\scriptstyle g$} (bottom 1);
\draw[->] (top 2) -- node[right] {$\scriptstyle f$} (bottom 2);
\draw[->] (top 1) -- node[above] {$\scriptstyle q$} (top 2);
\draw[->] (bottom 1) -- node[above] {$\scriptstyle p$} (bottom 2);
\draw[->] (min) -- node[above] {$\scriptstyle \upvarphi$} (bottom 1);
\draw[->] (bottom 1) -- node[right] {$\scriptstyle \upgamma$} (V);
\draw[->,densely dotted] (min) --node[below left] {$\scriptstyle \uppi$} (V);
\end{tikzpicture}
\end{array}\label{stack diagram 2}
\end{equation}

\subsection{Tilting on $\bT^{\ox}$ and $T^{\ox}$}
Write $\cV\colonequals \cO_{\bP^1}\oplus\cO_{\bP^1}(1)\in\coh\bP^1$, and $\cE\colonequals \bigoplus_{\oy\in[0,\oc\,]}\cO_{\bX}(\oy)\in\coh\bX$. The following result is well known.

\begin{thm}\label{tilting on P1 and X}
The following statements hold.
\begin{enumerate}
\item $\cV$ is a tilting bundle on $\mathbb{P}^1$.
\item $\cE$ is a tilting bundle on $\bX$.
\end{enumerate}
\end{thm}
\begin{proof}
Part (1) is \cite{B83} and part (2) is \cite{GL1}.
\end{proof}

In this subsection we lift these tilting bundles to tilting bundles on both $T^{\ox}$ and $\bT^{\ox}$, again under the assumption that $0\neq \ox\in\bL_+$.  This is the singular line bundle (respectively stack) version of the usual trick of lifting tilting bundles on projective Fano varieties to the total spaces of various vector bundles, considered by many authors  \cite{TarigUeda, Bridgelandt, BH, VdB2}.  We remark that without the restriction to $\bL_+$, the following is false. 

\begin{thm}\label{tilting on stack T}
If $0\neq\ox\in\bL_+$, then $q^*\cE$ is a tilting bundle on $\bT^{\ox}$ such that
\[
\begin{tikzpicture}[xscale=1]
\node (a1) at (0,0) {$\D(\Qcoh\bT^{\ox})$};
\node (a2) at (5,0) {$\D(\Mod\End_{\bT^{\ox}}(q^*\cE))$};
\node (b1) at (0,-1.5) {$\D(\Qcoh\bX)$};
\node (b2) at (5,-1.5) {$\D(\Mod\Lambda)$};
\draw[->] (a1) -- node[above] {$\scriptstyle \RHom_{\bT^{\ox}}(q^*\cE,-)$} node[below] {$\scriptstyle \sim$} (a2);
\draw[->] (b1) -- node[above] {$\scriptstyle \RHom_{\bX}(\cE,-)$} node[below] {$\scriptstyle \sim$} (b2);
\draw[->] (a1) -- node[left] {$\scriptstyle \Rq_*$} (b1);
\draw[->] (a2) -- node[right] {{\scriptsize\rm res}}  (b2);
\end{tikzpicture}
\]
commutes, where $\Lambda$ is the canonical algebra.
\end{thm}
\begin{proof}
To simplify, we drop all $\ox$ from the notation and set $\bT\colonequals \bT^{\ox}$.  The generation argument is standard, as in \cite[4.1]{TarigUeda} and \cite[4.1]{Bridgelandt}, namely if $M\in\D(\Qcoh \bT)$ with $\Hom_{\D(\bT)}(q^*\cE,M[i])=0$ for all $i$, then $\Hom_{\D(\bX)}(\cE,q_*M[i])=0$ for all $i$, so since $\cE$ generates  $\D(\Qcoh \bX)$, $q_*M=0$ and so since $q$ is affine $M=0$.  Hence $q^*\cE$ generates $\D(\Qcoh \bT)$.

For Ext vanishing,
\begin{align*}
\Ext^1_{\bT}(q^*\cE,q^*\cE)&\cong\Ext^1_{\bX}(\cE,q_*q^*\cE)\tag{by adjunction}\\
&\cong\Ext^1_{\bX}(\cE,\bigoplus_{k\geq 0}\cE\otimes\cO_{\bX}(k\ox))\tag{by projection formula}\\
&\cong\bigoplus_{k\geq 0}\Ext^1_{\bX}(\cE,\cE\otimes\cO_{\bX}(k\ox))\\
&\cong\bigoplus_{k\geq 0}\bigoplus_{i\in[0,\oc\,]}\bigoplus_{j\in[0,\oc\,]}H^1(\bX,\cO_{\bX}(i-j+k\ox))\\
&\cong\bigoplus_{k\geq 0}\bigoplus_{i\in[0,\oc\,]}\bigoplus_{j\in[0,\oc\,]}H^0(\bX,\cO_{\bX}(\vec{\omega}-i+j-k\ox))^*\tag{by Serre duality}
\end{align*}
It suffices to check that $\vec{\omega}-i+j-k\ox\notin\bL_+$ for all $k\geq 0$ and all $i,j\in[0,\oc\,]$.  Since $0\neq\ox\in\bL_+$, clearly if suffices to check $i=k=0$ and $j=\oc$, this being the most positive case.  But $\vec{\omega}=(n-2)\oc-\sum_{t=1}^n\ox_t$, and so $\vec{\omega}+\oc=(n-1)\oc-\sum_{t=1}^n\ox_t\notin\bL_+$,	
as required.  Replacing $\Ext^1$ by $\Ext^i$, the above proof also shows that the higher Exts vanish.

The commutativity is just the adjunction $\RHom_{\bX}(\cE,\Rq_*(-))\cong\RHom_{\bT}(q^*\cE,-)$.
\end{proof}

\begin{thm}\label{tilting on coarse T}
If $0\neq\ox\in\bL_+$, then $p^*\cV$ is a tilting bundle on $T^{\ox}$. \end{thm}
\begin{proof}
As above, when possible we drop all $\ox$ from the notation.  The generation argument is identical to the argument in \ref{tilting on stack T}.  The Ext-vanishing is also similar, namely writing $\cF\colonequals f^*\cV=\cO_{\bX}\oplus\cO_{\bX}(\oc)$ then there is a chain of isomorphisms
\begin{align*}
\Ext^i_T(p^*\cV,p^*\cV)&\cong\Ext^i_T(p^*\cV,g_*\cO_{\bT}\otimes_T p^*\cV)\tag{$g_*\cO_{\bT}=\cO_T$}\\
&\cong\Ext^i_T(p^*\cV,g_*g^*p^*\cV)\tag{projection formula}\\
&\cong\Ext^i_{\bT}(g^*p^*\cV,g^*p^*\cV)\tag{adjunction}\\
&\cong\Ext^i_{\bT}(q^*f^*\cV,q^*f^*\cV)\tag{commutativity of \eqref{stack diagram 2}}\\
&\cong\Ext^i_{\bT}(q^*\cF,q^*\cF)
\end{align*}
which is zero by \ref{tilting on stack T} since $q^*\cF$ is a summand of $q^*\cE$.
\end{proof}

Since $\uppi\colon Y^{\ox}\to\Spec S^{\ox}$ is a resolution of a rational surface singularity, the fundamental cycle exists.
\begin{cor}\label{fund is reduced}
If $0\neq\ox\in\bL_+$, then the fundamental cycle associated to the morphism $\uppi\colon Y^{\ox}\to\Spec S^{\ox}$ is reduced. 
\end{cor}
\begin{proof}
Resolving the singularities in \eqref{T picture 2} it is clear that the dual graph of $\uppi$ is star shaped, with the middle curve of this star corresponding to the $\mathbb{P}^1$ in $T^{\ox}$.  By \ref{tilting on coarse T} the line bundle $\cL\colonequals p^{*}\cO_{\mathbb{P}^1}(1)$ on $T^{\ox}$ satisfies $\Ext^1_{T^{\ox}}(\cL,\cO_{T^{\ox}})=0$.  It clearly has degree one on the exceptional curve.  Then $\cL_Y\colonequals \upvarphi^*\cL=\mathbf{L}\upvarphi^*\cL$ is a line bundle on $Y^{\ox}$, with degree one on the middle curve and degree zero on all other curves.  Furthermore
\begin{align*}
H^1(\cL_Y^{-1})&\cong\Ext^1_{Y^{\ox}}(\cL_Y,\cO_{Y^{\ox}})\\
&\cong\Hom_{\Db(Y^{\ox})}(\mathbf{L}\upvarphi^*\cL,\cO_{Y^{\ox}}[1])\\
&\cong\Hom_{\Db(T^{\ox})}(\cL,\mathbf{R}\upvarphi_*\cO_{Y^{\ox}}[1])\\
&\cong\Ext^1_{T^{\ox}}(\cL,\cO_{T^{\ox}})\\
&=0.
\end{align*}
Since $\cL_Y$ has rank one, by \ref{Wunram main} (see also \cite[3.5.4]{VdB1d}), this implies that in the fundamental cycle of $\uppi$, the middle curve has coefficient one.  In \eqref{Zf dot Ei}, this implies that $\upbeta-v\geq 0$, and thus the fundamental cycle is reduced by the paragraph following  \eqref{Zf dot Ei}.
\end{proof}

In the sequel, we require the following description of some degenerate cases.
\begin{lemma}\label{degen lemma}
Let $0\neq\ox\in\bL_+$ and write $\ox=\sum_{i=1}^na_i\ox_i+a\oc$ with $a\ge0$ in normal form.
\begin{enumerate}
\item\label{degen lemma 1} If all $a_i=0$ (so necessarily $a>0$), then $Y^{\ox}=T^{\ox}=\cO_{\mathbb{P}^1}(-a)$ and $S^{\ox}=\KK[x,y]^{\frac{1}{a}(1,1)}$.
\item\label{degen lemma 2} If $a_i\neq 0$ and $a_j=0$ for all $j\neq i$, then the dual graph of $\uppi$ in \eqref{stack diagram 2} is
\[
\begin{tikzpicture}[xscale=1.2]
  \node (1) at (1,0) [vertex] {};
  \node (2) at (2,0) [vertex] {};
 \node (3) at (4,0) [vertex] {};
 \node (4) at (5,0)[vertex] {};
 \node at (3,0) {$\cdots$};
  \node (1a) at (0.9,-0.25) {$\scriptstyle -1-a$};
  \node (2a) at (1.9,-0.25) {$\scriptstyle - \upalpha_{i1}$};
 \node (3a) at (3.9,-0.25) {$\scriptstyle -\upalpha_{im_i-1}$};
 \node (4a) at (4.9,-0.25) {$\scriptstyle - \upalpha_{im_i}$};
\draw [-] (1) -- (2);
\draw [-] (2) -- (2.6,0);
\draw [-] (3.4,0) -- (3);
\draw [-] (3) -- (4);
\end{tikzpicture}
\]
where $\frac{p_i}{p_i-a_i}=[\upalpha_{i1},\hdots,\upalpha_{im_i}]$. 
\end{enumerate}
\end{lemma}
\begin{proof}
(1) By \ref{basic observation} $S^{\oc}=\KK[t_0,t_1]$, and hence $S^{a\oc}$ is the $a$-th Veronese of $\KK[t_0,t_1]$, which is $\KK[x,y]^{\frac{1}{a}(1,1)}$.  Since $(S_{t_i}[t])_0=\KK[\frac{t_{1-i}}{t_i},t_i^{a}t]$, the description \eqref{T=V0 cup V1} of $T^{\ox}$ coincides with that of $\cO_{\mathbb{P}^1}(-a)$.  Thus $\cO_{\mathbb{P}^1}(-a)=T^{\ox}=Y^{\ox}$. \\
(2) There is only one singularity in \eqref{T picture 2}, which implies that the dual graph has the above Type $A$ form.  It is standard that $\upalpha_{ij}$ from $\frac{p_i}{p_i-a_i}=[\upalpha_{i1},\hdots,\upalpha_{im_i}]$ resolves $\frac{1}{p_i}(1,-a_i)$, and thus the only thing still to be verified is the self-intersection number $-1-a$.  There are two ways of doing this: since the fundamental cycle of $\uppi$ is reduced by \ref{fund is reduced}, the reconstruction algebra is easy to calculate and it can be directly verified that its quiver  has the form given by intersection rules in \ref{GL2 for R} (which, by \cite{WemGL2}, hold for non-minimal resolutions too).  Alternatively, the number $-1-a$ can be determined by an explicit gluing calculation on $T^{\ox}$; in both cases we suppress the details.
\end{proof}

\subsection{Special CM Modules and the Dual Graph}\label{dual subsect}
Choose $0\neq\ox\in\bL_+$. In this subsection we first give a precise criterion for when $\uppi\colon Y^{\ox}\to\Spec S^{\ox}$ in \eqref{stack diagram 2} is the minimal resolution, then we use the results of the previous subsections to determine the indecomposable special CM $S^{\ox}$-modules.

\begin{prop}\label{when Y min res}
Let $0\neq\ox\in\bL_+$.  Then $\uppi$ is the minimal resolution if and only if $\ox\notin[0,\oc\,]$.
\end{prop}
\begin{proof}
Write $\ox=\sum_{i=1}^na_i\ox_i+a\oc$ in normal form, then since $\ox\in\bL_+$, necessarily $a\geq 0$.  As in \ref{fund is reduced}, resolving the singularities in \eqref{T picture 2} it is clear that the dual graph of $\uppi$ is star shaped, and the only curve that might be a ($-1$)-curve is the middle one. \\
($\Leftarrow$) Suppose that $\ox\notin[0,\oc\,]$.  If all $a_i=0$ then necessarily $a\geq 2$, and so \ref{degen lemma}\eqref{degen lemma 1} shows that $\uppi$ is the minimal resolution.  Similarly, if $a_i\neq 0$ but $a_j=0$ for all $j\neq i$, then the assumption $\ox\notin[0,\oc\,]$ forces $a\geq 1$, and \ref{degen lemma}\eqref{degen lemma 2} then shows that $\uppi$ is minimal.

Hence we can assume that $\ox\notin[0,\oc\,]$ with at least two of the $a_i$ being non-zero.  This being the case, there are at least two singular points in \eqref{T picture 2}.  By \ref{fund is reduced}, since the fundamental cycle is reduced, the calculation \eqref{Zf dot Ei} shows that the middle curve then cannot be a ($-1$)-curve, hence the resolution is minimal.\\
($\Rightarrow$) By contrapositive, suppose that $0\neq \ox\in[0,\oc\,]$, say $\ox=a_i\ox_i$ for some $i$ and some $0<a_i<p_i$.  Since $a=0$, by \ref{degen lemma}\eqref{degen lemma 2} the resolution $\uppi$ is not minimal.
\end{proof}

Hence if $x\in\bL_+$ with $x\notin[0,\oc\,]$, it follows that the dual graph of the minimal resolution $\uppi\colon Y^{\ox}\to\Spec S^{\ox}$ is  \eqref{key dual graph}, except that we have not yet determined the precise value of $\upbeta$.  We will do this later in \ref{middle SI number}, since for the moment this value is not needed.  As notation, for $\oy\in\bL$ write $S(\oy)^{\ox}\colonequals \bigoplus_{i\in\bZ}S_{\oy+i\ox}$.

\begin{thm}\label{specials determined thm}
Suppose that $\ox\in\bL_+$ with $\ox\notin[0,\oc\,]$, and write $\ox=\sum_{i=1}^na_i\ox_i+a\oc$ with $a\ge0$ in normal form.
Then 
\[
\SCM S^{\ox}=\add\{S(u\ox_j)^{\ox}\mid j\in [1,n],  u\in I(p_j,p_j-a_j)   \}.
\]
\end{thm}
\begin{proof}
The ring $S^{\ox}=S^{\bN\ox}$ has a unique singular point corresponding to the graded maximal ideal by \ref{Two dim and normal prelim}\eqref{L-factorial 7}. Thus, by \ref{Check locally gen} we may complete $S^{\ox}$ at this point and pass to the formal fibre, which is still the minimal resolution.  However, to aid readability, we do not add $\widehat{(-)}$ to the notation.

Consider the bundle $q^*\cE$ on $\bT$, and its pushdown $g_*q^*\cE$ on $T^{\ox}$.  At the point $\lambda_1$ of $T^{\ox}$, which is the singularity $\frac{1}{p_1}(1,-a_1)$ by \ref{sings on T}, the sheaves
\begin{equation}\label{CM at 1}
g_*q^*\cO,g_*q^*\cO(\ox_1),\hdots,g_*q^*\cO((p_1-1)\ox_1)
\end{equation}
are all locally free away from the point $\lambda_1$, since at any other singular point $\lambda_i$, multiplication by $x_1$ is invertible.  Further, at the point $\lambda_1$, \eqref{CM at 1} is a full list of the CM modules, indexed by the characters of  $\bZ_{p_1}=\frac{1}{p_1}(1,-a_1)$ in the obvious way.  Hence by \ref{Wunram specials thm}, which does not require any coprime assumption, the torsion-free pullbacks under $\upvarphi$ of
\[
\{ g_*q^*\cO(u\ox_1)\mid u\in I(p_1,p_1-a_1)\backslash \{0,p_1\}\}
\] 
are precisely the line bundles on $Y^{\ox}$ corresponding to the curves in arm $1$ of the dual graph. By \ref{Wunram main} and \ref{fund is reduced} they are the special bundles on $Y^{\ox}$ corresponding to the curves in arm $1$ of the dual graph, hence their pushdown (via $\uppi$) to $S^{\ox}$ are the special CM $S^{\ox}$-modules corresponding to arm $1$.  Since the pushdown under $\upvarphi$ of the torsion-free pullback of $\upvarphi$ is the identity, the pushdown to $S^{\ox}$ gives the modules
\[
\left\{ \upgamma_*g_*q^*\cO(u\ox_1)\mid u\in I(p_1,p_1-a_1)\backslash \{0,p_1\}\right\}.
\] 
Then, by \ref{pushforward L}\eqref{pushforward L 0}, $\upgamma_*g_*q^*\cO(u\ox_1)=\bigoplus_{i\geq 0}S_{i\ox+ux_1}$.
But since $\ox\in\bL_+$, $\ox\notin[0,\oc\,]$ and $u\leq p_1$ we see that $ux_1-\ox\leq \oc-\ox\ngeq 0$, and hence $\upgamma_*g_*q^*\cO(u\ox_1)=\bigoplus_{i\in\bZ}S_{i\ox+ux_1}\colonequals S(u\ox_1)^{\ox}$. 

The argument for the other arms is identical. The argument that the middle curve gives the special CM module $S(\oc)^{\ox}$ follows again by \ref{fund is reduced}.
\end{proof}

\begin{remark}\label{positions forced}{\rm
It is possible to assign each special CM $S^{\ox}$-module to its vertex in the dual graph of the minimal resolution across the bijection in \ref{Wunram main}, see below \ref{Ex3.10} for a typical example.  As in \ref{Ex3.10}, there are obvious irreducible morphisms between the special CM $S^{\ox}$-modules, so they must appear in the quiver of the reconstruction algebra.  By the intersection theory in \ref{GL2 for R}, we conclude that $S(\oc)^{\ox}$ corresponds to the middle vertex, and this forces the positions of the other special CM modules relative to the dual graph.}
\end{remark}

\begin{example}\label{Ex3.10}{\rm
Consider the example $(p_1,p_2,p_3)=(3,5,5)$ and $\ox=2\ox_1+2\ox_2+3\ox_3$.  The continued fractions for $\frac{p_i}{p_i-a_i}$, and the corresponding $i$-series are given by:
\[
\begin{array}{ll}
\frac{3}{3-2}=[3] & 3>1>0\\
\frac{5}{5-2}=[2,3] & 5>3>1>0\\
\frac{5}{5-3}=[3,2] & 5>2>1>0
\end{array}
\]
It follows from \ref{specials determined thm} that an additive generator of $\SCM S^{\ox}$ is given by the direct sum of the following circled modules:
\[
\begin{array}{c}
\begin{tikzpicture}[xscale=1.3,yscale=1.3]
\node (0) at (0,0) {$\scriptstyle S(\oc)^{\ox}$};
\node (A1) at (-1.5,1) {$\scriptstyle S(2\ox_1)^{\ox}$};
\node (A4) at (-1.5,4) {$\scriptstyle S(\ox_1)^{\ox}$};
\node (B1) at (0,1) {$\scriptstyle S(4\ox_2)^{\ox}$};
\node (B2) at (0,2) {$\scriptstyle S(3\ox_2)^{\ox}$};
\node (B3) at (0,3) {$\scriptstyle S(2\ox_2)^{\ox}$};
\node (B4) at (0,4) {$\scriptstyle S(\ox_2)^{\ox}$};
\node (C1) at (1.5,1) {$\scriptstyle S(4\ox_3)^{\ox}$};
\node (C2) at (1.5,2) {$\scriptstyle S(3\ox_3)^{\ox}$};
\node (C3) at (1.5,3) {$\scriptstyle S(2\ox_3)^{\ox}$};
\node (C4) at (1.5,4) {$\scriptstyle S(\ox_3)^{\ox}$};
\node (T) at (0,5) {$\scriptstyle S^{\ox}$};
\draw [->] (A1) -- node[fill=white,inner sep=1pt]{$\scriptstyle x_1$}(0);
\draw [->] (B1) -- node[fill=white,inner sep=1pt]{$\scriptstyle x_2$}(0);
\draw [->] (C1) -- node[fill=white,inner sep=1pt]{$\scriptstyle x_3$}(0);
\draw [->] (A4) -- node[right] {$\scriptstyle x_1$} (A1);
\draw [->] (B2) -- node[right]  {$\scriptstyle x_2$}(B1);
\draw [->] (C2) -- node[right]  {$\scriptstyle x_3$}(C1);
\draw [->] (B3) -- node[right]  {$\scriptstyle x_2$}(B2);
\draw [->] (C3) -- node[right]  {$\scriptstyle x_3$}(C2);
\draw [->] (B4) -- node[right]  {$\scriptstyle x_2$}(B3);
\draw [->] (C4) -- node[right]  {$\scriptstyle x_3$}(C3);
\draw [->] (T) -- node[fill=white,inner sep=1pt]{$\scriptstyle x_1$}(A4);
\draw [->] (T) -- node[fill=white,inner sep=1pt]{$\scriptstyle x_2$}(B4);
\draw [->] (T) -- node[fill=white,inner sep=1pt]{$\scriptstyle x_3$}(C4);
\draw[blue] (T) circle (11pt);
\draw[blue] (A4) circle (11pt);
\draw[blue] (B4) circle (11pt);
\draw[blue] (B2) circle (11pt);
\draw[blue] (C4) circle (11pt);
\draw[blue] (C3) circle (11pt);
\draw[blue] (0) circle (11pt);
\end{tikzpicture}
\end{array}
\]}
\end{example}

Consider the tilting bundle $\cM$ on $Y^{\ox}$, generated by global sections, constructed in \cite[3.5.4]{VdB1d}.
\begin{cor}\label{stackminres}
If $0\neq\ox\in\bL_+$, then the following statements hold.
\begin{enumerate}
\item\label{stackminres 0} $\uppi_*\cM$ is a summand of $\bigoplus_{\oy\in[0,\oc\,]}S(\oy)^{\ox}=\upgamma_*g_*(q^*\cE)$.
\item\label{stackminres 1} There is an idempotent $e\in\End_{\bT^{\ox}}(q^*\cE)$ such that $e\End_{\bT^{\ox}}(q^*\cE)e\cong\End_{Y^{\ox}}(\cM)$.
\item\label{stackminres 2} There is a fully faithful embedding 
\[
\Db(\coh Y^{\ox})\hookrightarrow \Db(\coh\bT^{\ox}).
\]
\end{enumerate}
\end{cor}
\begin{proof}
(1) By \ref{fund is reduced} the fundamental cycle is reduced. It follows that $\uppi_*\cM$ is a summand of $\bigoplus_{\oy\in[0,\oc\,]}S(\oy)^{\ox}$, by the argument in the proof of \ref{specials determined thm}.\\
(2)  Even although $\uppi\colon Y^{\ox}\to\Spec S^{\ox}$ need not be the minimal resolution, it is still true by \cite[4.3]{DW2} that
\[
\End_{Y^{\ox}}(\cM)\cong\End_{S^{\ox}}(\uppi_*\cM)
\]
On the other hand,
\[
\End_{\bT^{\ox}}(q^*\cE)
\cong \End_{S^{\ox}}(\upgamma_*g_*q^*\cE)
\cong\End_{S^{\ox}}(\bigoplus_{\oy\in[0,\oc\,]}S(\oy)^{\ox}).
\]
Thus by \eqref{stackminres 0} there is an idempotent $e\in\End_{\bT^{\ox}}(q^*\cE)$ such that $e\End_{\bT^{\ox}}(q^*\cE)e\cong\End_{Y^{\ox}}(\cM)$.\\
(3)  By \eqref{stackminres 1}, writing $A\colonequals \End_{\bT^{\ox}}(q^*\cE)$ then $\End_{Y^{\ox}}(\cM)=eAe$, thus there is an obvious embedding of derived categories
\[
\RHom_{eAe}(Ae,-)\colon\D(\Mod\End_{Y^{\ox}}(\cM))\hookrightarrow\D(\Mod\End_{\bT^{\ox}}(q^*\cE)),
\]
and also an embedding given by $-\otimes^{\bf L}_{eAe}eA$.
Regardless, since $\gl eAe<\infty$, the above induces an embedding
\[
\Db(\mod\End_{Y^{\ox}}(\cM))\hookrightarrow\Db(\mod\End_{\bT^{\ox}}(q^*\cE)).
\]
The left hand side is equivalent to $\Db(\coh Y^{\ox})$, so it suffices to show that the right hand side is equivalent to $\Db(\coh \bT^{\ox})$.   By \ref{tilting on stack T}, there is an unbounded derived equivalence $\D(\Mod\End_{\bT^{\ox}}(q^*\cE))\simeq\D(\Qcoh \bT^{\ox})$.  This automatically restricts to an equivalence on compact objects.  The compact objects of $\D(\Qcoh \bT^{\ox})$ are $\Db(\coh \bT^{\ox})$ by \cite[A.3]{BLS}, and since $\End_{\bT^{\ox}}(q^*\cE)$ has finite global dimension, the compact objects of $\D(\Mod\End_{\bT^{\ox}}(q^*\cE))$ are $\Db(\mod\End_{\bT^{\ox}}(q^*\cE))$, as required.
\end{proof}

We give a simple criterion for when the above is an equivalence later in \ref{stackminres 2A}.  Note that the above result is formally very similar to the case of quotient singularities, where the reconstruction algebra embeds into the quotient stack $[\KK^2/G]$, but this embedding is also very rarely an equivalence.

\section{Categorical Equivalences}\label{qgr Veronese section}

In this section we investigate the conditions on $\ox$ under which 
\[
\coh\bX\simeq\qgr^{\bZ}\!S^{\ox}\simeq\qgr^{\bZ}\!\Upgamma_{\ox}
\]
holds.  This allows us, in \ref{stackminres 2A}, to give a precise criterion for when the embedding in \ref{stackminres}\eqref{stackminres 2} is an equivalence, and further it allows us in \ref{middle SI number} to determine the middle self-intersection number in \eqref{key dual graph}. Throughout many results in this section, a coprime condition $(p_i,a_i)=1$ naturally appears, and in \S\ref{changing parameters} we show that we can always change parameters so that this coprime condition holds.

\subsection{General Results on Categorical Equivalences}
To simplify the notation, in this subsection we first produce categorical equivalences in a very general setting, before specialising in the next subsection to the case of the weighted projective line.  Throughout this subsection, $k$ denotes an arbitrary field.

We start with a basic observation. Let $G$ be an abelian group and $A$ a noetherian $G$-graded $k$-algebra.  As in \S\ref{Section 1.3} we consider the categories $\mod^G\! A$, $\mod^G_0\! A$ and $\qgr^G\! A$.  For an idempotent $e\in A_0$, $B\colonequals eAe$ is a noetherian $G$-graded $k$-algebra. The functor 
\[
E\colonequals e(-)\colon\mod^G\!A\to\mod^G\!B
\]
has a left adjoint functor $E_\lambda$ and a right adjoint functor $E_\rho$ given by
\begin{align*}
E_\lambda\colonequals Ae\otimes_B-&\colon\mod^G\!B\to\mod^G\!A\\
E_\rho\colonequals \Hom_B(eA,-)&\colon\mod^G\!B\to\mod^G\!A.
\end{align*}
Moreover $EE_\lambda={\rm id}_{\mod^G\!B}=EE_\rho$ holds, and for the natural morphism $m\colon Ae\otimes_BeA\to A$, the counit $\upeta\colon E_\lambda E\to {\rm id}_{\mod^G\!A}$ is given by $m\otimes_{A}-$ and the unit $\upvarepsilon\colon {\rm id}_{\mod^G\!A}\to E_\rho E$ is given by $\Hom_A(m,-)$.

The following basic observation is a prototype of our results in this subsection.

\begin{prop}\label{comm to noncomm}
If $\dim_{k}A/(e)<\infty$, then $E$ induces an equivalence $\qgr^G\!A\simeq \qgr^G\!B$.
\end{prop}
\begin{proof}
Clearly $E_\lambda$ and $E$ induce an adjoint pair $E_\lambda\colon \qgr^G\!A\to \qgr^G\!B$ and $E\colon \qgr^G\!B\to\qgr^G\!A$. For any $X\in\mod^G\!A$, both the kernel and cokernel of $m\otimes_{A}X\colon E_\lambda EX\to X$ are finite dimensional since they are finitely generated $(A/(e))$-modules. Therefore $E_\lambda$ and $E$ give the desired equivalences.
\end{proof}

In the rest of this subsection, let $G$ be an abelian group and $H$ a subgroup of $G$ of finite index. Assume that $A$ is a noetherian $G$-graded $k$-algebra, and let $B\colonequals A^H=\bigoplus_{g\in H}A_g$ be the $H$-Veronese subring of $A$. There is a natural functor
\[
(-)^H\colon \mod^G\! A\to\mod^H\! B
\]
given by $X^H\colonequals \bigoplus_{h\in H}X_h$.

\begin{lemma}\label{Veron is fg2}
$B$ is a noetherian $k$-algebra and $A$ is a finitely generated $B$-module.
\end{lemma}

\begin{proof}
There is a finite direct sum decomposition $A=\bigoplus_{g\in G/H}A(g)^H$ as $B$-modules. For any submodule $M$ of $A(g)^H$, it is easy to check that the ideal  $AM$ of $A$ satisfies $AM\cap A(g)^H=M$. Therefore $A(g)^H$ is a noetherian $B$-module, since $A$ is a noetherian ring. The assertion follows.
\end{proof}

We say that $X\in\mod^G\!A$ has \emph{depth at least two} if $\Ext^i_A(Y,X)=0$ for any $i=0,1$ and $Y\in\mod^G\!A$ with $\dim_kY<\infty$. We write $\mod_2^G\!A$ for the full subcategory of $\mod^G\!A$ consisting of modules with depth at least two. We define $\mod^H_{2}B$ similarly.

\begin{thm}\label{WPL as qgrZ2}
Let $G$ be an abelian group, $H$ a subgroup of $G$ of finite index, $A$ a noetherian $G$-graded $k$-algebra, and $B\colonequals A^H$. Then the following conditions are equivalent. 
\begin{enumerate}
\item\label{WPL as qgrZ 2-1} The natural functor $(-)^{H}\colon\qgr^G\!A\to\qgr^H\!B$ is an equivalence.
\item\label{WPL as qgrZ 3-1} For any $i\in G$, the ideal $I^{i}\colonequals A(i)^{H}\cdot A(-i)^{H}$ of $B$ satisfies $\dim_k(B/I^{i})<\infty$.
\end{enumerate}
If $A$ belongs to $\mod_2^G\!A$, then the following condition is also equivalent.
\begin{enumerate}
\item[\rm (3)]\label{WPL as qgrZ 1-1} The natural functor $(-)^{H}\colon\mod_2^G\!A\to\mod_2^H\!B$ is an equivalence.
\end{enumerate}
\end{thm}
\begin{proof}
Consider the matrix algebra
\[
C=(A(i-j)^H)_{i,j\in G/H}
\]
whose rows and columns are indexed by $G/H$, and the product is given by the matrix multiplication together with the product in $A$, namely
\[
(s_{i,j})\cdot(t_{i,j})\colonequals (\sum_{k\in G/H}s_{i,k}\cdot t_{k,j}).
\]
Now we fix a complete set $I$ of representatives of $G/H$ in $G$. Then $C$ has an $H$-grading given by
\[
C_h\colonequals (A_{i-j+h})_{i,j\in I}.
\]
By \cite[Theorem 3.1]{IL} there is an equivalence
\begin{equation}\label{f functor}
F\colon\mod^G\!A\simeq \mod^H\!C
\end{equation}
sending $M=\bigoplus_{i\in G}M_i$ to $F(M)=\bigoplus_{h\in H}F(M)_h$, where $F(M)_h$ is defined by
\[
F(M)_h\colonequals (M_{i+h})_{i\in I}
\]
and the $C$-module structure is given by
\[
(s_{i,j})_{i,j\in I}\cdot(m_i)_{i\in I}\colonequals (\sum_{j\in I}s_{i,j}\cdot m_j)_{i\in I}.
\]
On the other hand, let $e\in C$ be the idempotent corresponding to $0\in G/H$.
Since $eCe=B$ holds, there is an exact functor
\begin{equation}\label{e functor}
E\colonequals e(-)\colon\mod^H\!C\to\mod^H\!B
\end{equation}
such that the following diagram commutes
\[
\begin{tikzpicture}
\node (A) at (0,0) {$\mod^G\!A$};
\node (B) at (4,0) {$\mod^H\!B$};
\node (C) at (2,-1) {$\mod^H\!C$};
\draw[->] (A) -- node[above] {$\scriptstyle  (-)^H$} (B);
\draw[->] (A) -- node[below] {$\scriptstyle F$} node[above] {$\scriptstyle\sim$} (C);
\draw[->] (C) --node[below] {$\scriptstyle E$} (B);
\end{tikzpicture}
\]
The functor \eqref{e functor} has a left adjoint functor $E_\lambda\colonequals Ce\otimes_B-\colon\mod^H\!B\to\mod^H\!C$ and a right adjoint functor $E_\rho\colonequals \Hom_B(eC,-)\colon\mod^H\!B\to\mod^H\!C$.\\
(1)$\Leftrightarrow$(2) The functors $F$ and $E$ induce an equivalence $F\colon\qgr^G\!A\simeq\qgr^H\!C$ and a functor
\begin{equation}\label{e functor2}
E\colon\qgr^H\!C\to\qgr^H\!B
\end{equation}
respectively, which make the following diagram commutative
\[
\begin{tikzpicture}
\node (A) at (0,0) {$\qgr^G\!A$};
\node (B) at (4,0) {$\qgr^H\!B$};
\node (C) at (2,-1) {$\qgr^H\!C$};
\draw[->] (A) -- node[above] {$\scriptstyle  (-)^H$} (B);
\draw[->] (A) -- node[below] {$\scriptstyle F$} node[above] {$\scriptstyle\sim$} (C);
\draw[->] (C) --node[below] {$\scriptstyle E$} (B);
\end{tikzpicture}
\]
Thus the functor $(-)^H\colon\qgr^G\!A\to\qgr^H\!B$ is an equivalence if and only if the functor \eqref{e functor2} is an equivalence.
The functor $E_\lambda\colon\mod^H\!B\to\mod^H\!C$ induces a left adjoint functor $E_\lambda\colon\qgr^H\!B\to\qgr^H\!C$ of \eqref{e functor2}. Clearly $E E_\lambda={\rm id}_{\qgr^H\!B}$ holds, and the counit $\upeta\colon E_\lambda E\to {\rm id}_{\qgr^H\!C}$ is given by $m\otimes_{C}-$, where $m$ is the natural morphism
\begin{equation}\label{m morphism}
m\colon Ce\otimes_BeC\to C.
\end{equation}
Thus the condition (1) holds if and only if $\upeta$ is an isomorphism of functors if and only if $m$ is an isomorphism in $\qgr^H\!C$. On the other hand, the cokernel of $m$ is $C/(e)$, where $(e)$ is the two-sided ideal of $C$ generated by $e$, and the kernel of $m$ is a finitely generated $C/(e)$-module. Therefore (1) holds if and only if the factor algebra $C/(e)$ of $C$ is finite dimensional if and only if (2) holds, by the following observation.

\begin{lemma}\label{fin dim and (3)}
$\dim_{k}C/(e)<\infty$ if and only if the condition (2) holds.
\end{lemma}

\begin{proof}[Proof of Lemma \ref{fin dim and (3)}]
Since
\[
C/(e)=(A(i-j)^H/(A(i)^H\cdot A(-j)^H))_{i,j\in I}
\]
holds, $C/(e)$ is finite dimensional if and only if $A(i-j)^H/(A(i)^H\cdot A(-j)^H)$ is finite dimensional for any $i,j\in I$.
This implies the condition (2) by considering the case $i=j$.

Conversely assume that (2) holds. Since there is a surjective map
\[
A(i-j)^H\otimes_B\frac{B}{A(j)^H\cdot A(-j)^H}=\frac{A(i-j)^H}{A(i-j)^H\cdot A(j)^H\cdot A(-j)^H}\to
\frac{A(i-j)^H}{A(i)^H\cdot A(-j)^H}
\]
whose domain is finite dimensional, the target is also finite dimensional.
Thus the assertion holds.
\end{proof}

\noindent
(2)$\Leftrightarrow$(3) Assume that $A\in\mod_2^G\!A$. Clearly the equivalence \eqref{f functor} induces equivalences
\[
F\colon\mod_0^G\!A\simeq \mod_0^H\!C\ \mbox{ and }\ 
F\colon\mod_2^G\!A\simeq \mod_2^H\!C.
\]
The remainder of the proof requires the following general lemma.

\begin{lemma}\label{depth 2 properties}
With the setup as above,
\begin{enumerate}
\item\label{depth 2 properties 1} The functor \eqref{e functor} induces a functor
\begin{equation}\label{e functor3}
E\colon\mod_2^H\!C\to\mod_2^H\!B.
\end{equation}
\item\label{depth 2 properties 2} The functor $E_\rho\colon\mod^H\!B\to\mod^H\!C$ induces a functor $E_\rho\colon\mod_2^H\!B\to\mod_2^H\!C$.
\item\label{depth 2 properties 3} $X\in\mod^H\!C$ belongs to $\mod_0^H\!C$ if and only if $\Ext^i_{C}(X,\mod_2^H\!C)=0$ for $i=0,1$.
\end{enumerate}
\end{lemma}
\begin{proof}[Proof of Lemma \ref{depth 2 properties}]
(1) Let $X\in\mod_2^H\!C$, $Y\in\mod_0^H\!B$ and ${\bf E}_\lambda Y\colonequals Ce\Ltimes_BY$.
Since $H^i({\bf E}_\lambda Y)$ is zero for any $i>0$ and belongs to $\mod_0^H\!C$ for any $i\le0$, we have $\Hom_{\Db(\mod C)}({\bf E}_\lambda Y,X[i])=0$ for $i=0,1$.
Using $\RHom_B(Y,EX)=\RHom_C({\bf E}_\lambda Y,X)$, we have $\Ext^i_B(Y,EX)=0$ for $i=0,1$.\\
(2) Let $X\in\mod_2^H\!B$, $Y\in\mod_0^H\!C$ and ${\bf E}_\rho X\colonequals \RHom_B(eC,X)$. Since $\RHom_C(Y,{\bf E}_\rho X)=\RHom_B(EY,X)$ and $EY\in\mod_0^H\!B$ hold, we have $\Hom_{\Db(\mod C)}(Y,{\bf E}_\rho X[i])=0$ for $i=0,1$.
There is a triangle
$$E_\rho X\to{\bf E}_\rho X\to Z\to E_\rho X[1]$$
satisfying $H^i(Z)=0$ for all $i\le 0$. Applying $\Hom_{\Db(\mod C)}(Y,-)$ gives $\Ext^i_C(Y,E_\rho X)=0$ for $i=0,1$.\\
(3) It suffices to prove the `if' part. Our assumption $A\in\mod_2^G\!A$ implies $C=\bigoplus_{i\in I}F(A(-i))\in\mod_2^H\!C$, since $Ce_j=(A(i-j)^H)_{i\in I}=F(A(-j))$.
Let $0\to T\to X\to F\to0$ and $0\to\Omega F\to P\to F\to0$ be exact sequences in $\mod^H\!C$ such that $T$ is the largest submodule of $X$ which belongs to $\mod_0^H\!C$ and $P$ is an $H$-graded projective $C$-module. Then $\Omega F$ belongs to $\mod_2^H\!C$ since $C\in\mod_2^H\!C$. Applying $\Hom_C(-,\Omega F)$ to the first sequence gives an exact sequence
\[
0=\Hom_C(T,\Omega F)\to\Ext^1_C(F,\Omega F)\to\Ext^1_C(X,\Omega F)=0.
\]
Thus $\Ext^1_C(F,\Omega F)=0$ holds, and $F$ is projective in $\mod^H\!C$. Hence $X=T\oplus F$, and so $\Hom_C(X,C)=0$ implies that $F=0$. Therefore $X=T$ belongs to $\mod_0^H\!C$.
\end{proof}

It follows from \ref{depth 2 properties} that there is a commutative diagram
\[
\begin{tikzpicture}
\node (A) at (0,0) {$\mod^G_{2}A$};
\node (B) at (4,0) {$\mod^H_{2}B$};
\node (C) at (2,-1) {$\mod^H_{2}C$};
\draw[->] (A) -- node[above] {$\scriptstyle  (-)^{H}$} (B);
\draw[->] (A) -- node[below] {$\scriptstyle F$} node[above] {$\scriptstyle\sim$} (C);
\draw[->] (C) --node[below] {$\scriptstyle E$} (B);
\end{tikzpicture}
\]
Thus the functor $(-)^H\colon\mod_2^G\!A\to\mod_2^H\!B$ is an equivalence if and only if the functor \eqref{e functor3} is an equivalence. By \ref{depth 2 properties}\eqref{depth 2 properties 2}, there is a right adjoint functor $E_\rho\colon\mod_2^H\!B\to\mod_2^H\!C$ of \eqref{e functor3}. Clearly $E E_\rho={\rm id}_{\mod_2^H\!B}$ holds, and the unit $\upvarepsilon\colon {\rm id}_{\mod_2^H\!C}\to E_\rho E$ is given by $\Hom_{C}(m,-)$, where $m$ is the morphism \eqref{m morphism}.
Thus the condition (3) holds if and only if $\upvarepsilon=\Hom_{C}(m,-)$ is an isomorphism of functors.

Now fix $X\in\mod_2^H\!C$ and apply $\Hom_{C}(-,X)$ to exact sequences $0\to (e)\to C\to C/(e)\to0$ and $0\to\Ker m\to Ce\otimes_BeC\to (e)\to0$.  This gives exact sequences
\begin{eqnarray*}
&0\to\Hom_{C}(C/(e),X)\to X\to\Hom_{C}((e),X)\to\Ext^1_{C}(C/(e),X)\to0&\\
&0\to\Hom_{C}((e),X)\to\Hom_{C}(Ce\otimes_BeC,X)\to\Hom_{C}(\Ker m,X).&
\end{eqnarray*}
Therefore, if $C/(e)$ is finite dimensional, then so is $\Ker m$ and hence $\upvarepsilon$ is an isomorphism. Conversely, if $\upvarepsilon$ is an isomorphism, then $\Ext^i_{C}(C/(e),X)=0$ for $i=0,1$ for any $X\in\mod_2^H\!C$ and hence $C/(e)$ is finite dimensional by \ref{depth 2 properties}\eqref{depth 2 properties 3}.  Consequently (3) is equivalent to (2), again by \ref{fin dim and (3)}.
\end{proof}

Later we need the following observation.

\begin{lemma}\label{slightly stronger}
In the setting of \ref{WPL as qgrZ2}, assume that the condition \eqref{WPL as qgrZ 3-1} is satisfied. Then for any $X\in\mod^G\!A$ and $Y\in\mod_2^G\!A$, there is an isomorphism
\[
\Hom_B(X^H,Y^H)\cong\Hom_A(X,Y)^H
\]
of $H$-graded $k$-modules.
\end{lemma}

\begin{proof}
Clearly $\Hom_B(X^H,Y^H)=\Hom_B(EFX,EFY)=\Hom_C(E_\lambda EFX,FY)$. This is isomorphic to $\Hom_C(FX,FY)$ since the kernel and the cokernel of $\upeta_X\colon E_\lambda EX\to X$ are finite dimensional by our assumptions.  Finally,
\[
\Hom_C(FX,FY)=\bigoplus_{h\in H}\Hom_{C}^H(FX,(FY)(h))
=\bigoplus_{h\in H}\Hom_{A}^G(X,Y(h))
=\Hom_A(X,Y)^H.\qedhere
\]
\end{proof}

\subsection{Categorical Equivalences for Weighted Projective Lines} In this subsection, we apply the general results of the previous subsection to describe the precise conditions on $\ox\in\bL$ for which $\qgr^\bZ\! S^{\ox}\simeq\coh\bX$ holds. As before, write $S(\oy)^{\ox}\colonequals \bigoplus_{i\in\bZ}S_{\oy+i\ox}$. This subsection does not require the condition that $\ox$ belongs to $\bL_+$, instead assuming that $\ox$ is not torsion.  The following is the main result, where the special case $\ox=\ow$ was given in \cite{GL91}.  Another approach can be found in \cite{H}.

\begin{thm}\label{WPL as qgrZ}
If $\ox=\sum_{i=1}^na_i\ox_i+a\oc\in\bL$ is not torsion, then the following conditions are equivalent.
\begin{enumerate}
\item\label{WPL as qgrZ 1} The natural functor $(-)^{\ox}\colon\CM^\bL\!S\to\CM^\bZ\!S^{\ox}$ is an equivalence.
\item\label{WPL as qgrZ 2} The natural functor $(-)^{\ox}\colon\qgr^\bL\!S\to\qgr^\bZ\!S^{\ox}$ is an equivalence.
\item\label{WPL as qgrZ 3} For any $\oz\in\bL$, the ideal $I^{\oz}\colonequals S(\oz)^{\ox}\cdot S(-\oz)^{\ox}$ of $S^{\ox}$ satisfies $\dim_\KK(S^{\ox}/I^{\oz})<\infty$.
\item\label{WPL as qgrZ 4} $(p_i,a_i)=1$ for all $1\le i\le n$.
\end{enumerate}
\end{thm}
\begin{proof}
To ease notation, write $R\colonequals S^{\ox}$.\\
(1)$\Leftrightarrow$(2)$\Leftrightarrow$(3) These are shown in \ref{WPL as qgrZ2} since $\CM^\bL\!S=\mod_2^\bL\!S$ and $\CM^\bZ\!R=\mod_2^\bZ\!R$.\\
(3)$\Rightarrow$(4). By contrapositive, assume that $a_1$ and $p_1$ are not coprime. Then the normal form of any element in $\ox_1+\bZ\ox$ (respectively, $-\ox_1+\bZ\ox$) contains a positive multiple of $\ox_1$. Thus we have
\[
I^{\ox_1}\subset Sx_1\cdot Sx_1=Sx_1^2.
\]
Therefore the condition (3) implies that the algebra $R/(R\cap Sx_1^2)$ is finite dimensional. Since $S/Sx_1^2$ is a finitely generated $R/(R\cap Sx_1^2)$-module by \ref{Two dim and normal prelim}\eqref{L-factorial 2}, it is also finite dimensional. This is a contradiction since $S$ has Krull dimension two.\\ 
(4)$\Rightarrow$(3). Assume that $(p_i,a_i)=1$ for all $i$.  If $R/I^{\oy}$ and $R/I^{\oz}$ are finite dimensional, then so is $R/I^{\oy+\oz}$ since $I^{\oy}\cdot I^{\oz}\subset I^{\oy+\oz}$ holds.
Thus we only have to show that $R/I^{\ox_i}$ is finite dimensional for each $i$ with $1\le i\le n$. We will show that $I^{\ox_i}$ contains a certain power $A$ of $x_i$ and a certain monomial $B$ of $x_j$'s with $j\neq i$. Then it is easy to check that $S/(SA+SB)$ is finite dimensional, and hence $R/(RA+RB)=(S/(SA+SB))^{\ox}$ and $R/I^{\ox_i}$ are also finite dimensional.

For the least common multiple $p$ of $p_1,\ldots,p_n$, we have $p\ox=q\oc$ for some $q>0$. Then
\[
I^{\ox_i}= S(\ox_i)^{\ox}\cdot S(-\ox_i)^{\ox}
\supset S_{\ox_i}\cdot S_{-\ox_i+p\ox}\ni x_i\cdot x_i^{p_iq-1}=x_i^{p_iq}.
\]
Thus $I^{\ox_i}$ contains a power of $x_i$.
On the other hand, since $a_i$ and $p_i$ are coprime, there exist integers $\ell$ and $m$ such that $a_i\ell+1=p_im$ and
$\ox_i+\ell\ox\in\bL_+$. Then the normal form of $\ox_i+\ell\ox$ does not contain a positive multiple of $\ox_i$, and hence $S(\ox_i)^{\ox}\supset S_{\ox_i+\ell\ox}$ contains a monomial of $x_j$'s with $j\neq i$.
Applying a similar argument to $S(-\ox_i)^{\ox}$, we have that $I^{\ox_i}=S(\ox_i)^{\ox}\cdot S(-\ox_i)^{\ox}$ contains a monomial of $x_j$'s with $j\neq i$. Thus the assertion follows.
\end{proof}

The following is a geometric corollary of the results in this subsection.
\begin{cor}\label{stackminres 2A} 
Suppose that $0\neq\ox\in\bL_+$ and write $\ox=\sum_{i=1}^{n}a_i\ox_i+a\oc$ in normal form. If $n\ge1$ and $(p_i,a_i)=1$ for all $1\leq i\leq n$, then the fully faithful embedding
\[
\Db(\coh Y^{\ox})\hookrightarrow \Db(\coh\bT^{\ox})
\]
in \ref{stackminres} is an equivalence if and only if every $a_i=1$, that is $\ox=\sum_{i=1}^nx_i+a\oc$.
\end{cor}
\begin{proof}
We use the notation from the proof of \ref{stackminres}. Note that from the assumption $(p_i,a_i)=1$ for every $1\leq i\leq n$, necessarily each $a_i$ is non-zero. Next, the indecomposable summands of $\uppi_*\cM$ are pairwise non-isomorphic by combining \cite[3.5.3]{VdB1d} and \cite[4.3]{DW2}, and the summands of $\bigoplus_{\oy\in[0,\oc\,]}S(\oy)^{\ox}$ are pairwise non-isomorphic by \ref{WPL as qgrZ}\eqref{WPL as qgrZ 1}. 

The embedding in \ref{stackminres} is induced from idempotents using the observation that $\uppi_*\cM$ is a summand of  $\bigoplus_{\oy\in[0,\oc\,]}S(\oy)^{\ox}$.  It follows that the embedding is an equivalence if and only if for all $t=1,\hdots,n$, the $i$-series on arm $t$ has maximum length.  By \ref{i series all}  this holds if and only if every $a_i=1$.\end{proof}

\subsection{Changing Parameters}\label{changing parameters}

Our next main result, \ref{change parameters so coprime}, shows that we can always change parameters, without changing the category of coherent sheaves, so that the condition $(p_i,a_i)=1$ for all $1\le i\le n$ appearing in both \ref{WPL as qgrZ}\eqref{WPL as qgrZ 4} and  \ref{stackminres 2A} holds.

We now fix notation. Let $S\colonequals S_{\bp,\bl}$, and fix a subset $I$ of $\{1,\ldots,n\}$. For each $i\in I$, choose a divisor $d_i$ of $p_i$. Let $\sp_i\colonequals p_i/d_i$, $\bp'\!\!\colonequals (\sp_i\mid i\in I)$, $\bl'\colonequals (\lambda_i\mid i\in I)$,
\[
S'\colonequals S_{\bp'\!,\bl'}=\frac{\KK[\st_0,\st_1,\sx_i\mid i\in I]}{(\sx_i^{\sp_i}-\ell_i(\st_0,\st_1)\mid i\in I)}
\]
and $\bL^{\! \prime}_{\phantom\prime}\colonequals \bL(\sp_i\mid i\in I)=\langle\osx_i,\osc\mid i\in I\rangle/
(\sp_i\osx_i-\osc\mid i\in I)$. Then $S'$ is an $\bL^{\! \prime}$-graded $\KK$-algebra, and there is an equivalence $\coh\bX_{\bp'\!,\bl'}=\qgr^{\bL^{\! \prime}}S'$ as before.

\begin{prop}\label{Veronese trick}
With notation as above,
\begin{enumerate}
\item\label{Veronese trick 1} There is a monomorphism $\iota\colon\bL^{\! \prime}_{\phantom\prime}\to\bL$ of groups sending
$\osx_i$ to $d_i\ox_i$ for each $i\in I$ and $\osc$ to $\oc$.
\item\label{Veronese trick 2}  There is a monomorphism $S'\to S$ of $\KK$-algebras
sending $\sx_i$ to $x_i^{d_i}$ for each $i\in I$ and $\st_j$ to $t_j$ for $j=0,1$, which induces an isomorphism
$S'\simeq \bigoplus_{\osx\in\bL^{\! \prime}_{\phantom\prime}}S_{\iota(\osx)}$.
\item\label{Veronese trick 3} Let $\ox\in\bL$ be an element with normal form
$\ox=\sum_{i\in I}a_i\ox_i+a\oc$ such that $a_i$ is a multiple of $d_i$.  For $\sa_i\colonequals a_i/d_i$ and $\osx\colonequals \sum_{i\in I}\sa_i\osx_i+a\osc\in\bL^{\! \prime}_{\phantom\prime}$, we have $(S')^{\osx}=S^{\ox}$.
\end{enumerate}
\end{prop}
\begin{proof}
(1) Clearly $\iota$ is well-defined. Assume that $\osx\in\bL^{\! \prime}_{\phantom\prime}$ with
normal form $\osx=\sum_{i\in I}a_i\osx_i+a\osc$ belongs to the kernel of $\iota$. Then $0=\iota(\osx)=
\sum_{i\in I}a_id_i\ox_i+a\oc$, where the right
hand side is a normal form in $\bL$, and so $a_i=0=a$
for all $i$. Hence $\osx=0$.\\
(2) Take any element $\osx\in\bL^{\! \prime}_{\phantom\prime}$ with
a normal form $\osx=\sum_{i\in I}a_i\osx_i+a\osc$. We prove $S'_{\osx}\simeq S_{\iota(\osx)}$.
If $\osx\notin\bL^{\! \prime}_+$, then $\iota(\osx)\notin\bL_+$ and both sides are zero.
Assume $\osx\in\bL^{\! \prime}_+$. Then by \ref{basic observation}, $S'_{\osx}$ has a $\KK$-basis
\[
\st_0^{j}\st_1^{a-j}\prod_{i\in I}\sx_i^{a_i}\ \ \ 0\le j\le a.
\]
Since $\iota(\osx)$ has a normal form $\sum_{i\in I}a_id_i\ox_i+a\oc$,
it follows from \ref{basic observation} that $S_{\iota(\osx)}$ has a $\KK$-basis $t_0^{j}t_1^{a-j}\prod_{i\in I}x_i^{a_id_i}$ for $0\le j\le a$. The assertion follows.\\
(3) Immediate from \eqref{Veronese trick 2}.
\end{proof}

\begin{prop}\label{change parameters so coprime}
Suppose that $\ox\in\bL$ is not torsion, and write $\ox=\sum_{i=1}^na_i\ox_i+a\oc\in\bL$ in normal form. Let $I\colonequals \{1\le i\le n\mid a_i\neq0\}$, and consider the parameters $(\bp'\!,\bl')$ defined by  $\bp'\!\colonequals (\sp_i\mid i\in I)$ for $\sp_i\colonequals p_i/(a_i,p_i)$ and $\bl'\colonequals (\lambda_i\mid i\in I)$.  As above, set $\osx\colonequals \sum_{i\in I}\sa_i\osx_i+a\osc\in\bL^{\! \prime}_{\phantom\prime}$, then the following statements hold.
\begin{enumerate}
\item\label{change parameters so coprime 1} There is an isomorphism $S_{\bp,\bl}^{\ox}\cong S_{\bp'\!,\bl'}^{\osx}$ as $\bZ$-graded $\KK$-algebras.
\item\label{change parameters so coprime 2} There are equivalences $\CM^\bZ\!S_{\bp,\bl}^{\ox}\simeq \CM^\bZ\! S_{\bp'\!,\bl'}^{\osx}\simeq \CM^\bL\! S_{\bp'\!,\bl'}$.
\item\label{change parameters so coprime 3} There are equivalences $\qgr^\bZ\!S_{\bp,\bl}^{\ox}\simeq \qgr^\bZ\!S_{\bp'\!,\bl'}^{\osx}\simeq \coh \bX_{\bp'\!,\bl'}$.
\end{enumerate}
\end{prop}
\begin{proof}
Part \eqref{change parameters so coprime 1} follows directly from \ref{Veronese trick}\eqref{Veronese trick 3}. Certainly this induces the left equivalences in \eqref{change parameters so coprime 2} and \eqref{change parameters so coprime 3}. Applying \ref{WPL as qgrZ} to $S_{\bp'\!,\bl'}^{\osx}$ gives the right equivalences in \eqref{change parameters so coprime 2} and \eqref{change parameters so coprime 3}.
\end{proof}

Thus we can always replace  $(\bp,\bl,\ox)$ by $(\bp',\bl',\osx)$ such that $S^{\ox}_{\bp,\bl}=S^{\osx}_{\bp',\bl'}$ and the coprime assumptions in both \ref{WPL as qgrZ}\eqref{WPL as qgrZ 4} and  \ref{stackminres 2A} hold, applied to   $(\bp',\bl',\osx)$. Note also that the above implies that if $\ox\in\bL$ is any non-torsion element, then $\qgr^\bZ\!S_{\bp,\bl}^{\ox}$ always gives the category of coherent sheaves over a weighted projective line, perhaps with different parameters.

\subsection{Algebraic Approach to Special CM Modules}
In this subsection we give an algebraic treatment of the special CM $S^{\ox}$-modules, and show how to determine the rank one special CM modules without using geometric arguments.  Hence this subsection is independent of \S\ref{stack}, and the techniques developed will be used later to obtain geometric corollaries.  Note however that the geometry is required to deduce that there are no higher rank indecomposable special CM modules; this algebraic approach seems only to be able to deal with the rank one specials.

Consider $\bX_{\bp,\bl}$ and let $\ox\in\bL$ be an element with normal form $\ox=\sum_{i=1}^na_i\ox_i+a\oc$ with $a\ge0$.  By \ref{change parameters so coprime} we can assume, by changing parameters if necessary, that $(a_i,p_i)=1$ for all $1\le i\le n$.  Then, by \ref{WPL as qgrZ}, there is an equivalence
\[
(-)^{\ox}\colon\CM^\bL\!S\to\CM^\bZ\!S^{\ox}.
\]
Below we will often use the identification 
\begin{equation}\label{Hom between Ss}
S_{\oy-\ox}\cong\Hom_S^{\bL}(S(\ox),S(\oy))
\end{equation}
for any $\ox,\oy\in\bL$. Recall that the \emph{AR translation functor} of $S^{\ox}$ is given by
\[
\tau_{S^{\ox}}\colonequals \Hom_{S^{\ox}}(-,\omega_{S^{\ox}})\circ\Hom_{S^{\ox}}(-,S^{\ox})\colon\CM^{\bZ}\!S^{\ox}\to\CM^{\bZ}\!S^{\ox},
\]
where $\omega_{S^{\ox}}$ is the $\bZ$-graded canonical module of $S^{\ox}$ \cite{AR, IT}.

\begin{prop}\label{AR prep}
With the setup as above, the following statements hold.
\begin{enumerate}
\item\label{AR prep 1} There is an isomorphism $\omega_{S^{\ox}}\cong S(\ow)^{\ox}$.
\item\label{AR prep 2} There is a commutative diagram
\begin{equation}\label{Veronese equivalence}
\begin{array}{c}
\begin{tikzpicture}
\node (top 1) at (0,0) {$\CM^\bL\!S$};
\node (top 2) at (2.5,0) {$\CM^\bZ\!S^{\ox}$};
\node (bottom 1) at (0,-1.5) {$\CM^\bL\!S$};
\node (bottom 2) at (2.5,-1.5) {$\CM^\bZ\!S^{\ox}$};
\draw[->] (top 1) -- node[left] {$\scriptstyle \tau_S\colonequals (\ow)$} (bottom 1);
\draw[->] (top 2) -- node[right] {$\scriptstyle \tau_{S^{\ox}}$} (bottom 2);
\draw[->] (top 1) -- node[above] {$\scriptstyle (-)^{\ox}$} (top 2);
\draw[->] (bottom 1) -- node[above] {$\scriptstyle (-)^{\ox}$} (bottom 2);
\end{tikzpicture}
\end{array}
\end{equation}
\end{enumerate}
\end{prop}
\begin{proof}
Again, to ease notation write $R\colonequals S^{\ox}$.\\
(1) Taking a projective resolution of $\KK$ in $\mod^{\bL}\!S$, applying $\Hom_S(-,S(\ow))$ and using \ref{slightly stronger} we see that $\Ext^i_R(\KK,S(\ow)^{\ox})=\Ext^i_S(\KK,S(\ow))^{\ox}$.  This is $\KK$ for $i=2$ and zero for $i\neq2$ \cite{BHe}. Thus $S(\ow)^{\ox}$ is the $\bZ$-graded canonical module of $R$.\\
(2) Let $X\in\CM^{\bL}\!S$. Using \eqref{AR prep 1} and \ref{slightly stronger}, 
\[
\tau_R(X^{\ox})=\Hom_R(\Hom_R(X^{\ox},S^{\ox}),S(\ow)^{\ox})=\Hom_S(\Hom_S(X,S),S(\ow))^{\ox}=X(\ow)^{\ox}.\qedhere
\]

\vspace{-0.85cm}

\end{proof}

\vspace{0.3cm}

The following gives an algebraic criterion for certain CM $S^{\ox}$-modules to be special.

\begin{lemma}\label{criterion for special}
For $\oy\in\bL$, the CM $S^{\ox}$-module $S(\oy)^{\ox}$ is special
if and only if
\begin{equation}
S_{\oy+\ow+\ell\ox}=\sum_{m\in\bZ}S_{\ow+m\ox}\cdot S_{\oy+(\ell-m)\ox}\label{specials equation}
\end{equation}
holds for all $\ell\in\bZ$.
\end{lemma}
\begin{proof}
Set $R\colonequals S^{\ox}$ and as above write $\tau_R\colon\CM^\bZ\!R\simeq\CM^\bZ\!R$ for the AR-translation. If $\overline{\CM}^{\bZ}\!R$ is the quotient category  of $\CM^\bZ\!R$ by the ideal generated by $\{\omega_R(\ell)\mid \ell\in\bZ\}$, this yields AR duality
\begin{equation}\label{AR duality}
D\Ext^1_{\mod^{\bZ}\!R}(X,Y)\simeq\Hom_{\overline{\CM}^{\bZ}\!R}(Y,\tau_RX)
\end{equation}
for any $X,Y\in\CM^\bZ\!R$ \cite{AR, IT}.   By \ref{AR prep}\eqref{AR prep 1}, $S(\ow+\ell\ox)^{\ox}=\omega_R(\ell)$ holds, and hence there is an induced equivalence
\begin{equation}\label{Veronese equivalence2}
(-)^{\ox}\colon (\CM^\bL\!S)/I\simeq\overline{\CM}^{\bZ}\!R
\end{equation}
for the ideal $I$ of the category $\CM^\bL\!S$ generated by $\add\{S(\ow+\ell\ox)\mid\ell\in\bZ\}$. It follows that
\begin{eqnarray*}
D\Ext^1_R(S(\oy)^{\ox},R)&=&
\bigoplus_{\ell\in\bZ}D\Ext^1_{\mod^{\bZ}\!R}(S(\oy)^{\ox},R(\ell))\\
&\stackrel{\eqref{AR duality}}{\simeq}&
\bigoplus_{\ell\in\bZ}\Hom_{\overline{\CM}^{\bZ}\!R}(R,(\tau_R(S(\oy)^{\ox}))(\ell))\\
&\stackrel{\eqref{Veronese equivalence}\eqref{Veronese equivalence2}}{\simeq}&
\bigoplus_{\ell\in\bZ}\frac{\Hom_{\CM^\bL\!S}(S,S(\oy+\ow+\ell\ox))}{I(S,S(\oy+\ow+\ell\ox))}.
\end{eqnarray*}
Thus $S(\oy)^{\ox}$ is special if and only if
$\Hom_{\CM^\bL\!S}(S,S(\oy+\ow+\ell\ox))=I(S,S(\oy+\ow+\ell\ox))$ holds for all $\ell\in\bZ$.
Since $\Hom_{\CM^\bL\!S}(S,S(\oy+\ow+\ell\ox))=S_{\oy+\ow+\ell\ox}$
and $I(S,S(\oy+\ow+\ell\ox))=\sum_{m\in\bZ}S_{\ow+m\ox}\cdot S_{\oy+(\ell-m)\ox}$ hold by \eqref{Hom between Ss}, the assertion follows.
\end{proof}

We will also require the next result, which is much more elementary, and follows from \ref{basic observation}. 
\begin{lemma}[{\cite{GL1}}]\label{basic observation B}
Suppose that $\ox\in\bL$ has normal form  $\ox=\sum_{i=1}^na_i\ox_i+a\oc$.
\begin{enumerate}
\item\label{basic observation B 3} If $\oy\in\bL_+$ and $\ox-\oy\in\bL_+$, write $\oy=\sum_{i=1}^nb_i\ox_i+b\oc$
in normal form.  Then for $I\colonequals \{1\le i\le n\mid a_i<b_i\}$, 
\[
\ox\ge|I|\oc\quad\mbox{ and }\quad S_{\oy}\cdot S_{\ox-\oy}=(\prod_{i\in I}x_i^{p_i})S_{\ox-|I|\oc}.
\]
\item\label{basic observation B 4} 
Let $X,Y$ be a basis of $S_{\oc}$. If $\ox\ge i\oc\ge0$, then 
\[
S_{\ox}=XS_{\ox-\oc}+f(X,Y)S_{\ox-i\oc}
\]
for any $f(X,Y)\in S_{i\oc}$ which is not a multiple of $X$. 
\end{enumerate}
\end{lemma}

Before proving the main result \ref{specials lag approach thm}, we first illustrate a special case.

\begin{example}{\rm
Let $\os_a=\sum_{i=1}^n\ox_i+a\oc$ with $a\ge0$ and $n+a\ge2$ (since $a\geq 0$, the last condition is equivalent to $\os_a\notin[0,\oc\,]$).  Then $S(\oy)^{\os_a}$ is a special CM $S^{\os_a}$-module for all $\oy\in[0,\oc\,]$.}
\end{example}
\begin{proof}
We use \ref{criterion for special}.  When $\ell\le0$, both sides of \eqref{specials equation} are zero.  When $\ell>0$, since $\ow+\os_a=(n-2+a)\oc$ we have
\[
S_{\oy+\ow+\ell\os_a}=S_{\oy+(\ell-1)\os_a+(n-2+a)\oc}
\stackrel{\mbox{\scriptsize\ref{basic observation}\eqref{basic observation 3}}}{=}S_{(n-2+a)\oc}\cdot S_{\oy+(\ell-1)\os_a}=S_{\ow+\os_a}\cdot S_{\oy+(\ell-1)\os_a}
\]
and so \eqref{specials equation} holds.
\end{proof}

The following is the main result in this section.  The algebraic method of proof describes all the rank one indecomposable special CM modules directly, and the geometry is only required to verify that there are no further indecomposable special CM modules of higher rank.   The algebraic method of proof developed below feeds back into the geometry, and allows us to extract the middle self-intersection number in \ref{middle SI number}.  As notation, we write $\SCM^{\bZ}\! S^{\ox}$ for  those special CM $S^{\ox}$-modules that are $\mathbb{Z}$-graded.

\begin{thm}\label{specials lag approach thm}
Let $\ox\in\bL_+$ with $\ox\notin[0,\oc\,]$.  Write $\ox=\sum_{i=1}^na_i\ox_i+a\oc$ in normal form, then the following statements hold.
\begin{enumerate}
\item\label{specials lag approach thm 1} Up to degree shift, the indecomposable objects in $\SCM^{\bZ}\!S^{\ox}$ are precisely those $S(u\ox_j)^{\ox}$ with $1\le j\le n$ and $u\in I(p_j,p_j-a_j)$.
\item\label{specials lag approach thm 2} Forgetting the grading, $\add\{S(u\ox_j)^{\ox}\mid j\in [1,n],  u\in I(p_j,p_j-a_j)   \}= \SCM S^{\ox}$.
\end{enumerate}
In particular, $S^{\ox}$, $S(\oc)^{\ox}$ and $S((p_j-a_j)\ox_j)^{\ox}$ for all $j\in[1,n]$ are always special.
\end{thm}

\begin{proof}
We only prove \eqref{specials lag approach thm 1}, since the other statements follow immediately.  By \ref{change parameters so coprime}\eqref{change parameters so coprime 1} we can assume that $(a_i,p_i)=1$ for all $1\le i\le n$.  Write $R\colonequals S^{\ox}$.

(a) We first claim that, up to degree shift, $\bZ$-graded special  CM $R$-modules of rank one must have the form $S(u\ox_j)^{\ox}$ for some $1\le j\le n$ and $0\le u\le p_j$.  

By \ref{L-factorial} $S$ is an $\bL$-graded factorial domain, so all rank one objects in $\CM^{\bL}\!S$ have the form $S(\oy)$ for some $\oy\in\bL$. Under the rank preserving equivalence \ref{WPL as qgrZ}\eqref{WPL as qgrZ 1}, it follows that all rank one objects in $\CM^{\bZ}\!R$ have the form $S(\oy)^{\ox}$ for some $\oy\in\bL$.  Since we are working up to degree shift, and $\ox\ge0$, we can assume without loss of generality that $\oy\ge0$ and $\oy\ {\not\ge}\ \ox$, by, if necessary, replacing $\oy$ by $\oy-\ell\ox$ for some $\ell\in\bZ$.

Hence we can assume that our rank one special CM module has the form $S(\oy)^{\ox}$ with $\oy\geq 0$ and $\oy\ {\not\ge}\ \ox$.  
Now assume that $\oy$ can not be written as $u\ox_j$ for some $1\le j\le n$ and $0\le u\le p_j$. Then there exists $j\neq k$ such that $\oy\ge\ox_j+\ox_k$. By applying \ref{criterion for special} for $\ell=0$, it follows that
\[
S_{\oy+\ow}=\sum_{m\in\bZ}S_{\ow+m\ox}\cdot S_{\oy-m\ox}.
\]
Now $S_{\oy+\ow}\neq0$ by our assumption $\oy\ge\ox_j+\ox_k$,  hence there exists $m\in\bZ$ such that $S_{\ow+m\ox}\neq0$ and $S_{\oy-m\ox}\neq0$.  On one hand, since $\ow\ {\not\ge}\ 0$, this implies that $m>0$. On the other hand, since $\oy\ {\not\ge}\ \ox$, this implies that $m\leq 0$, a contradiction. Thus the rank one special CM modules have the claimed form $S(u\ox_j)^{\ox}$.

(b) Let $1\le j\le n$ and $0\le u\le p_j$.  We now show that $S(u\ox_j)^{\ox}$ is a special CM $R$-module if and only
if $u\in I(p_j,p_j-a_j)$.  By \ref{criterion for special}, the CM $R$-module $S(u\ox_j)^{\ox}$ is special if and only if
\begin{equation}\label{special for ux_j}
S_{u\ox_j+\ow+\ell\ox}=
\sum_{m\in\bZ}S_{\ow+m\ox}\cdot S_{u\ox_j+(\ell-m)\ox}
\end{equation}
holds for all $\ell\in\bZ$, or equivalently, for all $\ell>0$  since
the left hand side vanishes for $\ell\le0$ (in that case we have $u\ox_j+\ow\le \oc+\ow\ {\not\ge}\ 0$). Thus in what follows, we fix an arbitrary $\ell> 0$.  

Clearly equality holds in \eqref{special for ux_j} if and only if $\subseteq$ holds.
To simplify notation, for $m\in\bZ$ write
\begin{align*}
\osx&\colonequals u\ox_j+\ow+\ell \ox\\
\osym&\colonequals \ow+m\ox.
\end{align*}
Then $S_{\osym}\cdot S_{\osx-\osym}=S_{\ow+m\ox}\cdot S_{u\ox_j+(\ell-m)\ox}$. Notice that $\osym,\,\osx-\osym\in\bL_+$ holds if and only if $1\le m\le\ell$.
The `if' part follows easily from $\ox\notin[0,\oc\,]$, and the `only if' part follows from $\ow\ {\not\ge}\ 0$ and $u\ox_j-\ox\le\oc-\ox\ {\not\ge}\ 0$.  Thus \eqref{special for ux_j} holds if and only if

\begin{equation}\label{special for ux_j easy}
S_{\osx}\subseteq
\sum_{m=1}^{\ell}S_{\osym}\cdot S_{\osx-\osym}
\end{equation}
holds.  Note that $\osx$ and $\osym$ can be written more explicitly as 
\begin{equation}
\left.\begin{array}{l}
\osx=\left( \sum_{i\neq j}(\ell a_i-1)\ox_i \right)+(u+\ell a_j-1)\ox_j+(n-2+a\ell)\oc\\
\osym=\sum_{i=1}^n(m a_i-1)\ox_i +(n-2+am)\oc.
\end{array}\right\}\label{forms of x and y}
\end{equation}
Since $\osym, \,\osx-\osym\in\bL_+$ for each $1\le m\le\ell$, \ref{basic observation B}\eqref{basic observation B 3} implies that
\begin{equation}
S_{\osym}\cdot S_{\osx-\osym}=
(\prod_{i\in I_{m}}x_i^{p_i})S_{\osx-|I_{m}|\oc},\label{simple product}
\end{equation}
where $I_m$ is the set $I$ in \ref{basic observation B}\eqref{basic observation B 3} for $\osx$ and $\osym$.  
As before, for an integer $k$, we write $[k]_{p_i}$ for the integer $k'$ satisfying $0\le k'\le p_i-1$ and $k-k'\in p_i\bZ$. Simply writing out $\osx$ and $\osym$ into normal form, from \eqref{forms of x and y} we see that 
\begin{equation}\label{I set}
I_{m}=\{1\le i\le n\mid [u_i+\ell a_i-1]_{p_i}<[ma_i-1]_{p_i}\}
\end{equation}
where $u_i\colonequals u$ if $i=j$ and $u_i\colonequals 0$ otherwise. 

For the case $m=\ell$, it is clear that $I_{\ell}\subseteq\{j\}$.  Hence  we see that 
\begin{equation}
S_{\osyl}\cdot S_{\osx-\osyl}\stackrel{\eqref{simple product}}{=}
(\prod_{i\in I_{\ell}}x_i^{p_i})S_{\osx-|I_{\ell}|\oc}
\supseteq x_j^{p_j}S_{\osx-\oc}.\label{ell case product}
\end{equation}
Now we claim that \eqref{special for ux_j easy} holds
if and only if $j\notin I_m$ for some $1\le m\le\ell$.\\
($\Rightarrow$) Assume that  \eqref{special for ux_j easy} holds.  If further $j\in I_m$ for all $1\le m\le\ell$, then using
\[
S_{\osx}
\stackrel{\mbox{\scriptsize{\eqref{special for ux_j easy}}}}{=}
\sum_{m=1}^{\ell}S_{\osym}\cdot S_{\osx-\osym}
\stackrel{\mbox{\scriptsize{\eqref{simple product}}}}{=}
\sum_{m=1}^{\ell}(\prod_{i\in I_{m}}x_i^{p_i})S_{\osx-|I_{m}|\oc}
\]
we see that $x_j^{p_j}$ divides every element in $S_{\osx}$.  This gives a contradiction, since we can use the normal form of $\osx$ to obtain elements of $S_{\osx}$ which are not divisible by $x_j^{p_j}$ .\\
($\Leftarrow$)  Suppose that $j\notin I_m$ for some $1\le m\le\ell$.  Since $\ox\ge|I_m|\oc\ge0$ holds by \ref{basic observation B}\eqref{basic observation B 3}, we have
\[
S_{\osx}=x_j^{p_j}S_{\osx-\oc}+(\prod_{i\in I_{m}}x_i^{p_i})S_{\osx-|I_{m}|\oc}.
\]
by choosing $X\colonequals x_j^{p_j}$ and $f(X,Y)\colonequals \prod_{i\in I_{m}}x_i^{p_i}$ in \ref{basic observation B}\eqref{basic observation B 4}.
Finally, using \eqref{simple product} and \eqref{ell case product} this gives
\[
S_{\osx}\subseteq S_{\osyl}\cdot S_{\osx-\osyl}+S_{\osym}\cdot S_{\osx-\osym},
\]
which clearly implies \eqref{special for ux_j easy}.

Consequently, \eqref{special for ux_j easy} holds if and only if
$j\notin I_m$ for some $1\leq m\leq \ell$, which by \eqref{I set} holds if and only if $[u+\ell a_j-1]_{p_j}\ge[ma_j-1]_{p_j}$ for some $1\le m\le\ell$.  By \ref{i-series region}, this holds if and only if $u\in I(p_j,p_j-a_j)$, proving claim (b).

(c) We now prove part \eqref{specials lag approach thm 1}. Combining (a) and (b), it suffices to show that there is no indecomposable object $X$ in $\SCM^{\bZ}\! R$ with rank bigger than one. Otherwise, by \cite[15.2.1]{Y}, $\widehat{X}$ is an indecomposable object in $\SCM\mathfrak{R}$ with rank bigger than one, where $\mathfrak{R}$ is the completion of $R$.  This is a contradiction to \ref{Wunram main} and \ref{fund is reduced}, and so Part (1) follows.
\end{proof}

\subsection{The Middle Self-Intersection Number}
In this subsection we use the techniques of the previous subsections to determine the middle self-intersection number in \eqref{key dual graph}. This requires the following two elementary but technical lemmas.

\begin{lemma}\label{basis of Sn-1}
Let $\ell_1,\ldots,\ell_m$ be elements in $S_{\oc}$ such that any two elements are linearly independent. Then $\prod_{j\neq1}\ell_j,\ldots,\prod_{j\neq m}\ell_j$ is a basis of $S_{(m-1)\oc}$.
\end{lemma}
\begin{proof}
Assume that the assertion holds for $m-1$. Then
$\prod_{j\neq1}\ell_j,\ldots,\prod_{j\neq m-1}\ell_j$ gives a basis of $\ell_mS_{(m-2)\oc}$.
Since $S_{(m-1)\oc}=\ell_mS_{(m-2)\oc}+\KK\prod_{j\neq m}\ell_j$ holds, the assertion also holds for $m$.
\end{proof}

The following lemma is general, and does not require $n>0$.
\begin{lemma}\label{decompose Stx-c}
Let $\ox\in\bL_+$, and write $\ox=\sum_{i=1}^na_i\ox_i+a\oc$ in normal form.
If  $t\ge2$, then every morphism in $\Hom_S^{\bL}(S(\oc),S(t\ox))$ factors through $\add\{S(\ox+(p_i-a_i)\ox_i)\mid1\le i\le n\}$.
\end{lemma}
\begin{proof}
It suffices to show that
\[
S_{t\ox-\oc}\subset\sum_{i=1}^nS_{\ox-a_i\ox_i}\cdot S_{(t-1)\ox-(p_i-a_i)\ox_i}.
\]
For each $i$ with $1\le i\le n$, take $m_i\ge0$ and $\upvarepsilon_i\in\{0,1\}$ such that
\[
(t-1)a_i=[(t-1)a_i]_{p_i}+m_ip_i\ \mbox{ and }\ ta_i=[ta_i]_{p_i}+(m_i+\upvarepsilon_i)p_i.
\]
Let $m\colonequals \sum_{i=1}^nm_i$ and $\upvarepsilon\colonequals \sum_{i=1}^n\upvarepsilon_i$. Then the equality
\[
t\ox-\oc=\sum_{j=1}^n[ta_j]_{p_j}\ox_j+(m+\upvarepsilon-1+ta)\oc
\] 
implies that
\begin{equation}\label{tx-c}
S_{t\ox-\oc}=(\prod_{j=1}^nx_j^{[ta_j]_{p_j}})S_{(m+\upvarepsilon-1+ta)\oc}.
\end{equation}
Similarly the equality 
\[
(t-1)\ox-(p_i-a_i)\ox_i=[ta_i]_{p_i}\ox_i+\sum_{j\neq i}[(t-1)a_j]_{p_j}\ox_j+(m+\upvarepsilon_i-1+(t-1)a)\oc
\]
implies that
\[
S_{(t-1)\ox-(p_i-a_i)\ox_i}=x_i^{[ta_i]_{p_i}}(\prod_{j\neq i}x_j^{[(t-1)a_j]_{p_j}})S_{(m+\upvarepsilon_i-1+(t-1)a)\oc}.
\]
Multiplying $S_{\ox-a_i\ox_i}=(\prod_{j\neq i}x_j^{a_j})S_{a\oc}$ and using $[(t-1)a_j]_{p_j}+a_j=[ta_j]_{p_j}+\upvarepsilon_jp_j$ gives
\begin{equation}\label{product of S}
S_{\ox-a_i\ox_i}\cdot S_{(t-1)\ox-(p_i-a_i)\ox_i}=(\prod_{j=1}^nx_j^{[ta_j]_{p_j}})(\prod_{j\neq i}x_j^{\upvarepsilon_jp_j})S_{a\oc}\cdot S_{(m+\upvarepsilon_i-1+(t-1)a)\oc}.
\end{equation}
Now set $I\colonequals \{1\le i\le n\mid \upvarepsilon_i=1\}$.  Clearly $|I|=\upvarepsilon$ holds.

First we assume $I\neq\emptyset$. By \ref{basis of Sn-1} we have $\sum_{i\in I}\KK\prod_{j\neq i}x_j^{\upvarepsilon_jp_j}=S_{(\upvarepsilon-1)\oc}$ and thus
\begin{eqnarray*}
\sum_{i\in I}S_{\ox-a_i\ox_i}\cdot S_{(t-1)\ox-(p_i-a_i)\ox_i}&\stackrel{\eqref{product of S}}{=}&(\prod_{j=1}^nx_j^{[ta_j]_{p_j}})S_{(\upvarepsilon-1)\oc}\cdot S_{a\oc}\cdot S_{(m+(t-1)a)\oc}\\
&=&(\prod_{j=1}^nx_j^{[ta_j]_{p_j}})S_{(m+\upvarepsilon-1+ta)\oc}\\
&\stackrel{\eqref{tx-c}}{=}&S_{t\ox-\oc},
\end{eqnarray*}
as desired.

Next we assume $I=\emptyset$.
If further $m-1+(t-1)a\ge0$, then \eqref{product of S} is equal to
\[
(\prod_{j=1}^nx_j^{[ta_j]_{p_j}})S_{a\oc}\cdot S_{(m-1+(t-1)a)\oc}=(\prod_{j=1}^nx_j^{[ta_j]_{p_j}})S_{(m-1+ta)\oc}\stackrel{\eqref{tx-c}}{=}S_{t\ox-\oc},
\]
as desired, so we can assume that $m-1+(t-1)a<0$. But $a\ge0$ since $\ox\in\bL_+$, and $t\ge2$ by assumption, so necessarily $m=0=a$.
Then $m-1+ta<0$ holds, so $S_{t\ox-\oc}=0$ by \eqref{tx-c}, which implies the assertion.
\end{proof}

The following is the main result of this subsection; the main point is that the manipulations above involving the combinatorics of the weighted projective line give the geometric corollary in \ref{middle SI number} below.

\begin{thm}\label{number of arrows}
Let $\ox\in\bL_+$ with $\ox\notin[0,\oc\,]$, and write $\ox=\sum_{i=1}^na_i\ox_i+a\oc$ in normal form. Set $R\colonequals S^{\ox}$ and $N\colonequals S(\oc)^{\ox}$, and consider their completions $\mathfrak{R}$ and $\widehat{N}$.  Then in the quiver of the reconstruction algebra of $\mathfrak{R}$, the number of arrows from $\widehat{N}$ to $\mathfrak{R}$ is $a$.
\end{thm}
\begin{proof}
By \ref{change parameters so coprime}\eqref{change parameters so coprime 2},  $S_{\bp,\bl}^{\ox}\cong S_{\bp'\!,\bl'}^{\osx}$ as $\bZ$-graded algebras, where $\osx\colonequals \sum_{i\in I}\sa_i\osx_i+a\osc\in\bL^{\! \prime}_{\phantom\prime}$ satisfies the condition in \ref{WPL as qgrZ}\eqref{WPL as qgrZ 4}.  Note that this change in parameters has not changed the value $a$ on $\oc$, hence in what follows, we can assume that $\CM^\bL\!S\simeq\CM^\bZ\!R$ holds, via the functor $(-)^{\ox}$.

Let $\cC$ be the full subcategory of $\CM^{\bL}\!S$ corresponding to $\SCM^\bZ\!R$ via the functor $(-)^{\ox}$.
Then the number of arrows from $\widehat{N}$ to $\mathfrak{R}$ is equal to the dimension of the $\KK$-vector space
\[
\frac{{\rm rad}_{\SCM\mathfrak{R}}(\widehat{N},\mathfrak{R})}{{\rm rad}^2_{\SCM\mathfrak{R}}(\widehat{N},\mathfrak{R})}
\cong\prod_{t\in\bZ}\frac{\Hom_R^{\bZ}(N,R(t))}{{\rm rad}^2_{\SCM^{\bZ}\!R}(N,R(t))}\cong\prod_{t\in\bZ}\frac{\Hom_S^{\bL}(S(\oc),S(t\ox))}{{\rm rad}^2_{\cC}(S(\oc),S(t\ox))}.
\]
By \ref{specials lag approach thm}, $\cC$ is the additive closure of $S(u\ox_j+s\ox)$, where $s\in\bZ$, $1\le j\le n$ and $u\in I(p_j,p_j-a_j)$. We split into three cases.
\begin{enumerate}
\item  If $t\le 0$, then $\Hom_S^{\bL}(S(\oc),S(t\ox))=0$.

\item If $t\ge2$, then since $S(\ox+(p_i-a_i)\ox_i)$ belongs to $\cC$ by \ref{specials lag approach thm}, and is not isomorphic to both $S(\oc)$ and $S(t\ox)$ in $\mod^{\bL}\!S$ (since $\ox\notin[0,\oc\,]$), we have $\Hom_S^{\bL}(S(\oc),S(t\ox))={\rm rad}^2_{\cC}(S(\oc),S(t\ox))$ by \ref{decompose Stx-c}.

\item Suppose that $t=1$.  By definition any morphism in ${\rm rad}^2_{\cC}(S(\oc),S(\ox))$ can be written as
a sum of compositions $S(\oc)\to S(u\ox_j+s\ox)\to S(\ox)$.
If $s\le0$, then $\Hom_S^{\bL}(S(\oc),S(u\ox_j+s\ox))\cong S_{u\ox_j+s\ox-\oc}$, and hence ${\rm rad}_{\cC}(S(\oc),S(u\ox_j+s\ox))=0$.  If $s\ge1$, then  $\Hom_S^{\bL}(S(u\ox_j+s\ox),S(\ox))\cong S_{(1-s)\ox-u\ox_j}$, and hence ${\rm rad}_{\cC}(S(u\ox_j+s\ox),S(\ox))=0$.  Either way, ${\rm rad}^2_{\cC}(S(\oc),S(t\ox))=0$ in this case.
\end{enumerate}
Combining all cases, the desired number is thus
\[
\sum_{t\in\bZ}\dim_{\KK}\left(\frac{\Hom_S^{\bL}(S(\oc),S(t\ox))}{{\rm rad}^2_{\cC}(S(\oc),S(t\ox))}\right)=\dim_{\KK}\Hom_S^{\bL}(S(\oc),S(\ox))=\dim_{\KK}S_{\ox-\oc}
\stackrel{\scriptsize\mbox{\ref{basic observation}}}{=}a.\qedhere
\]

\vspace{-1.1cm}

\end{proof}

\vspace{0.6cm}

This allows us to finally complete the proof of \ref{dual graph general intro} from the introduction.
\begin{cor}\label{middle SI number}
Let $\ox\in\bL_+$ with $\ox\notin[0,\oc\,]$, and write $\ox=\sum_{i=1}^na_i\ox_i+a\oc$ in normal form. Then the morphism $\uppi\colon Y^{\ox}\to\Spec S^{\ox}$ is the minimal resolution, and its dual graph is precisely \eqref{key dual graph} with $\upbeta=a+v=a+\#\{i\mid a_i\neq0\}$.
\end{cor}
\begin{proof}
We know from \ref{when Y min res} that $\uppi$ is the minimal resolution, and we know from construction of $Y^{\ox}$ that all the self-intersection numbers are determined by the continued fraction expansions (\S\ref{iseries section}), except the middle curve $E_i$ corresponding to the special CM module $S(\oc)^{\ox}$.   The dual graph does not change under completion. By \ref{number of arrows} the number of arrows in the reconstruction algebra from the middle vertex to the vertex $\circ$ is $a$.  Thus the calculation \eqref{Zf dot Ei} combined with \ref{GL2 for R} shows that $a=-E_i\cdot Z_f=\upbeta-v$.
\end{proof}

\subsection{The Reconstruction Algebra and its $\qgr$}
Using the above subsections, we next describe the quiver of the reconstruction algebra and determine the associated $\qgr$ category.  Consider the dual graph \eqref{key dual graph}, then with the convention that we only draw the arms that are non-empty, we see from \ref{fund is reduced}, \eqref{Zf dot Ei} and $Z_K\cdot E_i=E_i^2+2$ that
\begin{equation}\label{ZfZk dot Ei}
\begin{array}{cc}
((Z_K-Z_f)\cdot E_i)_{i}=&
\begin{array}{c}
\begin{tikzpicture}[xscale=0.75,yscale=0.75]
\node (0) at (0,0) {$\scriptstyle 2-v$};
\node (A1) at (-3,1) {$\scriptstyle 0$};
\node (A2) at (-3,2) {$\scriptstyle 0$};
\node (A3) at (-3,3) {$\scriptstyle 0$};
\node (A4) at (-3,4) {$\scriptstyle 1$};
\node (B1) at (-1.5,1) {$\scriptstyle 0$};
\node (B2) at (-1.5,2) {$\scriptstyle 0$};
\node (B3) at (-1.5,3) {$\scriptstyle 0$};
\node (B4) at (-1.5,4) {$\scriptstyle 1$};
\node (C1) at (0,1) {$\scriptstyle 0$};
\node (C2) at (0,2) {$\scriptstyle 0$};
\node (C3) at (0,3) {$\scriptstyle 0$};
\node (C4) at (0,4) {$\scriptstyle 1$};
\node (n1) at (2,1) {$\scriptstyle 0$};
\node (n2) at (2,2) {$\scriptstyle 0$};
\node (n3) at (2,3) {$\scriptstyle 0$};
\node (n4) at (2,4) {$\scriptstyle 1$};
\node at (-3,2.6) {$\vdots$};
\node at (-1.5,2.6) {$\vdots$};
\node at (0,2.6) {$\vdots$};
\node at (2,2.6) {$\vdots$};
\node at (1,3.5) {$\hdots$};
\node at (1,1.5) {$\hdots$};
\node (T) at (0,4.25) {};
\draw (A1) -- (0);
\draw (B1) -- (0);
\draw (C1) -- (0);
\draw (n1) -- (0);
\draw (A2) -- (A1);
\draw (B2) -- (B1);
\draw (C2) -- (C1);
\draw (n2) -- (n1);
\draw (A4) -- (A3);
\draw (B4) -- (B3);
\draw (C4) -- (C3);
\draw (n4) -- (n3);
\end{tikzpicture}
\end{array}
\end{array}
\end{equation}
Note that the cases $v=0$ and $v=1$ are degenerate, and are already well understood \cite{WemA}. Therefore in the next result, we only consider the case $v\geq 2$. 

Inspecting the list of special CM $S^{\ox}$-modules in \ref{specials determined thm}, the conditions in \ref{Gamma convention remark} are satisfied, so we consider the particular choice of reconstruction algebra
\[
\Upgamma_{\ox}\colonequals \End_{S^{\ox}}(M^{\ox})\quad\mbox{where}\quad M\colonequals S\oplus(\bigoplus_{j\in [1,n], u} S(u\ox_j))\oplus S(\oc),
\]
and $u$ in the middle direct sum ranges over $I(p_j,p_j-a_j)\backslash \{0,p_j\}$.  Since the above $S^{\ox}$-modules are clearly $\bZ$-graded, this induces a $\bZ$ grading on $\Upgamma_{\ox}$.

\begin{cor}\label{recon quiver and number relations}
For $\ox\in\bL_+$ with $\ox\notin[0,\oc\,]$, write $\ox=\sum_{i=1}^na_i\ox_i+a\oc$ in normal form.  For each $i$ with $a_i\neq 0$, as before $m_i$ is defined via $\frac{p_i}{p_i-a_i}=[\upalpha_{i1},\hdots,\upalpha_{im_i}]$, and if $a_i=0$ set $m_i=0$.  Suppose that $v=\#\{i\mid a_i\neq 0\}$ satisfies $v\geq 2$. Then the reconstruction algebra  $\Upgamma_{\ox}$ can be presented as a quiver with relations, where the relations are homogeneous with respect to the natural grading, and the quiver is the following: we first  consider the double quiver of the dual graph \eqref{key dual graph} and add an extending vertex (denoted $\begin{tikzpicture} \node at (0,0) [cvertex] {};\end{tikzpicture}$) as follows:
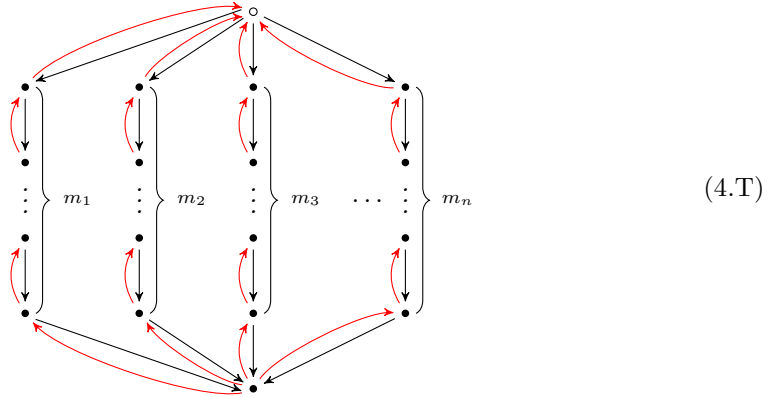
\begin{equation}\label{recon quiver}
\begin{array}{c}
\\
\begin{tikzpicture}[xscale=1,yscale=1,bend angle=30, looseness=1]
\node (0) at (0,0) [vertex] {};
\node (A1) at (-3,1) [vertex] {};
\node (A2) at (-3,2) [vertex] {};
\node (A3) at (-3,3) [vertex] {};
\node (A4) at (-3,4) [vertex] {};
\node (B1) at (-1.5,1) [vertex] {};
\node (B2) at (-1.5,2) [vertex] {};
\node (B3) at (-1.5,3) [vertex] {};
\node (B4) at (-1.5,4) [vertex] {};
\node (C1) at (0,1) [vertex] {};
\node (C2) at (0,2) [vertex] {};
\node (C3) at (0,3) [vertex] {};
\node (C4) at (0,4) [vertex] {};
\node (n1) at (2,1) [vertex] {};
\node (n2) at (2,2) [vertex] {};
\node (n3) at (2,3) [vertex] {};
\node (n4) at (2,4) [vertex] {};
\node at (-3,2.6) {$\vdots$};
\node at (-1.5,2.6) {$\vdots$};
\node at (0,2.6) {$\vdots$};
\node at (2,2.6) {$\vdots$};
\node at (1.5,2.5) {$\hdots$};
\node (T) at (0,5) [cvertex] {};
\draw [->] (A1)+(-30:4.5pt) -- ($(0) + (190:4.5pt)$);
\draw [->] (B1) --(0);
\draw [->] (C1) --(0);
\draw [->] (n1) --(0);
\draw [->] (A2) -- (A1);
\draw [->] (B2) --(B1);
\draw [->] (C2) --(C1);
\draw [->] (n2) -- (n1);
\draw [->] (A4) -- (A3);
\draw [->] (B4) --(B3);
\draw [->] (C4) --(C3);
\draw [->] (n4) --  (n3);
\draw [->] (T)+(-190:4.5pt) -- ($(A4) + (30:4.5pt)$);
\draw [->] (T) -- (B4);
\draw [->] (T) -- (C4);
\draw [->] (T) -- (n4);
\draw [bend right, bend angle=10, looseness=0.5, <-, red] (A1)+(-50:4.5pt) to ($(0) + (200:4.5pt)$);
\draw [bend right, bend angle=10, looseness=0.5, <-, red] (B1)+(-50:4.5pt) to ($(0) + (160:4.5pt)$);
\draw [bend right, <-, red](C1) to (0);
\draw [bend right, bend angle=10, looseness=0.5, <-, red](n1) to (0);
\draw [bend right, <-, red] (A2) to (A1);
\draw [bend right, <-, red] (B2) to (B1);
\draw [bend right, <-, red] (C2) to (C1);
\draw [bend right, <-, red] (n2) to (n1);
\draw [bend right, <-, red] (A4) to (A3);
\draw [bend right, <-, red] (B4) to (B3);
\draw [bend right, <-, red] (C4) to (C3);
\draw [bend right, <-, red] (n4) to (n3);
\draw [bend right, bend angle=10, looseness=0.5, <-, red] (T)+(-200:4.5pt) to ($(A4) + (50:4.5pt)$);
\draw [bend right, bend angle=10, looseness=0.5, <-, red] (T)+(-155:4.5pt) to ($(B4) + (60:4.5pt)$);
\draw [bend right, <-, red] (T) to (C4);
\draw [bend right, bend angle=10, looseness=0.5,  <-, red] (T) to (n4);
\draw [decorate,decoration={brace,amplitude=5pt,mirror},xshift=4pt,yshift=0pt]
(2,1) -- (2,4) node [black,midway,xshift=0.55cm] 
{$\scriptstyle m_n$};
\draw [decorate,decoration={brace,amplitude=5pt,mirror},xshift=4pt,yshift=0pt]
(0,1) -- (0,4) node [black,midway,xshift=0.55cm] 
{$\scriptstyle m_3$};
\draw [decorate,decoration={brace,amplitude=5pt,mirror},xshift=4pt,yshift=0pt]
(-1.5,1) -- (-1.5,4) node [black,midway,xshift=0.55cm] 
{$\scriptstyle m_2$};
\draw [decorate,decoration={brace,amplitude=5pt,mirror},xshift=4pt,yshift=0pt]
(-3,1) -- (-3,4) node [black,midway,xshift=0.55cm] 
{$\scriptstyle m_1$};
\end{tikzpicture}
\end{array}
\end{equation}
where by convention if $m_i=0$ the $i$th arm does not exist. 
Further, we add extra arrows subject to the following rules:
\begin{enumerate}
\item If some $\upalpha_{ij}>2$, add $\upalpha_{ij}-2$ extra arrows from that vertex to the top vertex.
\item Add further $a$ arrows from the bottom vertex to the top vertex.
\end{enumerate} 
\end{cor}
\begin{proof}
As in \cite[\S4]{D1}, we first work on the completion $\Uppi_{i\in\bZ}\Hom^{\bZ}_{S^{\ox}}(M^{\ox},M^{\ox}(i\ox))$ of $\Upgamma_{\ox}$, where the result follows by combining \ref{positions forced}, \eqref{Zf dot Ei}, \eqref{ZfZk dot Ei} and \ref{GL2 for R}.  The result then follows from the easy fact that if $f\colon \bigoplus_{i\geq 0}A_i\to \bigoplus_{i\geq 0}B_i$ is a morphism of graded rings, then $f$ is an isomorphism if and only $\widehat{f}\colon \Uppi_{i\geq 0}A_i\to \Uppi_{i\geq 0}B_i$ is an isomorphism.
\end{proof}

It is possible to describe the relations of $\Upgamma_{\ox}$ in this level of generality, but for notational ease we will only do this for the $0$-Wahl Veronese in \S\ref{Hilb series and Veronese} below.  However, in full generality, we do have the following.

\begin{prop}\label{deg 0 general prop}
Suppose that $\ox=\sum_{i=1}^na_i\ox_i+a\oc\in\bL_+$ with $\ox\notin[0,\oc\,]$.  Then, with notation as in \ref{recon quiver and number relations}, $(\Upgamma_{\ox})_0$, the degree zero part of the reconstruction algebra $\Upgamma_{\ox}$, is isomorphic to the canonical algebra $\Lambda_{\bq,\bm}$, where $I\colonequals \{i\in[1,n]\mid a_i\neq 0\}$, $\bq\colonequals (m_i+1)_{i\in I}$ and $\bm\colonequals (\lambda_i)_{i\in I}$.
\end{prop}
\begin{proof}
By \ref{change parameters so coprime} we can change parameters to assume that the coprime assumption \ref{WPL as qgrZ}\eqref{WPL as qgrZ 4} holds. Thus we have $\CM^{\bZ}\!S^{\ox}\simeq\CM^{\bL}\!S$ and hence $(\Upgamma_{\ox})_0\cong\End^{\bL}_S(M)$.  But by  \cite{GL1} it is well known that there is a ring isomorphism $\End^{\bL}_S(\bigoplus_{\oy\in[0,\oc\,]} S(\oy))\cong\Lambda_{\bp,\bl}$.  Thus $(\Upgamma_{\ox})_0=e\Lambda_{\bp,\bl}e$ for an idempotent $e$ corresponding to a subset of $[0,\oc\,]$ containing $0$ and $\oc$.  Clearly $e\Lambda_{\bp,\bl}e\cong \Lambda_{\bq,\bm}$ for $\bq$ and $\bm$ in the statement.
\end{proof}

When $\ox\in\bL_+$ with $\ox\notin[0,\oc\,]$, we will next show in \ref{qgrR via qgrLambda} that $\qgr^\bZ\! S^{\ox}\simeq \qgr^\bZ\!\Upgamma_{\ox}$, since after combining with \ref{change parameters so coprime} this then allows us to realise any weighted projective line as $\qgr^\bZ\!\Upgamma_{\ox}$.  Since $\Upgamma_{\ox}$ is an $\bN$-graded ring, this allows us to interpret any weighted projective line as a `noncommutative projective scheme' over `$\Spec(\Upgamma_{\ox})_0$', and thus by \ref{deg 0 general prop} a noncommutative projective scheme over the canonical algebra.

\begin{prop}\label{qgrR via qgrLambda}
For $\ox\in\bL_+$ with $\ox\notin[0,\oc\,]$, there is an equivalence
\[
\qgr^\bZ\! S^{\ox}\simeq\qgr^\bZ\!\Upgamma_{\ox},
\]
where $\Upgamma_{\ox}$ is  a $\bZ$-graded $\KK$-algebra such that $(\Upgamma_{\ox})_i$ is $\Lambda_{\bq,\bm}$ for $i=0$ and zero for $i<0$.
\end{prop}
\begin{proof}
Set $R\colonequals S^{\ox}$.  Let $A\colonequals \Upgamma_{\ox}=\End_{R}(M^{\ox})$ and let $e\in A$ be the idempotent corresponding to the summand $R$ of $M^{\ox}$.  Clearly $B\colonequals eAe\cong R$.  Note that $\dim_\KK(A/\langle e\rangle)<\infty$ since by \ref{Two dim and normal prelim}\eqref{L-factorial 7} $R$ is normal, so $\add M^{\ox}_{\p}=\add R_{\p}$ for any non-maximal prime ideal $\p$ of $R$, and hence $(A/\langle e\rangle)_{\p}=0$. The first statement then follows from \ref{comm to noncomm}, and the second statement by \ref{deg 0 general prop}.
\end{proof}

\section{The $0$-Wahl Veronese}\label{Hilb series and Veronese}

Throughout this section we work with an arbitrary $\bX_{\bp,\bl}$ with $n\ge3$, and consider the  $0$-Wahl Veronese subring of $S=S_{\bp,\bl}$ from the introduction, namely $S^{\os}$, where $\os=\sum_{i=1}^n\ox_i$.  It is not too hard, but more notationally complicated, to extend to cover the case $\os_a=\os+a\oc$, but we shall not do this here.  We investigate the more general $S^{\os_a}$ for Dynkin type in \S\ref{domestic section}.

\subsection{Presenting the $0$-Wahl Veronese} 
The aim of this subsection is to give a presentation of the $0$-Wahl Veronese subring $S^{\os}$ of $S$ by constructing an isomorphism $S^{\os}\cong R_{\bp,\bl}$.
We define elements of $S^{\os}$ as follows:
\begin{eqnarray*}
\mathsf{u}_i&\colonequals &\left\{\begin{array}{ll}x_1^{p_1+p_2}x_3^{p_2}\hdots x_n^{p_2}&i=1,\\
x_2^{p_1+p_2}x_3^{p_1}\hdots x_n^{p_1}&i=2,\\
-x_1^{p_i}x_2^{p_2+p_i}x_3^{p_i}\hdots \widehat{x_i}\hdots x_n^{p_i}&3\leq i\leq n,\end{array}\right.\\
\mathsf{v}&\colonequals &x_1x_2\hdots x_n,
\end{eqnarray*}
where we write $\widehat{x_i}$ to mean `omit $x_i$'. Then with respect to the $\bZ$-grading $S^{\os}=\bigoplus_{i\in\bZ}S_{i\os}$, the element $\mathsf{v}$ is homogeneous of degree one, and $\mathsf{u}_i$ is homogeneous of degree $p_2$ if $i=1$, $p_1$ if $i=2$ and $p_i$ if $3\le i\le n$.

To construct an isomorphism between $R_{\bp,\bl}$ and $S^{\os}$, we first construct a morphism of graded algebras.

\begin{lemma}\label{check relations}
The morphism $\KK[u_1,\ldots,u_n,v]\to S^{\os}$ of graded algebras given by $u_i\mapsto\mathsf{u}_i$ for $1\le i\le n$ and $v\mapsto\mathsf{v}$ induces a morphism $R_{\bp,\bl}\to S^{\os}$ of graded algebras. 
\end{lemma}

\begin{proof}
It suffices to show that all $2\times2$ minors of the following matrix have determinant zero.
\begin{equation}\label{R matrix}
\left(
\begin{array}{ccccc}
{\mathsf u}_2&{\mathsf u}_3&\hdots&{\mathsf u}_{n}&{\mathsf v}^{p_2}\\
{\mathsf v}^{p_1}&\lambda_3{\mathsf u}_3+{\mathsf v}^{p_3}&\hdots&\lambda_n{\mathsf u}_n+{\mathsf v}^{p_n}&{\mathsf u}_1
\end{array}
\right)
\end{equation}
Since $S^{\os}$ is a domain, it suffices to show that all $2\times2$ minors containing the last column have determinant zero.  The outer $2\times 2$ minor has determinant
\[
\mathsf{u}_1\mathsf{u}_2-\mathsf{v}^{p_1+p_2}=(x_1\hdots x_n)^{p_1+p_2}-(x_1\hdots x_n)^{p_1+p_2}=0.  
\]
Further for any $i\geq 3$, using the relation $x_1^{p_1}=\lambda_ix_2^{p_2}-x_i^{p_i}$ it follows that
\begin{align*}
\mathsf{u}_1\mathsf{u}_i&=-x_1^{p_1+p_2+p_i}x_2^{p_2+p_i}x_3^{p_2+p_i}\hdots x_i^{p_2}\hdots x_n^{p_2+p_i}\\
&=-(\lambda_ix_2^{p_2}-x_i^{p_i})x_1^{p_2+p_i}x_2^{p_2+p_i}x_3^{p_2+p_i}\hdots x_i^{p_2}\hdots x_n^{p_2+p_i}\\
&=x_1^{p_2}x_2^{p_2}\hdots x_n^{p_2}(-\lambda_i  x_1^{p_i}x_2^{p_2+p_i}x_3^{p_i}\hdots \widehat{x_i}\hdots x_n^{p_i}+x_1^{p_i}x_2^{p_i}\hdots x_n^{p_i}) \\
&=\mathsf{v}^{p_2}(\lambda_i\mathsf{u}_i+\mathsf{v}^{p_i}).
\end{align*}
Thus the $2\times 2$ minor consisting of $i$th column and the last one has determinant zero.
\end{proof}

The following calculation is elementary.

\begin{prop}\label{S is generated}
\begin{enumerate}
\item The $\KK$-algebra $S^{\os}$ is generated by $\mathsf{v}$ and $\mathsf{u}_i$ with $1\le i\le n$.
\item The $\KK$-vector space $S^{\os}/\mathsf{v}S^{\os}$ is generated by $\mathsf{u}_i^\ell$ with $1\le i\le n$ and $\ell\ge0$.
\end{enumerate}
\end{prop}
\begin{proof}
It is enough to prove (2).  Let $V$ be the subspace of $S^{\os}/\mathsf{v}S^{\os}$ generated by $\mathsf{u}_i^\ell$ with $1\le i\le n$ and $\ell\ge0$.  Take any monomial $X\colonequals x_1^{a_1}\hdots x_n^{a_n}$ in $S_{N\os}$ with $N>0$, then
\[
a_1\ox_1+\hdots+a_n\ox_n=N\ox_1+\hdots +N\ox_n.
\]
For each $1\le i\le n$, there exists $\ell_i\in\bZ$ such that $a_i=N+\ell_ip_i$. Then $\sum_{i=1}^n\ell_i=0$ holds.\\
(i) We first show that $X$ belongs to $V$ in the situation when there exists $i$ with $1\le i\le n$ satisfying $\ell_i\le0$, $\ell_j\ge0$ and $\ell_k=0$ for all $k\neq i,j$, where $j$ is defined by $j\colonequals 2$ if $i\neq2$ and $j\colonequals 1$ if $i=2$. 

If $a_i\neq0$, then by the assumptions, all $a_k\geq 1$ for $1\leq k\leq n$, and hence $X$ belongs to $\mathsf{v}S^{\os}$.  Thus we can assume that $a_i=0$. Then $N=-\ell_ip_i$ and $a_j=N+\ell_jp_j=-\ell_i(p_i+p_j)$ hold, and further
\[
X=x_j^{a_j}\prod_{k\neq i,j}x_k^N=(x_j^{p_i+p_j}\prod_{k\neq i,j}x_k^{p_i})^{-\ell_i}=\pm\mathsf{u}_{i'}^{-\ell_i},
\]
where $i'\colonequals 2$ if $i=1$, $i'\colonequals 1$ if $i=2$ and $i'=i$ if $i\ge 3$.
Thus the assertion follows.\\
(ii) We consider the general case.
Using induction on $\ell(X)\colonequals \sum_{1\le i\le n,\ \ell_i>0}\ell_i$, we show that $X$ belongs to $V$.

Assume $\ell(X)=0$. Then $X=\mathsf{v}^N$ holds, and hence $X$ belongs to $V$.

Assume that there exist $1\le i\neq j\le n$ such that $\ell_i<0$ and $\ell_j<0$.
Take $1\le k\le n$ such that $\ell_k>0$.  Since $\{x_i^{p_i}, x_j^{p_j}\}$ is a $\KK$-basis of $S_{\oc}$, there is a relation $x_k^{p_k}=\lambda'x_i^{p_i}+\lambda''x_j^{p_j}$ with $\lambda',\lambda''\in\KK$, and we have $X=\lambda'X'+\lambda''X''$ for some monomials $X',X''$ satisfying $\ell(X')<\ell(X)$ and $\ell(X'')<\ell(X)$.
Since $X'$ and $X''$ belong to $V$, so does $X$.

In the rest, assume that there exists a unique $1\le i\le n$ satisfying $\ell_i<0$. Define $j$ by $j\colonequals 2$ if $i\neq2$ and $j\colonequals 1$ if $i=2$.
Using the relation $x_k^{p_k}=\lambda'_kx_i^{p_i}+\lambda''_kx_j^{p_j}$ with $\lambda'_k,\lambda''_k\in\KK$, we have
\[
X
=x_i^{a_i}x_j^{a_j}\prod_{k\neq i,j}x_k^{N+\ell_kp_k}
=x_i^{a_i}x_j^{a_j}\prod_{k\neq i,j}x_k^N(\lambda'_kx_i^{p_i}+\lambda''_kx_j^{p_j})^{\ell_k}.
\]
This is a linear combination of monomials $Y=x_i^{b_i}x_j^{b_j}\prod_{k\neq i,j}x_k^N$ which satisfies the condition in (i).
Thus $X$ belongs to $V$.
\end{proof}

The above leads to the following, which is the main result of this subsection.
\begin{thm}\label{Veronese result}
There is a graded ring isomorphism $R_{\bp,\bl}\cong S^{\os}$ given by $u_i\mapsto\mathsf{u}_i$ for $1\le i\le n$ and $v\mapsto\mathsf{v}$.
\end{thm}
\begin{proof}
Combining \ref{check relations} and \ref{S is generated}, there is a surjective graded ring homomorphism $\vartheta\colon R_{\bp,\bl}\twoheadrightarrow S^{\os}$. But now $R_{\bp,\bl}$, being a rational surface singularity, is automatically a domain.  Since $S^{\os}$ is two-dimensional, $\vartheta$ must be an isomorphism.
\end{proof}

\subsection{Special CM $S^{\os}$-Modules and the Reconstruction Algebra}\label{Sect 5.2}
The benefit of our Veronese construction of $R_{\bp,\bl}$ is that it also produces the special CM modules, and we now describe them explicitly as $2$-generated ideals.   We first do this in the notation of $S$, then translate into the coordinates $\mathsf{u}_1,\hdots,\mathsf{u}_n,\mathsf{v}$.

\begin{prop}\label{2 gen as S modules}
The following are, up to degree shift, precisely the indecomposable non-free objects in $\SCM^{\bZ}\!S^{\os}$. Moreover, they have the following generators and degrees: 
\[
\begin{tabular}{*3c}
\toprule
Module & Generators & Degree of generators\\ 
\midrule
$S(q\ox_1)^{\os}$& $x_2^{p_2}(x_2x_3\hdots x_n)^{p_1-q}$  and $ x_1^{q}$ & $p_1-q$ and $0$\\
$S(q\ox_2)^{\os}$& $x_1^{p_1}(x_1x_3\hdots x_n)^{p_2-q}$ and  $x_2^{q}$ & $p_2-q$ and $0$\\ 
$S(q\ox_i)^{\os}$& $x_2^{p_2}(x_1\hdots\widehat{x_i}\hdots x_n)^{p_i-q}$ and $x_i^{q}$& $p_i-q$ and $0$\\
$S(\oc)^{\os}$ & $x_1^{p_1}$ and $x_2^{p_2}$& $0$ and $0$\\ 
\bottomrule
\end{tabular}
\]
where in row one $q\in[1, p_1-1]$, in row two $q\in [1,p_2-1]$, and in row three $i\in [3,n], \,q\in [1,p_i-1]$.
\end{prop}
\begin{proof}
The first statement is \ref{specials lag approach thm}\eqref{specials lag approach thm 1}. We only prove the assertions for $S(q\ox_1)^{\os}$ since all other cases are similar. Let $M$ be the submodule of $S(q\ox_1)^{\os}$ generated by $g_1\colonequals x_1^q$ and $g_2\colonequals x_2^{p_1+p_2-q}x_3^{p_1-q}\hdots x_n^{p_1-q}$. To prove $M=S(q\ox_1)^{\os}$, it suffices to show that any monomial $X=x_1^{a_1}\hdots x_n^{a_n}\in S(q\ox_1)^{\os}$ of degree $N\ge0$ has either $g_1$ or $g_2$ as a factor. Since
\[
a_1\ox_1+\hdots+a_n\ox_n=(N+q)\ox_1+N\ox_2+\hdots +N\ox_n.
\]  
holds, there exists $\ell_i\in\bZ$ for each $1\le i\le n$ such that $a_1=N+q+\ell_1p_1$ and $a_i=N+\ell_ip_i$ for $i\ge2$.
Then $\sum_{i=1}^n\ell_i=0$ holds.\\
(i) If $a_1\ge q$, then $X$ belongs to $M$ since $X$ has $g_1=x_1^q$ as a factor.\\
(ii) We show that $X$ belongs to $M$ if $\ell_3=\hdots=\ell_n=0$.
By (i), we can assume that $a_1<q$ and hence $\ell_1<0$. Then $N=a_1-q-\ell_1p_1\ge p_1-q$ holds. Since $\ell_2=-\ell_1>0$, we have $a_2=N+\ell_2p_2\ge p_1+p_2-q$, which implies that $X=x_1^{a_1}x_2^{a_2}x_3^N\hdots x_n^N$ has $g_2=x_2^{p_1+p_2-q}x_3^{p_1-q}\hdots x_n^{p_1-q}$ as a factor.\\
\noindent(iii) We show that $X$ belongs to $M$ if all $\ell_3, \hdots,\ell_n$ are non-positive. By (i), we can assume that $a_1<q$ and hence $\ell_1<0$.
Using $\ell_2=-\sum_{i\neq 2}\ell_i$ and the relation $x_2^{p_2}=\lambda'_ix_1^{p_1}+\lambda''_ix_i^{p_i}$, it follows that
\[
X=x_1^{a_1}x_2^{N-\sum_{i\neq 2}\ell_ip_2}x_3^{a_3}\hdots x_n^{a_n}=x_1^{a_1}x_2^{N-\ell_1p_2}\prod_{i\ge3}x_i^{a_i}(\lambda'_ix_1^{p_1}+\lambda''_ix_i^{p_i})^{-\ell_i}.
\]
Since $a_1+p_1>q$, this is a linear combination of monomials which have $g_1=x_1^q$ as a factor and of a monomial $x_1^{a_1}x_2^{N-\ell_1p_2}\prod_{i\ge3}x_i^{a_i-\ell_ip_i}=x_1^{a_1}x_2^{N-\ell_1p_2}x_3^N\hdots x_n^N$ satisfying (ii). Thus $X$ belongs to $M$.\\
(iv) We show that $X$ belongs to $M$ in general. Let $\ell_i^+=\max\{\ell_i,0\}$ and $\ell_i^-=\min\{\ell_i,0\}$, then $\ell_i=\ell_i^++\ell_i^-$.  Further, using the relation $x_i^{p_i}=-x_1^{p_1}+\lambda_ix_2^{p_2}$, 
\[
X=x_1^{a_1}x_2^{a_2}\prod_{i\ge3}x_i^{N+\ell_ip_i}=-x_1^{a_1}x_2^{a_2}\prod_{i\ge3}x_i^{N+\ell_i^-p_i}(-x_1^{p_1}+\lambda_ix_2^{p_2})^{\ell^+_i}.
\]
This is a linear combination of monomials satisfying (iii), so $X$ belongs to $M$.
\end{proof}

Using \ref{S is generated} we now translate the modules in \ref{2 gen as S modules} into ideals. 

\begin{prop}\label{dual graph assignment}
With notation in \ref{Veronese result}, up to degree shift, the non-free indecomposable objects in $\SCM^{\bZ}\!S^{\os}$ are precisely the following ideals of $S^{\os}$, and furthermore across the bijection in \ref{Wunram main} they correspond to the dual graph of the minimal resolution of $\Spec S^{\os}$ \eqref{s Veron dual graph} in the following way:
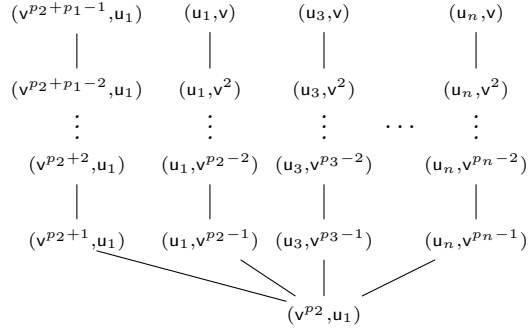
\[
\begin{array}{c}
\begin{tikzpicture}[xscale=1,yscale=0.8,bend angle=30, looseness=1]
\node (0) at (0,0) {$\scriptstyle (\mathsf{v}^{p_2},\mathsf{u}_1)$};
\node (A1) at (-3.25,1) {$\scriptstyle (\mathsf{v}^{p_2+1},\mathsf{u}_1)$};
\node (A2) at (-3.25,2) {$\scriptstyle (\mathsf{v}^{p_2+2},\mathsf{u}_1)$};
\node (A3) at (-3.25,3) {$\scriptstyle (\mathsf{v}^{p_2+p_1-2},\mathsf{u}_1)$};
\node (A4) at (-3.25,4) {$\scriptstyle (\mathsf{v}^{p_2+p_1-1},\mathsf{u}_1)$};
\node (B1) at (-1.5,1) {$\scriptstyle (\mathsf{u}_1,\mathsf{v}^{p_2-1})$};
\node (B2) at (-1.5,2) {$\scriptstyle (\mathsf{u}_1,\mathsf{v}^{p_2-2})$};
\node (B3) at (-1.5,3) {$\scriptstyle (\mathsf{u}_1,\mathsf{v}^2)$};
\node (B4) at (-1.5,4) {$\scriptstyle (\mathsf{u}_1,\mathsf{v})$};
\node (C1) at (0,1) {$\scriptstyle (\mathsf{u}_3,\mathsf{v}^{p_3-1})$};
\node (C2) at (0,2) {$\scriptstyle (\mathsf{u}_3,\mathsf{v}^{p_3-2})$};
\node (C3) at (0,3) {$\scriptstyle (\mathsf{u}_3,\mathsf{v}^{2})$};
\node (C4) at (0,4) {$\scriptstyle (\mathsf{u}_3,\mathsf{v})$};
\node (n1) at (2,1) {$\scriptstyle (\mathsf{u}_n,\mathsf{v}^{p_n-1})$};
\node (n2) at (2,2) {$\scriptstyle (\mathsf{u}_n,\mathsf{v}^{p_n-2})$};
\node (n3) at (2,3) {$\scriptstyle (\mathsf{u}_n,\mathsf{v}^{2})$};
\node (n4) at (2,4) {$\scriptstyle (\mathsf{u}_n,\mathsf{v})$};
\node at (-3.25,2.6) {$\vdots$};
\node at (-1.5,2.6) {$\vdots$};
\node at (0,2.6) {$\vdots$};
\node at (2,2.6) {$\vdots$};
\node at (1,2.5) {$\hdots$};
\draw [-] (A1)+(-30:15pt) -- ($(0) + (160:15pt)$);
\draw [-] (B1) --(0);
\draw [-] (C1) --(0);
\draw [-] (n1) --(0);
\draw [-] (A2) -- (A1);
\draw [-] (B2) --(B1);
\draw [-] (C2) --(C1);
\draw [-] (n2) --(n1);
\draw [-] (A4) -- (A3);
\draw [-] (B4) --(B3);
\draw [-] (C4) --(C3);
\draw [-] (n4) --  (n3);
\end{tikzpicture}
\end{array}
\]
\end{prop}
\begin{proof}
We first claim that $S(\ox_1)^{\os}\cong (\mathsf{v}^{p_1+p_2-1},\mathsf{u}_1)$.  Indeed, since $S$ is an $\bL$-domain by \ref{L-factorial}, multiplication by any homogeneous element $S\to S$ is injective.  Thus, multiplying by  $x_2\hdots x_n$, we see that $S(\ox_1)^{\os}$ is isomorphic to the $S^{\os}$-submodule of $S$ generated by $x_2^{p_1+p_2}x_3^{p_1}\hdots x_n^{p_1}$ and $x_1\hdots x_n$, that is generated by $\mathsf{u}_2$ and $\mathsf{v}$. But then
\[
\mathsf{u}_1(\mathsf{u}_2,\mathsf{v})=(\mathsf{u}_1\mathsf{u}_2,\mathsf{u}_1\mathsf{v})
\stackrel{{\scriptstyle\ref{Veronese result}}}{\cong}
(\mathsf{v}^{p_1+p_2},\mathsf{u}_1\mathsf{v})=(\mathsf{v}^{p_1+p_2-1},\mathsf{u}_1)\mathsf{v},
\]
which shows that  $S(\ox_1)^{\os}\cong(\mathsf{v}^{p_1+p_2-1},\mathsf{u}_1)$. The other cases are similar.  The statement regarding the bijection is a special case of \ref{positions forced}.
\end{proof}

\begin{prop}\label{all arrows}
The reconstruction algebra $\Upgamma_{\os}$ is given by the following quiver, where the arrows correspond to the following morphisms.
\[
\begin{array}{c}
\begin{tikzpicture}[xscale=1.5,yscale=1.3,bend angle=30, looseness=1]
\node (0) at (-0.75,0) {$\scriptstyle (v^{p_2},u_1)$};
\node (A1) at (-2.75,1) {$\scriptstyle (v^{p_2+1},u_1)$};
\node (A3) at (-2.75,2) {$\scriptstyle (v^{p_2+p_1-2},u_1)$};
\node (A4) at (-2.75,3) {$\scriptstyle (v^{p_2+p_1-1},u_1)$};
\node (B1) at (-1.5,1) {$\scriptstyle (u_1,v^{p_2-1})$};
\node (B3) at (-1.5,2) {$\scriptstyle (u_1,v^2)$};
\node (B4) at (-1.5,3) {$\scriptstyle (u_1,v)$};
\node (C1) at (0,1) {$\scriptstyle (u_3,v^{p_3-1})$};
\node (C3) at (0,2) {$\scriptstyle (u_3,v^{2})$};
\node (C4) at (0,3) {$\scriptstyle (u_3,v)$};
\node (n1) at (2,1) {$\scriptstyle (u_n,v^{p_n-1})$};
\node (n3) at (2,2) {$\scriptstyle (u_n,v^{2})$};
\node (n4) at (2,3) {$\scriptstyle (u_n,v)$};
\node at (-3,1.6) {$\vdots$};
\node at (-1.5,1.6) {$\vdots$};
\node at (0,1.6) {$\vdots$};
\node at (2,1.6) {$\vdots$};
\node at (1,2) {$\hdots$};
\node (T) at (-0.75,4)  {$\scriptstyle R$};
\draw [->] (A1)+(-30:8.5pt) -- node[gap] {$\scriptstyle \inc$} ($(0) + (180:11.5pt)$);
\draw [->] (B1) --node[gap] {$\scriptstyle v$}(0);
\draw [->] (C1) --node[right=-0.05] {$\scriptstyle \frac{v^{p_2+1}}{u_3}$}(0);
\draw [->] (n1) --node[above=-0.05] {$\scriptstyle \frac{v^{p_2+1}}{u_n}$}(0);
\draw [->] (A4) -- node[gap] {$\scriptstyle \inc$} (A3);
\draw [->] (B4) --node[gap] {$\scriptstyle v$}(B3);
\draw [->] (C4) --node[gap] {$\scriptstyle v$}(C3);
\draw [->] (n4) -- node[gap] {$\scriptstyle v$} (n3);
\draw [->] (T)+(-190:4.5pt) -- node[gap] {$\scriptstyle u_1$} ($(A4) + (30:7.5pt)$);
\draw [->] (T) -- node[gap] {$\scriptstyle v$} (B4);
\draw [->] (T) -- node[gap] {$\scriptstyle v$} (C4);
\draw [->] (T) -- node[gap] {$\scriptstyle v$}  (n4);
\draw [bend right, bend angle=10, looseness=0.5, <-, red] (A1)+(-50:6.5pt) to node[gap] {$\scriptstyle v$} ($(0) + (190:11.5pt)$);
\draw [bend right, bend angle=10, looseness=0.5, <-, red] (B1)+(-70:5.5pt) to node[gap,xshift=-3pt] {$\scriptstyle \inc$} ($(0) + (150:8pt)$);
\draw [bend right, <-, red] ($(C1)+(-140:7pt)$) to node[gap] {$\scriptstyle \frac{u_3}{v^{p_2}}$} (0);
\draw [bend left, bend angle=10, looseness=0.5, <-, red] (n1)+(-140:8pt) to node[gap] {$\scriptstyle \frac{u_n}{v^{p_2}}$} ($(0)+(7.5:14pt)$);
\draw [bend right, <-, red] (A4) to node[left] {$\scriptstyle v$} (A3);
\draw [bend right, <-, red] (B4) to node[left] {$\scriptstyle \inc$}(B3);
\draw [bend right, <-, red] (C4) to node[left] {$\scriptstyle \inc$} (C3);
\draw [bend left, <-, red] (n4) to node[right] {$\scriptstyle \inc$}(n3);
\draw [bend right, bend angle=10, looseness=0.5, <-, red] (T)+(-200:4.5pt) to node[gap,yshift=2pt] {$\scriptstyle\frac{v}{u_1}$} ($(A4) + (50:6.5pt)$);
\draw [bend right, bend angle=10, looseness=0.5, <-, red] (T)+(-155:4.5pt) to node[gap,xshift=-2pt] {$\scriptstyle \inc$} ($(B4) + (70:5.5pt)$);
\draw [bend right, <-, red] (T) to node[gap,xshift=-2pt] {$\scriptstyle \inc$} (C4);
\draw [bend left, bend angle=10, looseness=0.5,  <-, red] (T) to node[gap] {$\scriptstyle \inc$}  ($(n4) + (150:10pt)$);
\end{tikzpicture}
\end{array}
\]
\end{prop}
\begin{proof}
Under the isomorphisms in \ref{Veronese result} and \ref{dual graph assignment}, the morphisms induced by the canonical algebra become
\begin{equation}\label{canonical in R coordinates}
\begin{array}{c}
\begin{tikzpicture}[xscale=1.3,yscale=1,bend angle=30, looseness=1]
\node (0) at (0,0) {$\scriptstyle (\mathsf{v}^{p_2},\mathsf{u}_1)$};
\node (A1) at (-3,1) {$\scriptstyle (\mathsf{v}^{p_2+1},\mathsf{u}_1)$};
\node (A2) at (-3,2) {$\scriptstyle (\mathsf{v}^{p_2+2},\mathsf{u}_1)$};
\node (A3) at (-3,3) {$\scriptstyle (\mathsf{v}^{p_2+p_1-2},\mathsf{u}_1)$};
\node (A4) at (-3,4) {$\scriptstyle (\mathsf{v}^{p_2+p_1-1},\mathsf{u}_1)$};
\node (B1) at (-1.5,1) {$\scriptstyle (\mathsf{u}_1,\mathsf{v}^{p_2-1})$};
\node (B2) at (-1.5,2) {$\scriptstyle (\mathsf{u}_1,\mathsf{v}^{p_2-2})$};
\node (B3) at (-1.5,3) {$\scriptstyle (\mathsf{u}_1,\mathsf{v}^2)$};
\node (B4) at (-1.5,4) {$\scriptstyle (\mathsf{u}_1,\mathsf{v})$};
\node (C1) at (0,1) {$\scriptstyle (\mathsf{u}_3,\mathsf{v}^{p_3-1})$};
\node (C2) at (0,2) {$\scriptstyle (\mathsf{u}_3,\mathsf{v}^{p_3-2})$};
\node (C3) at (0,3) {$\scriptstyle (\mathsf{u}_3,\mathsf{v}^{2})$};
\node (C4) at (0,4) {$\scriptstyle (\mathsf{u}_3,\mathsf{v})$};
\node (n1) at (2,1) {$\scriptstyle (\mathsf{u}_n,\mathsf{v}^{p_n-1})$};
\node (n2) at (2,2) {$\scriptstyle (\mathsf{u}_n,\mathsf{v}^{p_n-2})$};
\node (n3) at (2,3) {$\scriptstyle (\mathsf{u}_n,\mathsf{v}^{2})$};
\node (n4) at (2,4) {$\scriptstyle (\mathsf{u}_n,\mathsf{v})$};
\node at (-3,2.6) {$\vdots$};
\node at (-1.5,2.6) {$\vdots$};
\node at (0,2.6) {$\vdots$};
\node at (2,2.6) {$\vdots$};
\node at (1,2.5) {$\hdots$};
\node (T) at (0,5)  {$\scriptstyle R$};
\draw [->] (A1)+(-30:8.5pt) -- node[gap] {$\scriptstyle \inc$} ($(0) + (160:11.5pt)$);
\draw [->] (B1) --node[gap] {$\scriptstyle \mathsf{v}$}(0);
\draw [->] (C1) --node[right=-0.1] {$\scriptstyle \frac{\mathsf{v}^{p_2+1}}{\mathsf{u}_3}$}(0);
\draw [->] (n1) --node[below,pos=0.2] {$\scriptstyle \frac{\mathsf{v}^{p_2+1}}{\mathsf{u}_n}$}(0);
\draw [->] (A2) -- node[right] {$\scriptstyle \inc$} (A1);
\draw [->] (B2) --node[right] {$\scriptstyle \mathsf{v}$}(B1);
\draw [->] (C2) --node[right] {$\scriptstyle \mathsf{v}$}(C1);
\draw [->] (n2) -- node[right] {$\scriptstyle \mathsf{v}$}(n1);
\draw [->] (A4) -- node[right] {$\scriptstyle \inc$} (A3);
\draw [->] (B4) --node[right] {$\scriptstyle \mathsf{v}$}(B3);
\draw [->] (C4) --node[right] {$\scriptstyle \mathsf{v}$}(C3);
\draw [->] (n4) -- node[right] {$\scriptstyle \mathsf{v}$} (n3);
\draw [->] (T)+(-160:7.5pt) -- node[gap] {$\scriptstyle \mathsf{u}_1$} ($(A4) + (30:9.5pt)$);
\draw [->] (T) -- node[gap] {$\scriptstyle \mathsf{v}$} (B4);
\draw [->] (T) -- node[right] {$\scriptstyle \mathsf{v}$} (C4);
\draw [->] (T) -- node[gap] {$\scriptstyle \mathsf{v}$}  (n4);
\end{tikzpicture}
\end{array}
\end{equation}
From here, exactly as in the proof of \ref{recon quiver and number relations}, we can work on the completion. We know the quiver of the reconstruction algebra from \eqref{recon quiver}, and we know that for every special CM module $X$, we must be able to hit the generators of $X$ by composing arrows starting at the vertex $R$ and ending at the vertex corresponding to $X$, without producing any cycles.  Since the arrows in \eqref{canonical in R coordinates} are already forced to be arrows in the reconstruction algebra, it remains to choose a basis for the remaining red arrows.  For example, the generator $v^{p_2+1}$ in $(v^{p_2+1},u_1)$ must come from a composition of arrows $R$ to $(v^{p_2},u_1)$, followed by the bottom left arrow.  Since we can see $v^{p_2}$ as a composition of maps from $R$ to $(v^{p_2},u_1)$, this forces the bottom left red arrow to be $v$.  The remaining arrows are similar.
\end{proof}

\begin{thm}\label{recon relations}
The reconstruction algebra $\Upgamma_{\os}$ is isomorphic to the path algebra of the double of the quiver $Q_{\bp}$, denoted $\overline{Q}_{\bp}$, subject to relations given by
\begin{enumerate}
\item The canonical algebra relations on the black arrows
\item At every vertex, all 2-cycles that exist at that vertex are equal. 
\end{enumerate}
\end{thm}
\begin{proof}
This is very similar to \cite[4.11]{D1}.  Set $Q\colonequals \overline{Q}_{\bp,\bl}$ (as in \eqref{recon quiver}), and denote the set of relations in the statement by $\cS^\prime$.  Exactly as in the proof of \ref{recon quiver and number relations}, we can work in the completed case (where we can use \cite[3.4]{BIRS}) and we prove that the completion of reconstruction algebra is given as the completion of $\KK Q$ (denoted $\KK\widehat{Q}$) modulo the closure of the ideal $\langle \cS^\prime\rangle$ (denoted $\overline{\langle \cS^\prime\rangle}$).  The non-completed version of the theorem then follows.

By \ref{all arrows} there is a natural surjection $\upgamma\colon\KK\widehat{Q}\rightarrow \hat{\Upgamma}$ with $\cS^\prime\subseteq I\colonequals \Ker\upgamma$.  Denote the radical of $\KK\widehat{Q}$ by $J$ and further let $V$ denote the set of vertices of $Q$.  Below we show that the elements of $\cS^\prime$ are linearly independent in $I/ (IJ+JI)$, hence we may extend $\cS^\prime$ to a basis $\cS$ of $I/ (IJ+JI)$.  Since $\cS$ is a basis, by \cite[3.4(a)]{BIRS} $I=\overline{\langle \cS\rangle}$, so it remains to show that $\cS=\cS^\prime$.  But by \cite[3.4(b)]{BIRS}
\[
\# (e_{a}\KK\widehat{Q}e_{b})\cap \cS=\dim \Ext^2_{\hat{\Upgamma}}(S_a,S_b)
\] 
for all $a,b\in V$, where $S_a$ is the simple module corresponding to vertex $a$.  From \ref{recon quiver and number relations} (i.e.\ \cite{WemGL2}), this is equal to some number given by intersection theory.  Simply inspecting our set $\cS^\prime$ and comparing to the numbers in \ref{recon quiver and number relations}, we see that 
\[
\# (e_{a}\KK\widehat{Q}e_{b})\cap \cS=\# (e_{a}\KK\widehat{Q}e_{b})\cap \cS^\prime
\]
for all $a,b\in V$, proving that the number of elements in $\cS$ and $\cS^\prime$ are the same.  Hence $\cS^\prime=\cS$ and so $I=\overline{\langle \cS^\prime\rangle}$, as required.

Thus it suffices to show that the elements of $\cS^\prime$ are linearly independent in $I/ (IJ+JI)$.   This is identical to the proof 
of \cite[4.12]{D1}, so we omit the details.
\end{proof}

Whilst thinking of the special CM modules as ideals makes everything much more explicit, doing this forgets the grading. Indeed, the reconstruction algebra $\Upgamma_{\os}$ has a natural grading induced from the Veronese construction. 

\begin{prop}\label{grading inherited}
The reconstruction algebra $\Upgamma_{\os}$ is generated in degree one over its degree zero piece, which is the canonical algebra $\Lambda_{\bp,\bl}$.
\end{prop}
\begin{proof}
It is clear that all the black arrows in the quiver in \ref{all arrows} have degree zero.  It is easy to see that any red arrow in the reverse direction to an arrow labelled $x_i$ has label $x_1\hdots\widehat{x_i}\hdots x_n$, and it is easy to check that these all have degree one, using \ref{2 gen as S modules}.  Hence the degree zero piece is the canonical algebra, and as an algebra $\Upgamma_{\os}$ is generated in degree one over its degree zero piece.
\end{proof}

Note that for $a>0$, the reconstruction algebra $\Upgamma_{\os_a}$ is not always generated in degree one over its degree zero piece.

\section{Domestic Case}\label{domestic section}

In this section we investigate the domestic case, that is when the dual graph is an ADE Dynkin diagram, and relate Ringel's work on the representation theory of the canonical algebra to the classification of the special CM modules for quotient singularities in \cite{IW}.  This will explain the motivating coincidence from the introduction. Since this involves AR theory, typically in this section rings will be complete.

Throughout this section we consider $\bX=\bX_{\bp,\bl}$ and $S=S_{\bp,\bl}$ with $n=3$ and one of the triples $(p_1,p_2,p_3)=(2,3,3)$, $(2,3,4)$ or $(2,3,5)$. For $m\geq 3$, we consider $\os_{m-3}=\os+(m-3)\oc$ with $\os=\sum_{i=1}^3\ox_i$, the $(m-3)$-Wahl Veronese subring $R=S^{\os_{m-3}}$, and its completion $\mathfrak{R}$.  Recall that here $\ow=\oc-\os$, since $n=3$.

\begin{prop}\label{Veronese=quotient}
In the above setting, $\Spec R$ has the following dual graph:
\begin{equation}\label{dual graph 6}
\begin{array}{c}  
\begin{tikzpicture}[xscale=1,yscale=0.7]
\node (0) at (0,0) [vertex] {};
\node (A1) at (-2,1) [vertex] {};
\node (A2) at (-2,2) [vertex] {};
\node (A3) at (-2,3) [vertex] {};
\node (A4) at (-2,4) [vertex] {};
\node (C1) at (0,1) [vertex] {};
\node (C2) at (0,2) [vertex] {};
\node (C3) at (0,3) [vertex] {};
\node (C4) at (0,4) [vertex] {};
\node (n1) at (2,1) [vertex] {};
\node (n2) at (2,2) [vertex] {};
\node (n3) at (2,3) [vertex] {};
\node (n4) at (2,4) [vertex] {};
\node at (-2,2.6) {$\vdots$};
\node at (0,2.6) {$\vdots$};
\node at (2,2.6) {$\vdots$};
\node (T) at (0,4.25) {};
\node at (0,-0.2) {$\scriptstyle -m$};
\node at (-1.7,1) {$\scriptstyle -2$};
\node at (-1.7,2) {$\scriptstyle -2$};
\node at (-1.7,3) {$\scriptstyle -2$};
\node at (-1.7,4) {$\scriptstyle -2$};
\node at (0.3,1) {$\scriptstyle -2$};
\node at (0.3,2) {$\scriptstyle -2$};
\node at (0.3,3) {$\scriptstyle -2$};
\node at (0.3,4) {$\scriptstyle -2$};
\node at (2.3,1) {$\scriptstyle -2$};
\node at (2.3,2) {$\scriptstyle -2$};
\node at (2.3,3) {$\scriptstyle -2$};
\node at (2.3,4) {$\scriptstyle -2$};
\draw (A1) -- (0);
\draw (C1) -- (0);
\draw (n1) -- (0);
\draw (A2) -- (A1);
\draw (C2) -- (C1);
\draw (n2) -- (n1);
\draw (A4) -- (A3);
\draw (C4) -- (C3);
\draw (n4) -- (n3);
\draw [decorate,decoration={brace,amplitude=5pt},xshift=-4pt,yshift=0pt]
(2,1) -- (2,4) node [black,midway,xshift=-0.55cm] 
{$\scriptstyle p_3-1$};
\draw [decorate,decoration={brace,amplitude=5pt},xshift=-4pt,yshift=0pt]
(0,1) -- (0,4) node [black,midway,xshift=-0.55cm] 
{$\scriptstyle p_2-1$};
\draw [decorate,decoration={brace,amplitude=5pt},xshift=-4pt,yshift=0pt]
(-2,1) -- (-2,4) node [black,midway,xshift=-0.55cm] 
{$\scriptstyle p_1-1$};
\end{tikzpicture}
\end{array}
\end{equation} 
Moreover $\mathfrak{R}$ is isomorphic to a quotient singularity $\KK[[x,y]]^G$ in the following list: 
\[
\begin{tabular}{*3c}
\toprule
$(p_1,p_2,p_3)$ & Dual Graph & $G$\\
\midrule
$(2,3,3)$&
$\begin{array}{c}
\begin{tikzpicture}[xscale=0.7,yscale=0.8]
\node (-1) at (-1,0) [vertex] {};
\node (0) at (0,0) [vertex] {};
\node (1) at (1,0) [vertex] {};
\node (1b) at (1,0.75) [vertex] {};
\node (2) at (2,0) [vertex] {};
\node (3) at (3,0) [vertex] {};
\node (-1a) at (-1.2,-0.3) {$\scriptstyle - 2$};
\node (0a) at (-0.2,-0.3) {$\scriptstyle - 2$};
\node (1a) at (0.8,-0.3) {$\scriptstyle -m$};
\node (1ba) at (0.5,0.75) {$\scriptstyle - 2$};
\node (2a) at (1.8,-0.3) {$\scriptstyle - 2$};
\node (2a) at (2.8,-0.3) {$\scriptstyle - 2$};
\draw [-] (-1) -- (0);
\draw [-] (0) -- (1);
\draw [-] (1) -- (2);
\draw [-] (2) -- (3);
\draw [-] (1) -- (1b);
\end{tikzpicture}
\end{array}$
& 
$\mathbb{T}_{6(m-2)+1}$\\
$(2,3,4)$ &
$\begin{array}{c}
\begin{tikzpicture}[xscale=0.7,yscale=0.8]
\node (-1) at (-1,0) [vertex] {};
\node (0) at (0,0) [vertex] {};
\node (1) at (1,0) [vertex] {};
\node (1b) at (1,0.75) [vertex] {};
\node (2) at (2,0) [vertex] {};
\node (3) at (3,0) [vertex] {};
\node (4) at (4,0) [vertex] {};
\node (-1a) at (-1.2,-0.3) {$\scriptstyle - 2$};
\node (0a) at (-0.2,-0.3) {$\scriptstyle - 2$};
\node (1a) at (0.8,-0.3) {$\scriptstyle -m$};
\node (1ba) at (0.5,0.75) {$\scriptstyle - 2$};
\node (2a) at (1.8,-0.3) {$\scriptstyle - 2$};
\node (2a) at (2.8,-0.3) {$\scriptstyle - 2$};
\node (4a) at (3.8,-0.3) {$\scriptstyle -2$};
\draw [-] (-1) -- (0);
\draw [-] (0) -- (1);
\draw [-] (1) -- (2);
\draw [-] (2) -- (3);
\draw [-] (3) -- (4);
\draw [-] (1) -- (1b);
\end{tikzpicture}
\end{array}$
&
$\mathbb{O}_{12(m-2)+1}$\\
$(2,3,5)$ &
$\begin{array}{c}
\begin{tikzpicture}[xscale=0.7,yscale=0.8]
\node (-1) at (-1,0) [vertex] {};
\node (0) at (0,0) [vertex] {};
\node (1) at (1,0) [vertex] {};
\node (1b) at (1,0.75) [vertex] {};
\node (2) at (2,0) [vertex] {};
\node (3) at (3,0) [vertex] {};
\node (4) at (4,0) [vertex] {};
\node (5) at (5,0)[vertex] {};
 \node (-1a) at (-1.2,-0.3) {$\scriptstyle - 2$};
\node (0a) at (-0.2,-0.3) {$\scriptstyle - 2$};
\node (1a) at (0.8,-0.3) {$\scriptstyle -m$};
\node (1ba) at (0.5,0.75) {$\scriptstyle - 2$};
\node (2a) at (1.8,-0.3) {$\scriptstyle - 2$};
\node (2a) at (2.8,-0.3) {$\scriptstyle - 2$};
\node (4a) at (3.8,-0.3) {$\scriptstyle -2$};
\node (5a) at (4.8,-0.3) {$\scriptstyle - 2$};
\draw [-] (-1) -- (0);
\draw [-] (0) -- (1);
\draw [-] (1) -- (2);
\draw [-] (2) -- (3);
\draw [-] (3) -- (4);
\draw [-] (4) -- (5);
\draw [-] (1) -- (1b);
\end{tikzpicture}
\end{array}$
&
$\mathbb{I}_{30(m-2)+1}$\\
\bottomrule\\
\end{tabular}
\]
For the precise definition of the above subgroups of $\GL(2,\KK)$ we refer the reader to \emph{\cite{IW}}.
\end{prop}

\begin{proof}
By \ref{middle SI number}, the dual graph of $R$ is known to be \eqref{s Veron dual graph}. On the other hand, the quotient singularity $\KK[[x,y]]^G$ has the same dual graph \cite[\S3]{Riemen}.  Since the dual graphs \eqref{s Veron dual graph} for ADE triples are known to be taut \cite[Korollar 2.12]{Brieskorn}, the result follows.
\end{proof}

Let us finally explain why Ringel's picture \eqref{RingelPicture} in the introduction is the same as the ones found in \cite{IW} and \cite[\S4]{WemGL2}.  For example, in the family of groups $\mathbb{O}_{12(m-2)+1}$ with $m\geq 3$ in \ref{Veronese=quotient}, by \cite{AR_McKayGraphs} the AR quiver of
$\mathfrak{R}\cong\KK[[x,y]]^{\mathbb{O}_{12(m-2)+1}}$ is
\[
\begin{array}{c}
\begin{tikzpicture}[xscale=1.25,yscale=1.25]
\draw[densely dotted] (0.5,0.5) -- (0.5,-2.5);
\draw[densely dotted] (1.5,0.5) -- (1.5,-2.5);
\node (R0) at (0.5,0.5) [gap] {$\scriptstyle \mathfrak{R}$};
\node (R1) at (1.5,0.5) [vertex] {};
\node (R2) at (2.5,0.5) [vertex] {};
\node (R11) at (3.5,0.5) [vertex] {};
\node (R12) at (4.5,0.5) [gap] {$\scriptstyle \mathfrak{R}$};
\node (A1) at (1,0) [gap] [vertex] {};
\node (A2) at (2,0) [gap] [vertex] {};
\node (A12) at (4,0) [vertex] {};
\node (B0) at (0.5,-0.5) [vertex] {};
\node (B1) at (1.5,-0.5) [vertex] {};
\node (B2) at (2.5,-0.5) [vertex] {};
\node (B11) at (3.5,-0.5) [vertex] {};
\node (B12) at (4.5,-0.5) [vertex] {};
\node (C1) at (0.5,-1.1)  [vertex] {};
\node (C2) at (1,-1)  [vertex] {};
\node (C3) at (1.5,-1.1)  [vertex] {};
\node (C4) at (2,-1)  [vertex] {};
\node (C5) at (2.5,-1.1)  [vertex] {};
\node (C23) at (3.5,-1.1)  [vertex] {};
\node (C24) at (4,-1)  [vertex] {};
\node (C25) at (4.5,-1.1)  [vertex] {};
\node (D0) at (0.5,-1.5)  [vertex] {};
\node (D1) at (1.5,-1.5)  [vertex] {};
\node (D2) at (2.5,-1.5)  [vertex] {};
\node (D11) at (3.5,-1.5)  [vertex] {};
\node (D12) at (4.5,-1.5)  [vertex] {};
\node (E1) at (1,-2)  [vertex] {};
\node (E2) at (2,-2)  [vertex] {};
\node (E12) at (4,-2)  [vertex] {};
\node (F0) at (0.5,-2.5)  [vertex] {};
\node (F1) at (1.5,-2.5)  [vertex] {};
\node (F2) at (2.5,-2.5)  [vertex] {};
\node (F11) at (3.5,-2.5)  [vertex] {};
\node (F12) at (4.5,-2.5)  [vertex] {};
\draw[->] (R0) -- (A1);
\draw[->] (A1) -- (R1);
\draw[->] (R1) -- (A2);
\draw[->] (A2) -- (R2);
\draw[->] (R11) -- (A12);
\draw[->] (A12) -- (R12);
\draw[->] (B0) -- (A1);
\draw[->] (A1) -- (B1);
\draw[->] (B1) -- (A2);
\draw[->] (A2) -- (B2);
\draw[->] (B11) -- (A12);
\draw[->] (A12) -- (B12);
\draw[->] (B0) -- (C2);
\draw[->] (C2) -- (B1);
\draw[->] (B1) -- (C4);
\draw[->] (C4) -- (B2);
\draw[->] (B11) -- (C24);
\draw[->] (C24) -- (B12);
\draw[->] (C1) -- (C2);
\draw[->] (C2) -- (C3);
\draw[->] (C3) -- (C4);
\draw[->] (C4) -- (C5);
\draw[->] (C23) -- (C24);
\draw[->] (C24) -- (C25);
\draw[->] (D0) -- (C2);
\draw[->] (C2) -- (D1);
\draw[->] (D1) -- (C4);
\draw[->] (C4) -- (D2);
\draw[->] (D11) -- (C24);
\draw[->] (C24) -- (D12);
\draw[->] (D0) -- (E1);
\draw[->] (E1) -- (D1);
\draw[->] (D1) -- (E2);
\draw[->] (E2) -- (D2);
\draw[->] (D11) -- (E12);
\draw[->] (E12) -- (D12);
\draw[->] (F0) -- (E1);
\draw[->] (E1) -- (F1);
\draw[->] (F1) -- (E2);
\draw[->] (E2) -- (F2);
\draw[->] (F11) -- (E12);
\draw[->] (E12) -- (F12);
\node at (3,-1) {$\hdots$};
\end{tikzpicture}
\end{array}
\]
where there are precisely $12(m-2)+1$ repetitions of the original $\tilde{E}_7$ shown in dotted lines.  The left and right hand sides of the picture are identified, and there is no twist in this AR quiver. 
Thus as $m$ increases (and the group $\mathbb{O}_{12(m-2)+1}$ changes), the AR quiver becomes longer.  

Regardless of $m\geq 3$, by \cite[8.2]{IW} the special CM $\mathfrak{R}$-modules always have the following position in the AR quiver:
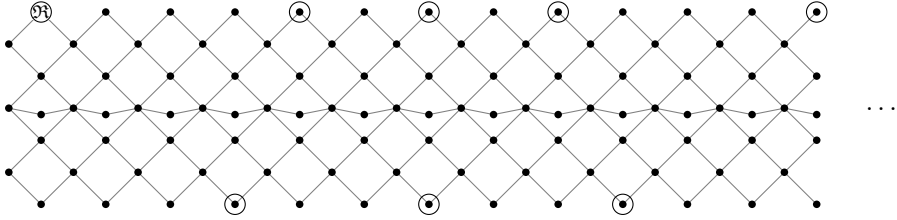
\[
\begin{array}{c}
\begin{tikzpicture}[xscale=0.85,yscale=0.85]
\draw[gray] (0,-2) -- (0.5,-2.5);
\draw[gray] (0,-1) -- (1.5,-2.5);
\draw[gray] (0,0) -- (2.5,-2.5);
\draw[gray] (0.5,0.5) -- (3.5,-2.5);
\draw[gray] (1.5,0.5) -- (4.5,-2.5);
\draw[gray] (2.5,0.5) -- (5.5,-2.5);
\draw[gray] (3.5,0.5) -- (6.5,-2.5);
\draw[gray] (4.5,0.5) -- (7.5,-2.5);
\draw[gray] (5.5,0.5) -- (8.5,-2.5);
\draw[gray] (6.5,0.5) -- (9.5,-2.5);
\draw[gray] (7.5,0.5) -- (10.5,-2.5);
\draw[gray] (8.5,0.5) -- (11.5,-2.5);
\draw[gray] (9.5,0.5) -- (12.5,-2.5);
\draw[gray] (10.5,0.5) -- (12.5,-1.5);
\draw[gray] (11.5,0.5) -- (12.5,-0.5);
\draw[gray] (0,0) -- (0.5,0.5);
\draw[gray] (0,-1) -- (1.5,0.5);
\draw[gray] (0,-2) -- (2.5,0.5);
\draw[gray] (0.5,-2.5) -- (3.5,0.5);
\draw[gray] (1.5,-2.5) -- (4.5,0.5);
\draw[gray] (2.5,-2.5) -- (5.5,0.5);
\draw[gray] (3.5,-2.5) -- (6.5,0.5);
\draw[gray] (4.5,-2.5) -- (7.5,0.5);
\draw[gray] (5.5,-2.5) -- (8.5,0.5);
\draw[gray] (6.5,-2.5) -- (9.5,0.5);
\draw[gray] (7.5,-2.5) -- (10.5,0.5);
\draw[gray] (8.5,-2.5) -- (11.5,0.5);
\draw[gray] (9.5,-2.5) -- (12.5,0.5);
\draw[gray] (10.5,-2.5) -- (12.5,-0.5);
\draw[gray] (11.5,-2.5) -- (12.5,-1.5);
\draw[gray] (0,-1) -- (0.5,-1.1) -- (1,-1) -- (1.5,-1.1) -- (2,-1) -- (2.5,-1.1) -- (3,-1) -- (3.5,-1.1) -- (4,-1) -- (4.5,-1.1) -- (5,-1) -- (5.5,-1.1) -- (6,-1) --(6.5,-1.1) -- (7,-1) -- (7.5,-1.1) -- (8,-1) -- (8.5,-1.1) -- (9,-1) -- (9.5,-1.1) --(10,-1) -- (10.5,-1.1) -- (11,-1) -- (11.5,-1.1) -- (12,-1) --  (12.5,-1.1);
\node (R0) at (0.5,0.5) [gap] {$\scriptstyle \mathfrak{R}$};
\node (R1) at (1.5,0.5) [vertex] {};
\node (R2) at (2.5,0.5) [vertex] {};
\node (R3) at (3.5,0.5) [vertex] {};
\node (R4) at (4.5,0.5) [vertex] {};
\node (R5) at (5.5,0.5) [vertex] {};
\node (R6) at (6.5,0.5) [vertex] {};
\node (R7) at (7.5,0.5) [vertex] {};
\node (R8) at (8.5,0.5) [vertex] {};
\node (R9) at (9.5,0.5) [vertex] {};
\node (R10) at (10.5,0.5) [vertex] {};
\node (R11) at (11.5,0.5) [vertex] {};
\node (R12) at (12.5,0.5) [vertex] {};
\node (A0) at (0,0) [vertex] {};
\node (A1) at (1,0) [vertex] {};
\node (A2) at (2,0) [vertex] {};
\node (A3) at (3,0) [vertex] {};
\node (A4) at (4,0) [vertex] {};
\node (A5) at (5,0) [vertex] {};
\node (A6) at (6,0) [vertex] {};
\node (A7) at (7,0) [vertex] {};
\node (A8) at (8,0) [vertex] {};
\node (A9) at (9,0) [vertex] {};
\node (A10) at (10,0) [vertex] {};
\node (A11) at (11,0) [vertex] {};
\node (A12) at (12,0) [vertex] {};
\node (B0) at (0.5,-0.5) [vertex] {};
\node (B1) at (1.5,-0.5) [vertex] {};
\node (B2) at (2.5,-0.5) [vertex] {};
\node (B3) at (3.5,-0.5) [vertex] {};
\node (B4) at (4.5,-0.5) [vertex] {};
\node (B5) at (5.5,-0.5) [vertex] {};
\node (B6) at (6.5,-0.5) [vertex] {};
\node (B7) at (7.5,-0.5) [vertex] {};
\node (B8) at (8.5,-0.5) [vertex] {};
\node (B9) at (9.5,-0.5) [vertex] {};
\node (B10) at (10.5,-0.5) [vertex] {};
\node (B11) at (11.5,-0.5) [vertex] {};
\node (B12) at (12.5,-0.5) [vertex] {};
\node (C0) at (0,-1)  [vertex] {};
\node (C1) at (0.5,-1.1)  [vertex] {};
\node (C2) at (1,-1)  [vertex] {};
\node (C3) at (1.5,-1.1)  [vertex] {};
\node (C4) at (2,-1)  [vertex] {};
\node (C5) at (2.5,-1.1)  [vertex] {};
\node (C6) at (3,-1)  [vertex] {};
\node (C7) at (3.5,-1.1)  [vertex] {};
\node (C8) at (4,-1)  [vertex] {};
\node (C9) at (4.5,-1.1)  [vertex] {};
\node (C10) at (5,-1)  [vertex] {};
\node (C11) at (5.5,-1.1)  [vertex] {};
\node (C12) at (6,-1)  [vertex] {};
\node (C13) at (6.5,-1.1)  [vertex] {};
\node (C14) at (7,-1)  [vertex] {};
\node (C15) at (7.5,-1.1)  [vertex] {};
\node (C16) at (8,-1)  [vertex] {};
\node (C17) at (8.5,-1.1)  [vertex] {};
\node (C18) at (9,-1)  [vertex] {};
\node (C19) at (9.5,-1.1)  [vertex] {};
\node (C20) at (10,-1)  [vertex] {};
\node (C21) at (10.5,-1.1)  [vertex] {};
\node (C22) at (11,-1)  [vertex] {};
\node (C23) at (11.5,-1.1)  [vertex] {};
\node (C24) at (12,-1)  [vertex] {};
\node (C25) at (12.5,-1.1)  [vertex] {};
\node (D0) at (0.5,-1.5)  [vertex] {};
\node (D1) at (1.5,-1.5)  [vertex] {};
\node (D2) at (2.5,-1.5)  [vertex] {};
\node (D3) at (3.5,-1.5)  [vertex] {};
\node (D4) at (4.5,-1.5)  [vertex] {};
\node (D5) at (5.5,-1.5)  [vertex] {};
\node (D6) at (6.5,-1.5)  [vertex] {};
\node (D7) at (7.5,-1.5)  [vertex] {};
\node (D8) at (8.5,-1.5)  [vertex] {};
\node (D9) at (9.5,-1.5)  [vertex] {};
\node (D10) at (10.5,-1.5)  [vertex] {};
\node (D11) at (11.5,-1.5)  [vertex] {};
\node (D12) at (12.5,-1.5)  [vertex] {};
\node (E0) at (0,-2)  [vertex] {};
\node (E1) at (1,-2)  [vertex] {};
\node (E2) at (2,-2)  [vertex] {};
\node (E3) at (3,-2)  [vertex] {};
\node (E4) at (4,-2)  [vertex] {};
\node (E5) at (5,-2)  [vertex] {};
\node (E6) at (6,-2)  [vertex] {};
\node (E7) at (7,-2)  [vertex] {};
\node (E8) at (8,-2)  [vertex] {};
\node (E9) at (9,-2) [vertex] {};
\node (E10) at (10,-2)  [vertex] {};
\node (E11) at (11,-2)  [vertex] {};
\node (E12) at (12,-2)  [vertex] {};
\node (F0) at (0.5,-2.5)  [vertex] {};
\node (F1) at (1.5,-2.5)  [vertex] {};
\node (F2) at (2.5,-2.5)  [vertex] {};
\node (F3) at (3.5,-2.5)  [vertex] {};
\node (F4) at (4.5,-2.5)  [vertex] {};
\node (F5) at (5.5,-2.5)  [vertex] {};
\node (F6) at (6.5,-2.5)  [vertex] {};
\node (F7) at (7.5,-2.5)  [vertex] {};
\node (F8) at (8.5,-2.5)  [vertex] {};
\node (F9) at (9.5,-2.5)  [vertex] {};
\node (F10) at (10.5,-2.5)  [vertex] {};
\node (F11) at (11.5,-2.5)  [vertex] {};
\node (F12) at (12.5,-2.5)  [vertex] {};
\draw (0.5,0.5) circle (4.5pt);
\draw (4.5,0.5) circle (4.5pt);
\draw (6.5,0.5) circle (4.5pt);
\draw (8.5,0.5) circle (4.5pt);
\draw (12.5,0.5) circle (4.5pt);
\draw (3.5,-2.5) circle (4.5pt);
\draw (6.5,-2.5) circle (4.5pt);
\draw (9.5,-2.5) circle (4.5pt);
\node at (13.5,-1) {$\hdots$};
\end{tikzpicture}
\end{array}
\]
In particular, comparing this to \eqref{RingelPicture}, we observe the following coincidences.
\begin{enumerate}
\item\label{ident} The AR quiver of $\CM\mathfrak{R}$ is the quotient of the AR quiver of $\vect\bX$ by $\tau^{12(m-2)+1}=((12(m-2)+1)\ow)$.
\item The canonical tilting bundle $\mathcal{E}$ on $\bX$ is given by the circled vertices in \eqref{RingelPicture}, and so under the identification in \eqref{ident}, this gives the additive generator of $\SCM\mathfrak{R}$.
\end{enumerate}
The same coincidence can also be observed for type $\mathbb{T}$ and $\mathbb{I}$ by replacing $12$ by $6$ and $30$ respectively. 
To give a theoretical explanation to these observations, we need the following preparation.
\begin{lemma}\label{h s lemma}
Define $h$ as follows
\[
\begin{tabular}{*2c}
\toprule
Type&$h$\\
\midrule
$\mathbb{T}$&$6$\\
$\mathbb{O}$&$12$\\
$\mathbb{I}$&$30$\\
\bottomrule
\end{tabular}
\]
Then $(h+1)\ow=-\os$ and $(h(m-2)+1)\ow=-\os_{m-3}$. 
\end{lemma}
\begin{proof}
If $(p_1,p_2,p_3)=(2,3,3)$, then $6\ow=(6-3-2-2)\oc=-\oc$ and so $7\ow=-\os$.  Similarly, in the case $(p_1,p_2,p_3)=(2,3,4)$ then $12\ow=(12-6-4-3)\oc=-\oc$, thus $13\ow=-\os$.  Lastly, if $(p_1,p_2,p_3)=(2,3,5)$ then $30\ow=(30-15-10-6)\oc=-\oc$, hence $31\ow=-\os$.

Therefore $(h(m-2)+1)\ow=-(m-2)\os-(m-3)\ow=-\os-(m-3)\oc=-\os_{m-3}$.
\end{proof}

Let $\mathcal{C}$ be an additive category with an action by a cyclic group $G=\langle g\rangle\cong\bZ$. 
Assume that, for any $X,Y\in\mathcal{C}$, $\Hom_{\mathcal{C}}(X,g^iY)=0$ holds for $i\gg0$.
The \emph{complete orbit category} $\mathcal{C}/G$ has the same object as $\mathcal{C}$ and the morphism sets are given by
\[
\Hom_{\mathcal{C}/G}(X,Y)\colonequals \prod_{i\in\bZ}\Hom_{\mathcal{C}}(X,g^iY)
\]
for $X,Y\in\mathcal{C}$, where the composition is defined in the obvious way.

\begin{thm}\label{equiv last}
Let $R$ be the $(m-3)$-Wahl Veronese subring associated with $(p_1,p_2,p_3)=(2,3,3)$, $(2,3,4)$ or $(2,3,5)$ and $m\ge3$, and $\mathfrak{R}$ its completion. Let $G\leq \bL$ be the infinite cyclic group generated by the element $-\os_{m-3}=(h(m-2)+1)\ow$.  Then
\begin{enumerate}
\item\label{equiv last 1} There are equivalences $\vect\bX\simeq\CM^\bZ\!R$ and 
\[
F\colon (\vect\bX)/G\xrightarrow{\sim}\CM\mathfrak{R}.
\]
\item\label{equiv last 2} For the canonical tilting bundle $\mathcal{E}$ on $\bX$, we have $\SCM\mathfrak{R}=\add F\mathcal{E}$.
\end{enumerate}
\end{thm}
\begin{proof}
Since $(h(m-2)+1)\ow=-\os_{m-3}$ is a non-zero element in $-\bL_+$, for any $X,Y\in\vect\bX$, necessarily $\Hom_{\bX}(X,Y(i(h(m-2)+1)\ow))=0$ holds for $i\gg0$.
Therefore the complete orbit category $(\vect\bX)/G$ is well-defined.\\
(1) There are equivalences $\vect\bX\simeq\CM^{\bL}\!S\simeq\CM^\bZ\!R$, where the first equivalence is standard \cite{GL1}, and the second is \ref{WPL as qgrZ}.  Furthermore, the following diagram commutes.
\[
\begin{array}{c}
\begin{tikzpicture}
\node (top 1) at (0,0) {$\vect\bX$};
\node (top 2) at (2.5,0) {$\CM^\bZ\!R$};
\node (bottom 1) at (0,-1.5) {$\vect\bX$};
\node (bottom 2) at (2.5,-1.5) {$\CM^\bZ\!R$};
\draw[->] (top 1) -- node[left] {$\scriptstyle (\os_{m-3})$} (bottom 1);
\draw[->] (top 2) -- node[right] {$\scriptstyle (1)$} (bottom 2);
\draw[->] (top 1) -- (top 2);
\draw[->] (bottom 1) -- (bottom 2);
\end{tikzpicture}
\end{array}
\]
Since $\mathfrak{R}$ has only finitely many indecomposable CM modules (see e.g.\ \cite[15.14]{Y}), there is an equivalence $(\CM^\bZ\!R)/\bZ\simeq\CM\mathfrak{R}$.
Therefore  $(\vect\bX)/G\simeq(\CM^\bZ\!R)/\bZ\simeq\CM\mathfrak{R}$.\\
(2) This follows by the equivalences in \eqref{equiv last 1}, the definition of $\cE$, and \ref{specials determined thm}.
\end{proof}

As one final observation, recall that for a canonical algebra $\Lambda=\Lambda_{\bp,\bl}$, the \emph{preprojective algebra} of $\Lambda$ is defined by
\[
\Uppi\colonequals \bigoplus_{i\ge0}\Uppi_i,\ \ \ \Uppi_i\colonequals \Hom_{\Db(\mod \Lambda)}(\Lambda,\tau^{-i}\Lambda),
\]
where $\tau$ is the Auslander-Reiten translation in the derived category $\Db(\mod \Lambda)$.  Moreover, for a positive integer $t$, we denote the $t$-th Veronese subring of $\Uppi$ by
\[
\Uppi^{(t)}\colonequals \bigoplus_{i\ge0}\Uppi_{ti}.
\]
As notation we write $\Upgamma_m$ for the reconstruction algebra of $R$ above, which  corresponds to one of the types $\mathbb{T}$, $\mathbb{O}$ or $\mathbb{I}$ in \ref{Veronese=quotient}.

The following is an analogue of \ref{deg 0 general prop}, but also describes the other graded pieces. 
\begin{prop}\label{VeroneseGL2}
There is an isomorphism of $\bZ$-graded algebras
\[
\Uppi^{(h(m-2)+1)}\cong \Upgamma_m.
\]
\end{prop}
\begin{proof}
By \ref{h s lemma} we know that $(h(m-2)+1)\ow=-\os_{m-3}$.
Setting $M=\bigoplus_{\oy\in[0,\oc\,]}S(\oy)$, then $\Uppi^{(h(m-2)+1)}_i$ for $i\ge0$ is given by
\begin{align*}
\Hom_{\Db(\mod\Lambda)}(\Lambda,\tau^{-(h(m-2)+1)i}\Lambda)&\cong\Hom_S^{\bL}(M,M(-i(h(m-2)+1)\ow))\\
&\cong\Hom_S^{\bL}(M,M(i\os_{m-3}))\\
&\cong(\Upgamma_m)_i.
\end{align*}
Thus all the graded pieces match. It is easy to see that the isomorphisms are natural, and so give an isomorphism of graded rings. 
\end{proof}

\begin{remark}{\rm
By \ref{VeroneseGL2}, it follows that in fact on the abelian level 
\[
\qgr^{\bZ}\! \Upgamma_m\simeq \qgr^{\bZ}\! \Uppi^{(h(m-2)+1)}
\]
and so, combining \ref{WPL as qgrZ} and \ref{qgrR via qgrLambda}, 
\[
\coh \bX\simeq  \qgr^{\bZ}\! \Uppi^{(h(m-2)+1)}
\]
for any $m\geq 3$.  This is a stronger version of results of \cite{GL1} and Minamoto \cite{Minamoto}, which combine to say that for the weighted projective lines of non-tubular type there are derived equivalences
\[
\Db(\coh \bX)\simeq \Db(\mod \Lambda)\simeq \Db(\qgr^{\bZ}\! \Uppi).
\]}
\end{remark}

\bibliographymark{References}

\providecommand{\bysame}{\leavevmode\hbox to3em{\hrulefill}\thinspace}
\providecommand{\arXiv}[2][]{\href{https://arxiv.org/abs/#2}{arXiv:#1#2}}
\providecommand{\MR}{\relax\ifhmode\unskip\space\fi MR }
\providecommand{\MRhref}[2]{%
  \href{http://www.ams.org/mathscinet-getitem?mr=#1}{#2}
}
\providecommand{\href}[2]{#2}


\begin{thebibliography}{GL91}
\addcontentsline{toc}{section}{References}

\bibitem[AU]{TarigUeda}
T.~Abdelgadir and K.~Ueda, \emph{Weighted projective lines as fine moduli spaces of quiver representations}, Comm.\ Algebra {\bf 43} (2015), no.~2, 636--649.
\MR{3274027}

\bibitem[A]{Artin}
M.~Artin, \emph{On isolated rational singularities of surfaces},  Amer. J. Math.  {\bf 88} (1966), 129--136.
\MR{0199191}

\bibitem[AR]{AR_McKayGraphs}
M.~Auslander and I.~Reiten, \emph{{M}c{K}ay quivers and extended {D}ynkin diagrams}, Trans. Amer. Math. Soc. {\bf 293} (1986), no.~1, 293--301.
\MR{0814923}

\bibitem[AR2]{AR}
M.~Auslander and I.~Reiten, \emph{Almost split sequences for $\mathbb{Z}$-graded rings}, Singularities, representation of algebras, and vector bundles (Lambrecht, 1985)	Lecture Notes in Math., \textbf{1273}, Springer, Berlin, 1987, pp. 232--243. \MR{915178}

\bibitem[B1]{B83}
A.~A.~Beilinson, \emph{The derived category of coherent sheaves on $\mathbb{P}^n$}, Selected translations, Selecta Math. Soviet. {\bf 3} (1983/84), no.~3, 233--237.
\MR{0863137}

\bibitem[BLS]{BLS}
D.~Bergh, V.~Lunts and O.~Schn\"rer, \emph{Geometricity for derived categories of algebraic stacks}, Selecta Math.\ (N.S.) \textbf{22} (2016), no.~4, 2535--2568.
\MR{3573964}

\bibitem[B2]{Bridgelandt}
T.~Bridgeland, \emph{$t$-structures on some local Calabi-Yau varieties}, J. Algebra {\bf 289} (2005), no.~2, 453--483. 
\MR{2142382}

\bibitem[B3]{Brieskorn}
E.~Brieskorn, \emph{Rationale singularit{\"{a}}ten komplexer fl{\"{a}}chen},  Invent. Math. {\bf 4} (1968), 336--358.
\MR{0222084}

\bibitem[BS]{BS} 
M.~P.~Brodmann and R.~Y.~Sharp, \emph{Local cohomology: an algebraic introduction with geometric applications}, Cambridge Studies in Advanced Mathematics, vol. 60, Cambridge University Press, Cambridge, 1998.
\MR{1613627}

\bibitem[BHe]{BHe} 
W.~Bruns and J.~Herzog, \emph{Cohen-Macaulay rings}, Cambridge Studies in Advanced Mathematics, vol. 39, Cambridge University Press, Cambridge, 1993.
\MR{1251956}

\bibitem[B4]{B4}
N.~Bourbaki, \emph{Commutative algebra. Chapters 1--7}, Elements of Mathematics (Berlin), Springer-Verlag, Berlin, 1998. 
\MR{1727221}

\bibitem[BH]{BH}
R.~Buchweitz and L.~Hille, \emph{Higher Representation--Infinite Algebras from Geometry}.
In:
Representation theory of quivers and finite dimensional algebras
(abstracts from the workshop held February 16--22, 2014, organized by William Crawley-Boevey, Osamu Iyama, Bernhard Keller and Henning Krause), pp. 466--470,
Oberwolfach Rep. {\bf 11} (2014), no. 1, 453--529.
\MR{3379303}

\bibitem[BIRS]{BIRS}
A.~Buan, O.~Iyama, I.~Reiten, and D.~Smith, \emph{Mutation of cluster-tilting objects and potentials},  Amer.\ J.\ Math.\ {\bf 133} (2011), no.~4, 835--887.
\MR{2823864}

\bibitem[CCZ]{CCZ} J.~Chen, X.-W.~Chen, and Z.~Zhou, \emph{Monadicity theorem and weighted projective lines of tubular type}, Int. Math. Res. Not. IMRN 2015, no. 24, 13324--13359.
\MR{3436148}

\bibitem[EGA]{EGA}
J.~Dieudonn{\'e} and A.~Grothendieck, \emph{\'{E}l\'ements de g\'eom\'etrie alg\'ebrique}, 
Inst. Hautes \'Etudes Sci. Publ. Math. {\bf 4} (1960), {\bf 8}, {\bf 11} (1961), {\bf 17} (1963),
{\bf 20}, {\bf 24} (1964), {\bf 28} (1966), {\bf 32} (1967).


\bibitem[D]{Dolgachev}
I.~Dolgachev, \emph{On algebraic properties of algebras of automorphic forms}. In: Modular functions in analysis and number theory, pp. 20--29, Lecture Notes Math.\ Statist., vol. 5, Univ. Pittsburgh, Pittsburgh, PA, 1983. 
\MR{0732959}

\bibitem[DW]{DW2}
W.~Donovan and M.~Wemyss, \emph{Contractions and deformations}, Amer.\ J.\ Math.\ {\bf 141} (2019), no.~3, 563--592. 
\MR{3956515}

\bibitem[GL1]{GL1}
W.~Geigle and H.~Lenzing, \emph{A class of weighted projective curves arising in representation theory of finite-dimensional algebras}. In: Singularities, representation of algebras, and vector bundles (Lambrecht, 1985), pp. 265--297, Lecture Notes in Math., vol. 1273, Springer, Berlin, 1987. 
\MR{0915180}

\bibitem[GL2]{GL91}
W.~Geigle and H.~Lenzing, \emph{Perpendicular categories with applications to representations and sheaves}, J.\ Algebra {\bf 144} (1991), no.~2, 273--343.
\MR{1140607}

\bibitem[H1]{Hartshorne}
R.~Hartshorne, \emph{Algebraic geometry}, Graduate Texts in Mathematics, vol.~52.\ Springer-Verlag, New York-Heidelberg, 1977.
\MR{0463157}

\bibitem[H2]{H}
M.~Hashimoto, \emph{Equivariant class group. III. Almost principal fiber bundles}. Preprint 2015. \arXiv[ ]{1503.02133}

\bibitem[HIMO]{HIMO}
M.~Herschend, O.~Iyama, H.~Minamoto, and S.~Oppermann, \emph{Representation theory of Geigle-Lenzing complete intersections}, to appear Mem.\ Amer.\ Math.\ Soc.
 
\bibitem[H3]{Herzog} 
J.~Herzog, \emph{Ringe mit nur endlich vielen Isomorphieklassen von maximalen, unzerlegbaren Cohen-Macaulay-Moduln}, Math. Ann. {\bf 233} (1978), no. 1, 21--34. 
\MR{0463155}

\bibitem[I]{Ito}
Y.~Ito, \emph{Special McKay correspondence}. In: Geometry of toric varieties, pp. 213--225, S\'emin. Congr., vol. 6, Soc. Math. France, Paris, 2002.
\MR{2075612}

\bibitem[IL]{IL}
O.~Iyama and B.~Lerner, \emph{Tilting bundles on orders on $\bP^d$}, Israel J. Math. {\bf 211} (2016), no. 1, 147--169.
\MR{3474959}

\bibitem[IT]{IT}
O.~Iyama and R.~Takahashi, \emph{Tilting and cluster tilting for
quotient singularities},  Math.\ Ann.\ \textbf{356} (2013), no.~3, 1065--1105. \MR{3063907}

\bibitem[IW]{IW}
O. Iyama and M. Wemyss, \emph{The classification of special Cohen Macaulay modules}, Math. Z. {\bf 265} (2010), no.~1, 41--83.
\MR{2606949}

\bibitem[IW2]{IW4}
O.~Iyama and M.~Wemyss, \emph{Maximal modifications and Auslander--Reiten duality for non-isolated singularities}, Invent.\ Math. {\bf 197} (2014), no.~3, 521--586.
\MR{3251829}

\bibitem[KM]{KM}
J.~Koll\'ar and S. Mori,  \emph{Birational geometry of algebraic varieties}, Cambridge Tracts in Mathematics, vol. 134, Cambridge Univ. Press, Cambridge, 1998. 
\MR{1658959}

\bibitem[L]{Laufer}
H.~Laufer, \emph{Taut two-dimensional singularities}, 
Math. Ann. {\bf 205} (1973), 131--164. 
\MR{0333238}

\bibitem[M]{Minamoto}
H.~Minamoto, \emph{Ampleness of two-sided tilting complexes}, Int.\ Math.\ Res.\ Not.\ IMRN 2012, no.~1, 67--101.
\MR{2874928}

\bibitem[OW]{OW}
P.~Orlik and P.~Wagreich, \emph{Isolated singularities of algebraic surfaces with $\mathbb{C}^*$ action}, Ann.\ of Math.\ (2) {\bf 93} (1971), 205--228.
\MR{0284435}

\bibitem[P]{Pinkham}
H.~Pinkham, \emph{Normal surface singularities with $\mathbb{C}^*$ action}, Math.\ Ann.\ {\bf 227} (1977), no.~2, 183--193. 
\MR{0432636}

\bibitem[R1]{Riemen}
O.~Riemenschneider, \emph{Die Invarianten der endlicher Untergruppen von $GL(2,\mathbb{C})$}, Math. Z.\ {\bf 153} (1977), no. 1, 37--50.
\MR{0447230}

\bibitem[R2]{Ringel}
C.~M.~Ringel, \emph{Tame algebras and integral quadratic forms}, Lecture Notes in Mathematics, vol. 1099, Springer-Verlag, Berlin, 1984.
\MR{0774589}

\bibitem[R3]{Rydh}
D.~Rydh, \emph{If the direct image of $f$ preserves coherent sheaves on noetherian schemes, how to show $f$ is proper?} URL (version: 2014-10-14): \href{https://mathoverflow.net/q/182902}{https://mathoverflow.net/q/182902}.

\bibitem[VdB]{VdB1d}
M.~van den Bergh, \emph{Three-dimensional flops and noncommutative rings}, 
Duke Math. J. {\bf 122} (2004), no.~3, 423--455.
\MR{2057015}

\bibitem[VdB2]{VdB2}
M.~van den Bergh, \emph{Non-commutative crepant resolutions}. In: The legacy of Niels Henrik Abel, pp. 749--770, Springer, Berlin, 2004.
\MR{2077594}

\bibitem[W1]{Wahl}
J.~M.~Wahl, \emph{Equations defining rational singularities}, Ann.\ Sci.\ \'Ecole Norm.\ Sup.\ (4) {\bf 10} (1977), no.~2, 231--263. 
\MR{0444655}

\bibitem[W2]{WemGL2}
M.~Wemyss, \emph{The $\GL(2,\mathbb{C})$ McKay correspondence}, Math.\  Ann.\, {\bf 350} (2011), no.~3,  631--659.
\MR{2805639}

\bibitem[W3]{WemA}
M.~Wemyss, \emph{Reconstruction algebras of type $A$}, Trans.\ Amer.\ Math.\ Soc.\ {\bf 363} (2011), no.~6, 3101--3132.
\MR{2775800}

\bibitem[W4]{D1}
M.~Wemyss, \emph{Reconstruction algebras of type $D$ (I)}, J. Algebra {\bf 356} (2012), 158--194.
\MR{2891127}

\bibitem[W5]{WunramCyclic}
J.~Wunram, \emph{Reflexive modules on cyclic quotient surface
singularities}. In:
Singularities, representation of algebras, and vector bundles (Lambrecht, 1985),
pp. 221--231,
Lecture Notes in Mathematics, vol. 1273, Springer, Berlin, 1987.
\MR{0915177}

\bibitem[W6]{Wunram}
J.~Wunram, \emph{Reflexive modules on quotient surface singularities},
 Math.\ Ann.\ {\bf 279} (1988), no.~4, 583--598.
\MR{0926422}

\bibitem[Y]{Y}
Y. Yoshino, \emph{Cohen-Macaulay modules over Cohen-Macaulay rings},
London Mathematical Society Lecture Note Series, vol. 146, Cambridge University Press, Cambridge, 1990.
\MR{1079937}

\bibitem[ZS]{ZS}
O. Zariski and P. Samuel, \emph{Commutative algebra. Vol. II}. Graduate Texts in Mathematics, vol. 29, Springer-Verlag, New York-Heidelberg, 1975.
\MR{0389876}

\end{thebibliography}
\end{document}